\documentclass[envcountsect,smallextended]{svjour3}       
\smartqed
\usepackage{graphicx}
\usepackage[utf8]{inputenc}
\usepackage[T1]{fontenc}
\usepackage{lmodern}
\usepackage{amsmath}
\usepackage{amssymb}
\usepackage{marginnote}
\usepackage{xcolor}
\usepackage{fancyhdr}
\usepackage{esint}
\usepackage{tikz} 
\usepackage[arrow, matrix, curve]{xy}
\usepackage{hyperref}
\usetikzlibrary{fit,positioning} 
\usepackage{geometry}
\usepackage{pifont}
\usepackage{nccmath}
\usepackage{tocloft}
\usepackage{setspace}
\usepackage{graphicx}
\usepackage{calligra}
\usepackage[english]{babel}
\usepackage{amssymb}
\usepackage{upgreek}
\usepackage[justification=centering,labelfont=bf,belowskip=-10pt,aboveskip=3pt,font=small]{caption}
\usepackage{booktabs}
\usepackage{mathrsfs}  

\usepackage{blindtext}

\newcommand{\vertiii}[1]{{\vert\kern-0.25ex\vert\kern-0.25ex\vert #1 
    \vert\kern-0.25ex\vert\kern-0.25ex\vert}}

\providecommand{\fdg}{{\,\big|\,}}

\usepackage{color}
\definecolor{Darkgreen}{rgb}{0,0.4,0}
\usepackage{hyperref}
\hypersetup{%
  pdfborder = {0 0 0},
  colorlinks,
  citecolor=blue,
  filecolor=black,
  linkcolor=blue,
  urlcolor=cyan!50!black!90
}

\newcommand{\bS}{\textrm {S}}
\newcommand{\bD}{\mathcal {D}}
\newcommand{\bR}{{R}}
\newcommand{\bE}{{E}}
\newcommand{\bN}{\text {N}}
\newcommand{\levy}{\boldsymbol \varepsilon}
\newcommand{\sym}{\textrm{sym}}
\newcommand{\anti}{\textrm{skew}}
\newcommand{\setR}{\mathbb{R}}
\DeclareMathOperator{\curl}{curl}
\DeclareMathOperator{\dist}{dist}

  \providecommand{\abstmp}[2]{{#1\lvert{#2}#1\rvert}}
  \providecommand{\abs}[1]{\abstmp{}{#1}}
  
  \providecommand{\Bigabs}[1]{\abstmp{\Big}{#1}}

    \providecommand{\Xint}[1]{\mathchoice
    {\XXint\displaystyle\textstyle{#1}}%
    {\XXint\textstyle\scriptstyle{#1}}%
    {\XXint\scriptstyle\scriptscriptstyle{#1}}%
    {\XXint\scriptscriptstyle\scriptscriptstyle{#1}}%
    \!\int}
  \providecommand{\XXint}[3]{{\setbox0=\hbox{$#1{#2#3}{\int}$}
      \vcenter{\hbox{$#2#3$}}\kern-.5\wd0}}
  
  \providecommand{\dashint}{\mathop{\Xint-}}

  \DeclareMathSymbol{\sm}{\mathbin}{AMSa}{"39}

\DeclareMathOperator*{\divo}{\operatorname{div}}
\DeclareMathAlphabet\mathbfcal{OMS}{cmsy}{b}{n}
\DeclareMathAlphabet\mathbfscr{OMS}{mdugm}{b}{n}

\spnewtheorem{defn}[equation]{Definition}{\bfseries}{\upshape}
\spnewtheorem{prop}[equation]{Proposition}{\bfseries}{\upshape}
\spnewtheorem{thm}[equation]{Theorem}{\bfseries}{\upshape}
\spnewtheorem{cor}[equation]{Corollary}{\bfseries}{\upshape}
\spnewtheorem{rmk}[equation]{Remark}{\bfseries}{\upshape}
\spnewtheorem{lem}[equation]{Lemma}{\bfseries}{\upshape}
\spnewtheorem{expl}[equation]{Example}{\bfseries}{\upshape}
\spnewtheorem{asum}[equation]{Assumption}{\bfseries}{\upshape}
\spnewtheorem{alg}[equation]{Algorithm}{\bfseries}{\upshape}
\spnewtheorem{ex}[equation]{Experiment}{\bfseries}{\upshape}

\numberwithin{equation}{section} 

\begin{document}
	
\title{Analysis of fully discrete, quasi non-conforming approximations of evolution
  equations and applications}
	
\author{Luigi C. Berselli \and Alex Kaltenbach \and {Michael R\r{u}\v{z}i\v{c}ka}}

\institute{L.C. Berselli \at Department of Mathematics, University of Pisa
  \\
  Via F. Buonarroti 1/c, I-56127 Pisa  (ITALY),
  \\
  \email{luigi.carlo.berselli@unipi.it}
  \\
  A. Kaltenbach \at Institute of Applied Mathematics,
  Albert-Ludwigs-University Freiburg,
  Ernst-Zermelo-Straße 1, D-79104 Freiburg (GERMANY),
  \\
  \email{alex.kaltenbach@mathematik.uni-freiburg.de}
  \\
  M. R\r{u}\v{z}i\v{c}ka \at Institute of Applied Mathematics,
  Albert-Ludwigs-University Freiburg,
  Ernst-Zermelo-Straße 1, D-79104 Freiburg (GERMANY),
  \\
  \email{rose@mathematik.uni-freiburg.de} 
}

\date{Received: date / Accepted: date}

\maketitle

\begin{abstract}
  In this paper we consider fully discrete approximations of abstract 
evolution
  equations, by means of a quasi non-conforming spatial approximation and finite
  differences in time (Rothe--Galerkin method). The main result is the convergence of the
  discrete solutions to a weak solution of the continuous problem. Hence, the result can
  be interpreted either as a justification of the numerical method, or as an alternative
  way of constructing weak solutions.
  
  We set the problem in the very general and abstract setting of
  pseudo-monotone operators, which allows for a unified treatment of
  several evolution problems. The examples --which fit into our
  setting and which motivated our research-- are problems describing
  the motion of incompressible fluids, since the quasi non-conforming
  approximation allows to handle problems with prescribed divergence.

  Our abstract results for pseudo-monotone operators allow to show
  convergence just by 
  verifying a few natural assumptions on the operator 
  time-by-time and on the discretization spaces.  Hence, applications and
  extensions to several other evolution problems can be easily performed. 
The results of
  some numerical experiments are reported in the final section.
  
\end{abstract}

\section{Introduction}
\label{intro}
We consider the numerical approximation of an abstract evolution
equation\footnote{ If not specified differently, we will denote in
  boldface elements of Bochner spaces (as the solution
  $\boldsymbol{u}(t)$ and the external source $\boldsymbol{f}(t)$), to
  highlight the difference with elements belonging to standard Banach
  spaces, denoted by usual symbols.}
\begin{align}
  \begin{aligned}
    \frac{d\boldsymbol{u}}{dt}(t)+A(t)(\boldsymbol{u}(t))&= \boldsymbol{f}(t)&&\quad\text{
      in } V^*,
    \\
    \boldsymbol{u}(0)&={u}_0&&\quad\text{ in }H,
  \end{aligned} \label{eq:1.1}
\end{align}
by means of a quasi non-conforming Rothe--Galerkin scheme. Here, ${V{\hookrightarrow} H
  \cong H^*{\hookrightarrow}V^*}$ is a given evolution triple, $I:=\left(0,T\right)$ a
finite time horizon, ${u}_0\in H$ an initial value, ${\boldsymbol{f}\in L^{p'}(I,V^*)}$,
$p \in (1,\infty)$, a right-hand side and $A(t):V\rightarrow V^*$, $t\in I$, a family of
operators.

In order to make~\eqref{eq:1.1} accessible to non-conforming approximation methods, we
will additionally require that there exists a further evolution triple $X{\hookrightarrow}
Y \cong Y^*{\hookrightarrow}X^*$, such that $V\subseteq X$ with $\|\cdot\|_V=\|\cdot\|_X$
on $V$ and $H\subseteq Y$ with $(\cdot,\cdot)_H=(\cdot,\cdot)_Y$ in $H$, and extensions
${\hat{A}(t):X\rightarrow X^*}$,~$t\in I$, and $\hat{\!\boldsymbol{f}}\in L^{p'}(I,X^*)$
of $\{A(t)\}_{t\in I}$ and $\boldsymbol{f}$, resp., i.e., $\langle
\hat{A}(t)v,w\rangle_X=\langle A(t)v,w\rangle_V$ and $\langle\,
\hat{\!\boldsymbol{f}}(t),v\rangle_X=\langle \boldsymbol{f}(t),v\rangle_V$ for all $v,w\in
V$ and almost every $t\in I$. For sake of readability we set $A(t):=\hat{A}(t)$ and
$\boldsymbol{f}(t):=\hat{\!\boldsymbol{f}}(t)$ for almost every $t\in I$.

\medskip

Recently, the existence theory for abstract evolution problems with
Bochner pseudo-monotone operators in \cite{alex-rose-hirano}, based on
the convergence of a Galerkin approximation, was extended in
\cite{BR19} to the convergence proof of a fully discrete
Rothe--Galerkin approximation. However, the result in \cite{BR19} is
not applicable to the treatment of problems describing the flow of
incompressible fluids. The main aim of this paper is develop an
abstract framework which shows that a fully discrete Rothe--Galerkin
approximation converges also for such problems (cf.~Theorem
\ref{5.17}). Even though the framework is rather abstract it is easily
applicable to many problems, since we show that it is enough to check
a few easily verifiable conditions (cf.~conditions
(\hyperlink{A.1}{A.1})--(\hyperlink{A.4}{A.4}) in
Proposition~\ref{3.9} and conditions
(\hyperlink{QNC.1}{QNC.1})--(\hyperlink{QNC.2}{QNC.2}) in
Definition~\ref{3.1}).

A prototypical example for the flow of incompressible non-Newtonian
fluids are the following equations describing the unsteady
motion of incompressible shear-dependent fluids 
\begin{align}
  \begin{split}
    \begin{alignedat}{2}
      \partial_t \boldsymbol{u}-\divo \big (
      (\kappa
      +|\mathcal{D}\boldsymbol{u}|)^{p-2}\mathcal{D}\boldsymbol{u}\big )+\text{div}(\boldsymbol{u}\otimes
      \boldsymbol{u}) +\nabla\boldsymbol{q}&=\boldsymbol{f}&&\quad\text{ in }I\times\Omega,
      \\
      \divo \boldsymbol{u}&=0&&\quad\text{ in }I\times\Omega,
      \\
      \boldsymbol{u}&=\mathbf{0}&&\quad\text{ on }I\times \partial\Omega,
      \\
      \boldsymbol{u}(0)&={u}_0&&\quad\text{ in }\Omega.
    \end{alignedat}
  \end{split}\label{eq:p-NS}
\end{align}
Here, $\Omega \subseteq \mathbb{R}^d$, $d\ge 2$, is a domain, 
$I=(0,T)$ a  time interval, and $\kappa\ge 0 $ and $p \in (1,\infty)$ are material parameters of
the shear-dependent fluid. Moreover,
$\boldsymbol{u}:I\times\Omega\to \mathbb{R}^d$ denotes the velocity,
$\boldsymbol{f}:I\times\Omega\to \mathbb{R}^d$ is a given external
force, ${u}_0:\Omega\to \mathbb{R}^d$ is an initial condition,
$\boldsymbol{q}:I\times\Omega\to \mathbb{R}$ is the pressure and
$\mathcal{D}\boldsymbol{u}:=\frac{1}{2}(\nabla \boldsymbol{u}+\nabla
\boldsymbol{u}^\top)$ denotes the symmetric
gradient.

We define for $p>\frac{3d+2}{d+2}$ the function spaces $X:=W^{1,p}_0(\Omega)^d$,
$Y:=L^2(\Omega)^d$, $V:=W^{1,p}_{0,\divo }(\Omega)$ as the closure of
$\mathcal{V}:=\{{v}\in C_0^\infty(\Omega)^d\fdg \divo \,{v}\equiv 0\}$ in
$X$, $H:=L^2_{\divo }(\Omega)$ as the closure of $\mathcal{V}$ in $Y$, and the operators ${S,B:X\to X^*}$ for all ${u},{v}\in X$ via
\begin{gather}\label{eq:sb}
  \langle S{u},{v}\rangle_X:= \int_\Omega (\kappa
  +|{\bD}{u}|)^{p-2}{\bD u} :{\bD v}\,dx
  \quad \text{ and } \quad \langle
  B{u},{v}\rangle_X:=-\int_\Omega{{u}\otimes {u}:{\nabla v}\,dx}.
\end{gather}
Then, \eqref{eq:p-NS} for ${u}_0\in H$ and
$\boldsymbol f \in L^{p'}(I,X^*)$ can be re-written as the abstract 
evolution equation
\begin{align}
  \begin{split}
    \begin{alignedat}{2}
      \frac{d\boldsymbol{u}}{dt}(t)+S(\boldsymbol{u}(t))+B(\boldsymbol{u}(t))&=\boldsymbol{f}(t)&&\quad\text{
        in }V^*,
      \\
      \boldsymbol{u}(0)&={u}_0&&\quad\text{ in }H.
      \label{eq:p-NS2}
    \end{alignedat}
  \end{split}
\end{align}

Typical fully discrete  approximations of \eqref{eq:p-NS2} are often based on 
appropriate finite element spaces $(V_n)_{n\in
  \mathbb{N}}$. As the construction of finite element spaces which meet the
divergence constraint exactly, i.e., satisfy $V_n\subseteq V$ for all $n\in \mathbb{N}$, highly
restricts the flexibility of the approximation, one usually
works with finite element spaces satisfying a discrete divergence
constraint only. Thus, we have $V_{n}\not\subseteq V$, which is the
reason that the theory in \cite{BR19} is not applicable. However, the
spaces $V_n$ often satisfy $V_n\subseteq X$, which allows us to
develop a convergence theory for such a setting. Note that this
problem is treated from a different point of view, namely the theory of maximal
monotone graphs, in
the recent contributions \cite{tscherpel-phd} and
\cite{sueli-tscherpel}. 

\subsection{The numerical scheme}
A quasi non-conforming Rothe--Galerkin approximation of the initial value
problem~\eqref{eq:1.1} usually consists of two parts:
	
The first part is a spatial discretization, often called Galerkin
approximation, which consists in the approximation of $V$ by a
sequence of closed subspaces $(V_n)_{n\in \mathbb{N}} $ of $X$. We
emphasize that we do not require $(V_n)_{n\in \mathbb{N}}$ to be a
sequence of subspaces of $V$, which motivates the prefix
\textit{non-conforming}. Hence, we do \textit{not} have
$V_{n}\subseteq V$ and $V_{n}\nearrow V$ (approximation from below).
The prefix \textit{quasi}\footnote{Observe that the subspaces
  $(V_n)_{n\in \mathbb{N}}$ 
  are spatially conforming in $X$.} indicates that the subspaces $V_n$ are for
all $n \in \mathbb N$ equipped with the norm of the space $X$. 
In this case we have, under appropriate assumptions, that
$V_{n}\searrow V$, i.e., we have an approximation of the space $V$
from above. The assumption that the subspaces $V_n$ are equipped with
the norm of the space $X$ (and not with a norm depending on $n$),
together with the assumption that $\|\cdot \|_X=\|\cdot\|_V$ on $V$
reflects the fact that the spaces $V_n$ and $X$ have the same \glqq
regularity\grqq. This excludes spaces $V_n$ resulting from spatial
discontinuous Galerkin approximations (cf.~\cite{DiPE12}). For a
spatial discontinuous Galerkin approximations in the above example one
would choose $X=L^2(\Omega)^d$ and $V_n\subseteq X$ as a
subspace\footnote{$\mathcal{P}_k(\mathcal{T}_h)$, $k\in \mathbb{N}_0$,
  denotes the space of possibly discontinuous scalar functions, which
  are polynomials of degree at most $m$ on each simplex
  $K\in \mathcal{T}_h$.} of $\mathcal{P}_k(\mathcal T_{h_n})^d$
satisfying a discrete divergence constraint and being equipped with the norm of the
\textit{broken} Sobolev space
$W^{1,p}_{0,\divo,\text{DG}}(\mathcal T_{h_n})^d$, where
$\mathcal T_{h_n}$ is an appropriate triangulation of $\Omega$. These
choices violate both our assumptions and are thus not included in our
treatment.

The second part is a temporal discretization, also called Rothe
scheme, which consists in the approximation of the unsteady problem
\eqref{eq:1.1} by a sequence of piece-wise constant, steady
problems. This is achieved by replacing the time derivative
$\frac{d}{dt}$ by so-called backwards difference quotients. These are
for a given step-size $\tau:=\frac{T}{K}>0$, where $K\in \mathbb{N}$,
and a given finite sequence $(u^k)_{k=0,\dots,K}\subseteq X$ defined via
\begin{align*}
  d_\tau u^k:=\frac{1}{\tau}(u^k-u^{k-1})\quad\text{ in }X\quad\text{ for all }k=1,\dots, K.
\end{align*}
Moreover, the operator family $A(t):X\to X^*$, $t\in I$, and the
right-hand side $\boldsymbol{f}\in L^{p'}(I,X^*)$ are discretized by
means of the Cleme\'nt $0$-order quasi interpolant. This means that
for a given step-size $\tau=\frac{T}{K}>0$, where $K\in \mathbb{N}$,
we replace them piece-wise by their local temporal means, i.e., by
$[A]^\tau_k:X\to X^*$, $k=1,\dots,K$,
and $([\boldsymbol{f}]^\tau_k)_{k=1,\dots,K}\subseteq X^*$, resp., for every
$k=1,\dots,K$ and $u\in X$ given via
\begin{align*}
  [A]^\tau_ku:=\fint_{\tau(k-1)}^{\tau
  k}{A(t)u\,dt} \qquad \textrm{and}\qquad[\boldsymbol{f}]^\tau_k:=\fint_{\tau(k-1)}^{\tau
  k}{\boldsymbol{f}(t)\,dt}\qquad\text{ in }X^*. 
\end{align*}
Altogether, using these two levels of approximation, we formulate the
following fully discrete or Rothe--Galerkin scheme of the evolution
problem \eqref{eq:1.1}: 
\begin{alg}[quasi non-conforming Rothe--Galerkin scheme]
  For given $K,n\in \mathbb{N}$ and $u_n^0\in V_n$ the sequence of
  iterates $(u_n^{k})_{k=0,\dots,K}\subseteq V_n$ is given solving the
  implicit scheme for $\tau=\frac{T}{K}$ and $k=1,\dots,K$ 
  \begin{align}
    (d_\tau u^k_n,v_n)_Y+\langle [A]^\tau_ku^k_n,v_n\rangle_X
    = \langle [\boldsymbol{f}]^\tau_k,v_n\rangle_X
    \quad\text{ for all }v_n\in V_n. \label{eq:fully}
  \end{align}
\end{alg}


Traditionally, the verification of the convergence of a
Rothe--Galerkin scheme like \eqref{eq:fully} to a weak solution of the
evolution equation \eqref{eq:1.1} causes a certain effort. In the case
that quasi non-conforming approximations are used in \eqref{eq:fully},
to the best of the authors' knowledge, there are no abstract results
guaranteeing the weak convergence of such a scheme. Therefore, the
purposes of this article are (i) to give general and easily verifiable
assumptions on both the operator family $A(t):X\to X^*$,~$t\in I$, and
the sequence of approximative spaces $(V_n)_{n\in \mathbb{N}}$ which
provide both the existence of iterates
$(u_n^k)_{k=0,\dots,K}\subseteq V_n$, solving \eqref{eq:fully}, for a
sufficiently small step-size
$\tau=\frac{T}{K}\in \left(0,\tau_0\right)$, where $\tau_0>0$, and
$K,n\in \mathbb{N}$; and (ii) to prove the stability of the scheme, i.e.,
the boundedness of the piece-wise constant interpolants
${\overline{\boldsymbol{u}}^\tau_n\in L^\infty(I,V_n)}$,
$K,n\in \mathbb{N}$ with
$\tau=\frac{T}{K}$, 
in $L^p(I,X)\cap L^\infty(I,Y)$; and finally (iii) to show the weak convergence of
a diagonal subsequence
$(\overline{\boldsymbol{u}}^{\tau_n}_{m_n})_{n\in \mathbb{N}}\subseteq
L^\infty(I,X)$, where $\tau_n=\frac{T}{K_n} $ and
$K_n,m_n\to \infty$~${(n\to\infty)}$, towards a weak solution of
problem \eqref{eq:1.1}. All these results are formulated exactly and proved
in Section \ref{sec:6} (cf.~Proposition~\ref{5.1}, Proposition
\ref{apriori} and  the main Theorem
\ref{5.17}).\\[-3mm]

Surprisingly, 
there are only few contributions with a rigorous convergence analysis
of fully discrete Rothe--Galerkin schemes towards weak solutions. Most
authors consider only semi-discrete schemes, i.e., either a pure Rothe
scheme (cf.~\cite{Rou05}) or a pure Galerkin scheme (cf.~\cite{GGZ},
\cite{Zei90B}, \cite{showalter}, \cite{alex-rose-hirano}). Much more results are
concerned with explicit convergence rates for more regular data and
more regular solutions (cf.~\cite{BarLiu94}, \cite{rulla},
\cite{NoSaVe00}, \cite{FeOePr05}, \cite{DiEbRu07}, \cite{P08},
\cite{CHP10}, \cite{BaNoSa14}, \cite{sarah-phd}, \cite{BDN},
\cite{breit-mensah}). Even in the conforming case, that is if the sequence $(V_n)_{n\in \mathbb{N}}$
satisfies the following two natural conditions:\footnote{If
  (\hyperlink{C.1}{C.1}) and (\hyperlink{C.2}{C.2}) are satisfied one
  can choose $X=V$ in the above setting.  }
\begin{description}
\item[\textbf{(C.1)}] \hypertarget{C.1}{} $(V_n)_{n \in \mathbb N}$ is
  an increasing sequence of closed subspaces of $V$, i.e.,
  $V_n\!\subseteq\! V_{n+1}\!\subseteq V$~for~all~${n\!\in\! \mathbb{N}}$.
\item[\textbf{(C.2)}] \hypertarget{C.2}{} $\bigcup_{n\in \mathbb{N}}{V_n}$ is dense in $V$.
\end{description}
there are very few results concerning the convergence analysis of a
fully discrete Rothe--Galerkin scheme.  We are only aware of the early
contribution \cite{AL83} which treats the porous media equation and
\cite{BR19} dealing with
a setting similar to the one proposed in the present paper. 

Let us shortly explain the strategy used in \cite{alex-rose-hirano} and
\cite{BR19}, since it will be extended in the present paper to handle
also a quasi non-conforming setting. Using the properties of the
operator family ${A(t):V\to V^*}$,~${t\in I}$ (cf.~\cite[condition
(C.1)--(C.4)]{alex-rose-hirano}) and the properties of the Rothe--Galerkin 
scheme
\eqref{eq:fully} one can show that the iterates
$(u_n^k)_{k=0,\dots,K}\subseteq V_n$, $K,n\in \mathbb{N}$, solving
\eqref{eq:fully}, generate for sufficiently small $\tau=\frac{T}{K}$ a
family of piece-wise constant interpolants
$\overline{\boldsymbol{u}}_n^\tau$ for which the following holds:

There exists a sequence
$(\overline{\boldsymbol{u}}_n)_{n\in
  \mathbb{N}}:=(\overline{\boldsymbol{u}}_{m_n}^{\tau_n})_{n\in
  \mathbb{N}}$ and an element
$\overline{\boldsymbol{u}}\in L^p(I,V)\cap L^\infty(I,H)$, such that
\begin{alignat*}{2}
  \overline{\boldsymbol{u}}_n&\;\;\rightharpoonup\;\;\overline{\boldsymbol{u}}&&\quad\text{
    in }L^p(I,V)\qquad \;\,(n\to \infty),
  \\
  \overline{\boldsymbol{u}}_n&\;\;\overset{\ast}{\rightharpoondown}\;\;\overline{\boldsymbol{u}}&&\quad\text{
    in }L^\infty(I,H)\qquad (n\to \infty),
  \\
  \overline{\boldsymbol{u}}_n(t)&\;\;\rightharpoonup\;\;\overline{\boldsymbol{u}}(t)&&\quad\text{
    in }H\quad\text{ for a.e. }t\in I\quad(n\to\infty),
  \\
  \limsup_{n\to \infty}&\;\langle
  \mathbfcal{A}\overline{\boldsymbol{u}}_n,\overline{\boldsymbol{u}}_n&&-\overline{\boldsymbol{u}}\rangle_{L^p(I,V)}\leq
  0,
\end{alignat*}
where $\mathbfcal{A}: L^p(I,V)\cap L^\infty(I,H)\to (L^p(I,V))^*$
denotes the induced operator, which is for every
$\boldsymbol{u}\in L^p(I,V)\cap L^\infty(I,H)$ and
$\boldsymbol{v}\in L^p(I,V)$ given via
$\langle
\mathbfcal{A}\boldsymbol{u},\boldsymbol{v}\rangle_{L^p(I,V)}:=\int_I{\langle
  A(t)(\boldsymbol{u}(t)),\boldsymbol{v}(t)\rangle_V\,dt}$.  Using
these properties and the fact that the induced operator $\mathbfcal A$
is \textit{Bochner pseudo-monotone} one can conclude that
$\mathbfcal{A}\overline{\boldsymbol{u}}_n\rightharpoonup\mathbfcal{A}\overline{\boldsymbol{u}}
$ in $(L^p(I,V))^*$ $(n \to \infty)$, and therefore  the weak
convergence of the scheme \eqref{eq:fully}. In this argumentation one
has used on several places the fact that the sequence
$(V_n)_{n\in \mathbb{N}}$ satisfies the conditions
(\hyperlink{C.1}{C.1}) and (\hyperlink{C.2}{C.2}).
	
Without the conditions (\hyperlink{C.1}{C.1}) and
(\hyperlink{C.2}{C.2}), i.e., $V\neq X$ and $V_n\not\subseteq V$ for
all $n\in \mathbb{N}$ one could hope to prove the above properties
with $V$ and $H$ replaced by $X$ and $Y$, respectively. Even if this
works it is not clear whether the weak limit lies in the right
function space, i.e., whether
$\overline{\boldsymbol{u}}\in L^p(I,V)\cap L^\infty(I,H)$. To
guarantee that this procedure works in an appropriate sense, we will assume that
$(V_n)_{n\in\mathbb{N}}$ satisfies the assumptions
(\hyperlink{QNC.1}{QNC.1}) and (\hyperlink{QNC.2}{QNC.2}) from
Definition \ref{3.1}.
Moreover, we have to adapt the notion of Bochner pseudo-monotone
operators to the quasi non-conforming setting.

\subsection{The example of the $p$-Navier-Stokes equations}
Let us indicate that the prototypical example \eqref{eq:p-NS} fits
into the abstract setting of the previous section. Full details will
be given in Section \ref{sec:7} where two different problems from the field of
incompressible fluid flows are treated. In fact, one of these examples
contains problem  \eqref{eq:p-NS} as a special case.

Let $Z:=L^{p'}(\Omega)$. For a given family of shape regular triangulations
(cf.~\cite{BS08}) $(\mathcal{T}_h)_{h>0}$ of a polygonal Lipschitz domain 
$\Omega$ and for 
given $m,\ell \in \mathbb N_0$, we denote by
${X_h\subset\mathcal{P}_m(\mathcal{T}_h)^d\cap X}$, equipped with the $X$-norm, and $Z_h\subset
\mathcal{P}_\ell(\mathcal{T}_h)\cap Z$, equipped with the $Z$-norm, appropriate finite element
spaces. 
In addition, we define for $h>0$ the \textit{discretely divergence
  free finite element spaces}
\begin{align*}
  V_h:=\{{v}_h\in X_h \fdg \langle\divo {v}_h,\eta_h\rangle_Z=0\text{ 
for all }\eta_h\in Z_h\}.
\end{align*}
For a null sequence\footnote{A null sequence is a sequence converging to zero.}
$(h_n)_{n\in \mathbb{N}}\subseteq\left(0,\infty\right)$ and
$V_n:=V_{h_n}$, $n\in \mathbb{N}$, we formulate the following
algorithm of a space-time discrete approximation of \eqref{eq:p-NS}:
\begin{alg}
  For given $K,n\in \mathbb{N}$ and ${u}_n^0\in V_n$ the sequence of
  iterates $({u}_n^{k})_{k=0,\dots,K}\subseteq V_n$ is given solving the
  implicit Rothe--Galerkin scheme for $\tau=\frac{T}{K}$ and $k=1,\dots,K$
  \begin{align}
    (d_\tau {u}^k_n,{v}_n)_Y+\langle [S]
    {u}^k_n,{v}_n\rangle_X+\langle \hat{B} {u}^k_n,{v}_n\rangle_X=\langle [\boldsymbol{f}]^\tau_k,{v}_n\rangle_X\quad\text{ for all }{v}_n\in V_n,\label{eq:p-NSnon}
  \end{align}
  where $\hat{B}:X\to X^*$ is given via
  $\langle \hat{B}{u},{v}\rangle_X:=\frac{1}{2}\int_{\Omega}
  {{v}\otimes{u}:\nabla{u}}\, dx-\frac{1}{2} \int_{\Omega}{{u}\otimes
    {u}:\nabla {v}}\, dx$ for all ${{u},{v}\in X}$.
\end{alg}

The operator $\hat{B}$ can be viewed as a symmetrized extension of $B$, as
$\langle\hat{B}{u},{v}\rangle_X=\langle B{u},{v}\rangle_X$ for all
${u},{v}\in V$, which in contrast to $B$ fulfils
$\langle \hat{B}{u},{u}\rangle_X=0$ for all ${u}\in X$ (and not only for all ${u}\in V$), and therefore
guarantees the stability of the scheme \eqref{eq:p-NSnon}.
	
      The sequence $(V_n)_{n\in \mathbb{N}}$ violates the conditions
      (\hyperlink{C.1}{C.1}) and (\hyperlink{C.2}{C.2}). However, the
      assumptions 
    (\hyperlink{QNC.1}{QNC.1}) and (\hyperlink{QNC.2}{QNC.2}) on the discrete spaces
      $(V_n)_{n\in \mathbb{N}}$ are often fulfilled under mild
      assumptions, e.g., if one assumes that
      ${\mathcal{P}_1(\mathcal{T}_h)^d\subset X_h}$,
      $\mathbb{R}\subset Z_h$, and that there exist linear interpolation operators 
     $\Uppi_h^{\divo }:X\to X_h$ and $\Uppi_h^{Z}:Z\to Z_h$
      which are locally  $W^{1,1}$-stable and locally  $L^{1}$-stable,
      resp., and that $\Uppi_h^{\divo }$ preserves the divergence in
      $Z_h^*$ (cf.~Section \ref{sec:7} or \cite{BBDR12}, \cite{tscherpel-phd} for
      more details). 

        \bigskip

	\textbf{Plan of the paper:} In Section \ref{sec:2} we recall
        some basic definitions and results concerning the theory of
        pseudo-monotone operators and evolution equations. In Section
        \ref{sec:3} we introduce the concept of quasi non-conforming
        approximations. In Section \ref{sec:4} we introduce  quasi
        non-conforming Bochner pseudo-monotonicity, and give
        sufficient and easily verifiable conditions on families of
        operators such that the corresponding induced operator
        satisfies this concept. In Section \ref{sec:5} we recall some
        basic facts about the Rothe scheme. In Section \ref{sec:6} we
        formulate the scheme of a fully discrete, quasi non-conforming
        approximation of an evolution equation, prove that this scheme
        is well-defined, i.e., the existence of iterates, that the
        corresponding family of piece-wise constant interpolants
        satisfies certain a-priori estimates. Moreover, we formulate
        and prove the main result of this paper,  Theorem \ref{5.17}, which shows the
        existence of  a diagonal subsequence which weakly converges to
        a weak solution of the corresponding evolution equation. In
        Section \ref{sec:7} we apply this approximation scheme to two
        problems describing incompressible non-Newtonian fluid flow.
        In Section \ref{sec:8} we present some numerical experiments
        for one of the problems.
	\section{Preliminaries}
	\label{sec:2} 
	\subsection{Operators}
	For a Banach space $X$ with norm $\|\cdot\|_X$ we denote by $X^*$ its 
	dual space equipped with the norm ${\|\cdot\|_{X^*}}$. The duality pairing is denoted 
	by $\langle\cdot,\cdot\rangle_X$. All occurring Banach spaces
        are assumed to be real. 
	\begin{defn}\label{2.7}
          Let $X$ and $Y$ be Banach spaces.  The operator
          $A: X\rightarrow Y$ is said to be
          \begin{description}[{(iii)}]
          \item[{(i)}] \textbf{bounded}, if for all bounded subsets
            $M\subseteq X$ the image $A(M)\subseteq Y$ is bounded.
          \item[{(ii)}] \textbf{pseudo-monotone}, if $Y=X^*$, and for
            $(u_n)_{n\in \mathbb{N}}\subseteq X$ from
            $u_n\rightharpoonup u$ in $X$ $(n\to\infty)$ and
            ${\limsup_{n\to\infty}{\langle Au_n,u_n-u\rangle_X}\leq
              0}$, it follows
            $\langle Au,u-v\rangle_X\leq \liminf_{n\to\infty}{\langle
              Au_n,u_n-v\rangle_X}$ for every $v\in X$.
          \item[{(iii)}] \textbf{coercive}, if $Y=X^*$ and
            $ \lim_{\|u\|_X\rightarrow\infty} {\frac{\langle
                Au,u\rangle_X}{\|u\|_X}}=\infty$.
          \end{description}
        \end{defn}

	\begin{prop}\label{2.9} If $X$ is a reflexive Banach space and $A:X\to X^*$ a bounded, pseudo-monotone, and coercive operator, then $R(A)=X^*$.
	\end{prop}

	\begin{proof}
		See \cite[Corollary 32.26]{Zei90B}.\hfill$\qed$
	\end{proof}

        	\begin{lem}\label{2.9a} If $X$ is a reflexive Banach
                  space and $A:X\to X^*$ a locally bounded and pseudo-monotone
                  operator, then $A$ is demi-continuous.
	\end{lem}

	\begin{proof}
		See \cite[Proposition 27.7]{Zei90B}.\hfill$\qed$
	\end{proof}

	\subsection{Evolution equations}
	We call $(V,H,j)$ an \textbf{evolution triple}, if $V$ is a
        reflexive Banach space, $H$ is a Hilbert space and $j:V\to H$
        is a dense embedding, i.e., $j$ is a linear, injective and bounded
        operator with $\overline{j(V)}^{\|\cdot \|_H}=H$. 
	Let $R:H\rightarrow H^*$ be the Riesz
	isomorphism with respect~to~${(\cdot,\cdot)_H}$. As $j$ is a dense
	embedding the adjoint \mbox{$j^*:H^*\rightarrow V^*$} and
	therefore $e:=j^*Rj: V \rightarrow V^*$ are embeddings as
	well. We call $e$ the
	\textbf{canonical embedding} of $(V,H,j)$. Note that
	\begin{align*} 
          \langle ev,w\rangle_V=(jv,jw)_H\quad\text{ for all }v,w\in
          V.
	\end{align*} 
	For an evolution triple $(V,H,j)$, $I:=\left(0,T\right)$,
        $T<\infty$, and $1\leq p\leq q\leq \infty$ we define operators
        ${\boldsymbol{j}:L^p(I,V)\to L^p(I,H)}\colon \boldsymbol u \to
        \boldsymbol{j}\boldsymbol{u} $ and
        $\boldsymbol{j}^*:L^{q'}(I,H^*)\to L^{q'}(I,V^*) \colon
        \boldsymbol v \to \boldsymbol{j}^*\boldsymbol{v}$, where
        $\boldsymbol{j}\boldsymbol{u} $ and
        $\boldsymbol{j}^*\boldsymbol{v} $ are 
        for every $\boldsymbol{u}\in L^p(I,V)$ and
        ${\boldsymbol{v}\in L^{q'}(I,H^*)}$ given via
	\begin{alignat*}{2}
          (\boldsymbol{j}\boldsymbol{u})(t)&:=j(\boldsymbol{u}(t))&&\quad\text{ in }H\quad\;\text{ for a.e. }t\in I,\\
          (\boldsymbol{j}^*\boldsymbol{v})(t)&:=j^*(\boldsymbol{v}(t))&&\quad\text{ in }V^*\quad\text{ for a.e. }t\in I.
        \end{alignat*}
        It is shown in \cite[Proposition~2.19]{alex-rose-hirano} that both $\boldsymbol{j}$
        and $\boldsymbol{j}^*$ are embeddings, which we call
        \textbf{induced embeddings}. 
	Moreover, we define the  intersection space
	\begin{align*}
			L^p(I,V)\cap_{\boldsymbol{j}}L^q(I,H):=\{\boldsymbol{u}\in L^p(I,V)\fdg \boldsymbol{j}\boldsymbol{u}\in L^q(I,H)\}, 
	\end{align*}
	which forms a Banach space equipped with the canonical sum norm
	\begin{align*}
		\|\cdot\|_{	L^p(I,V)\cap_{\boldsymbol{j}}L^q(I,H)}:=\|\cdot\|_{L^p(I,V)}+\|\boldsymbol{j}(\cdot)\|_{L^q(I,H)}.
	\end{align*}
	If $1<p\leq q<\infty$, then $L^p(I,V)\cap_{\boldsymbol{j}}L^q(I,H)$ is additionally reflexive. Furthermore, for each  $\boldsymbol{u}^*\in (L^p(I,V)\cap_{\boldsymbol{j}}L^q(I,H))^*$ there exist functions $\boldsymbol{g}\in L^{p'}(I,V^*)$ and $\boldsymbol{h}\in L^{q'}(I,H^*)$, such that for every $\boldsymbol{u}\in L^p(I,V)\cap_{\boldsymbol{j}}L^q(I,H)$ it holds
	\begin{align}
		\langle \boldsymbol{u}^*,\boldsymbol{u}\rangle_{L^p(I,V)\cap_{\boldsymbol{j}}L^q(I,H)}=\int_I{\langle\boldsymbol{g}(t)+(\boldsymbol{j}^*\boldsymbol{h})(t),\boldsymbol{u}(t)\rangle_V\,dt},\label{eq:dual}
	\end{align}
	and
        $\|\boldsymbol{u}^*\|_{(L^p(I,V)\cap_{\boldsymbol{j}}L^q(I,H))^*}:=\|\boldsymbol{g}\|_{L^{p'}(I,V^*)}+\|\boldsymbol{h}\|_{L^{q'}(I,H^*)}$,
        i.e., $(L^p(I,V)\cap_{\boldsymbol{j}}L^q(I,H))^*$ is
        isometrically isomorphic to the sum
        $L^{p'}(I,V^*)+\boldsymbol{j}^*(L^{q'}(I,H^*))$
        (cf.~\cite[Kapitel I, Bemerkung 5.13 \& Satz 5.13]{GGZ}), which is a Banach space equipped with the  norm
	\begin{align*}
		\|\boldsymbol{f}\|_{L^{p'}(I,V^*)+\boldsymbol{j}^*(L^{q'}(I,H^*))}:=\min_{\substack{\boldsymbol{g}\in L^{p'}(I,V^*)\\\boldsymbol{h}\in L^{q'}(I,H^*)\\\boldsymbol{f}=\boldsymbol{g}+\boldsymbol{j}^*\boldsymbol{h}}}{\|\boldsymbol{g}\|_{L^{p'}(I,V^*)}+\|\boldsymbol{h}\|_{L^{q'}(I,H^*)}}.
	\end{align*}
	
	\begin{defn}[Generalized time derivative]\label{2.15}
		Let $(V,H,j)$ be an evolution triple, ${I:=\left(0,T\right)}$, $T<\infty$, and $1< p\leq q<\infty$. A function
		$\boldsymbol{u}\in L^p(I,V)\cap_{\boldsymbol{j}}L^q(I,H)$ possesses a \textbf{generalized time derivative with respect to  the canonical embedding $e$ of $(V,H,j)$} if there exists a function $\boldsymbol{u}^*\in L^{p'}(I,V^*)+\boldsymbol{j}^*(L^{q'}(I,H^*))$ such that for all $v\in V$ and $\varphi\in C_0^\infty(I)$
		\begin{align*}
		-\int_I{(j(\boldsymbol{u}(s)),jv)_H\varphi^\prime(s)\,ds}=
		\int_I{\langle\boldsymbol{u}^*(s),v\rangle_{V}\varphi(s)\,ds}.
		\end{align*}
		As this function $\boldsymbol{u}^*\in L^{p'}(I,V^*)+\boldsymbol{j}^*(L^{q'}(I,H^*))$ is unique (cf. \cite[Proposition 23.18]{Zei90A}),
		$\frac{d_e\boldsymbol{u}}{dt}:=\boldsymbol{u}^*$ is well-defined. By 
		\begin{align*}
		\mathbfcal{W}^{1,p,q}_e(I,V,H):=\Big\{\boldsymbol{u}\in L^p(I,V)\cap_{\boldsymbol{j}}L^q(I,H)\;\Big|\;\exists\, \frac{d_e\boldsymbol{u}}{dt}\in L^{p'}(I,V^*)+\boldsymbol{j}^*(L^{q'}(I,H^*))\Big\}
		\end{align*}
		we denote the \textbf{Bochner-Sobolev space with respect to $e$}.
	\end{defn}

	\begin{prop}[Formula of integration by parts]\label{2.16}
		Let $(V,H,j)$ be an evolution triple, $I:=\left(0,T\right)$, $T<\infty$,  and $1<p\leq q<\infty$. Then, it holds:
		\begin{description}[(iii)]
			\item[(i)] The space $\mathbfcal{W}^{1,p,q}_e(I,V,H)$ forms a Banach space equipped with the norm 
			\begin{align*}
			\|\cdot\|_{\mathbfcal{W}^{1,p,q}_e(I,V,H)}:=\|\cdot\|_{L^p(I,V)\cap_{\boldsymbol{j}}L^q(I,H)}+\left\|\frac{d_e\,\cdot}{dt}\right\|_{L^{p'}(I,V^*)+{\boldsymbol{j}^*}(L^{q'}(I,H^*))}.
			\end{align*}
			\item[(ii)] Given $\boldsymbol{u}\in \mathbfcal{W}^{1,p,q}_e(I,V,H)$ the
			function $\boldsymbol{j}\boldsymbol{u}\in L^q(I,H)$
                        possesses a unique representation
			$\boldsymbol{j}_c\boldsymbol{u}\in C^0(\overline{I},H)$, and the resulting mapping
			${\boldsymbol{j}_c:\mathbfcal{W}^{1,p,q}_e(I,V,H)\rightarrow
			C^0(\overline{I},H)}$ is an embedding.
			\item[(iii)] \textbf{Generalized integration by parts formula:} It holds
			\begin{align*}
			\int_{t'}^t{\Big\langle
				\frac{d_e\boldsymbol{u}}{dt}(s),\boldsymbol{v}(s)\Big\rangle_V\,ds}
			=\left[((\boldsymbol{j}_c\boldsymbol{u})(s), (\boldsymbol{j}_c
			\boldsymbol{v})(s))_H\right]^{s=t}_{s=t'}-\int_{t'}^t{\Big\langle
				\frac{d_e\boldsymbol{v}}{dt}(s),\boldsymbol{u}(s)\Big\rangle_V\,ds},
			\end{align*}
			for all $\boldsymbol{u},\boldsymbol{v}\in \mathbfcal{W}^{1,p,q}_e(I,V,H)$ and
			$t,t'\in \overline{I}$ with $t'\leq t$.
		\end{description}  
	\end{prop}
	
	\begin{proof}
		See \cite[Kapitel IV, Satz 1.16 \&  Satz 1.17]{GGZ}.\hfill $\qed$
	\end{proof}

	For an evolution triple $(V,H,j)$, $I:=\left(0,T\right)$, $T<\infty$, and $1<p\leq q< \infty$ we call an operator $\mathbfcal{A}:L^p(I,V)\cap_{\boldsymbol{j}}L^q(I,H)\to (L^p(I,V)\cap_{\boldsymbol{j}}L^q(I,H))^*$ \textbf{induced} by a family of operators $A(t):V\to V^*$, $t\in I$, if for every $\boldsymbol{u},\boldsymbol{v}\in L^p(I,V)\cap_{\boldsymbol{j}}L^q(I,H)$ it holds
		\begin{align}
			\langle\mathbfcal{A}\boldsymbol{u},\boldsymbol{v}\rangle_{L^p(I,V)\cap_{\boldsymbol{j}}L^q(I,H)}=\int_I{\langle A(t)(\boldsymbol{u}(t)),\boldsymbol{v}(t)\rangle_V\,dt}.\label{eq:induced}
		\end{align}
		
	\begin{rmk}[Need for $L^p(I,V)\cap_{\boldsymbol{j}}L^q(I,H)$]
          Note that an operator family $A(t):V\to V^*$, $t\in I$ can
          define an induced operator in different spaces. In
          \cite{alex-rose-hirano}, \cite{K19} the induced operator $\mathbfcal{A}$
          is considered as an operator from
          $L^p(I,V)\cap_{\boldsymbol{j}}L^\infty(I,H)$ into
          $(L^p(I,V))^*$. Here, we consider the induced operator 
          $\mathbfcal{A}$ as an operator from
          $L^p(I,V)\cap_{\boldsymbol{j}}L^q(I,H)$ into
          $(L^p(I,V)\cap_{\boldsymbol{j}}L^q(I,H))^*$, which is  more
          general and enables us to consider
          operator families with significantly worse growth
          behavior. Here, the so-called \textit{Temam modification}
          $\hat{B}:X\to X^*$, tracing back to \cite{Tem68}, \cite{Tem77},
          of the convective term $B :X\to X^*$ defined in \eqref{eq:sb},
          defined for $p>\frac{3d+2}{d+2}$ and all ${u},{v}\in X$ via
          \begin{align*}
            \langle
            \hat{B}{u},{v}\rangle_X=
            \frac{1}{2}\int_\Omega{{v}\otimes{u}:
            \nabla{u}\,dx}-
            \frac{1}{2}\int_\Omega{{u}\otimes{u}:
            {\nabla}{v}\,dx},
          \end{align*}
          serves as a prototypical example. In fact, following
          \cite[Example 5.1]{alex-rose-hirano}, one can prove that ${B:X\to X^*}$ 
          satisfies for $d=3$ and $p\ge \frac{11}{5}$ the estimate
          \begin{gather}\label{eq:esti}
            \|B{u}\|_{X^*}\leq
            c(1+\|{u}\|_Y)(1+\|{u}\|_X^{p-1}),
          \end{gather}
          for all ${u}\in X$ and that corresponding induced
          operator $\mathbfcal{B}$ is well-defined and bounded as an
          operator from $L^p(I,X)\cap L^\infty(I,Y)$ to
          $(L^p(I,X))^*$. Regrettably, for the remaining term in
          Temam's modification, i.e., for the operator
          $\tilde{B}:=\hat{B}-\frac{1}{2}B:X\to X^*$, we can prove
          \eqref{eq:esti} for $d=3$ only for $p>\frac{13}{5}$. In
          order to reach 
          $p>\frac{11}{5}$ for $d=3$, one is forced to use a larger
          target space, i.e., we view the induced operator of
          $\tilde B$ as an operator from $ L^p(I,X)\cap L^q(I,Y)$ to $ (L^p(I,X)\cap L^q(I,Y))^*$,
          where $q\in [p,\infty)$ is specified in the proof Proposition~\ref{4.7}.
	\end{rmk}
	
                \begin{defn}[Weak solution]
                  \label{2.17}
                  Let $(V,H,j)$ be an evolution triple,
                  $I:=\left(0,T\right)$, $T<\infty$, and
                  $1<p\leq q< \infty$. Moreover, let
                  ${u}_0\in H$ be an initial value,
                  $\boldsymbol{f}\in L^{p'}(I,V^*)$ a right-hand side,
                  and
                  $\mathbfcal{A}:L^p(I,V)\cap_{\boldsymbol{j}}L^q(I,H)\to
                  (L^p(I,V)\cap_{\boldsymbol{j}}L^q(I,H))^*$ induced
                  by a family of operators $A(t):V\to V^*$, $t\in
                  I$. A function
                  $\boldsymbol u \in \mathbfcal{W}^{1,p,q}_e(I,V,H)$
                  is called \textbf{weak solution} of the initial
                  value problem \eqref{eq:1.1} if
                  $(\boldsymbol{j}_c\boldsymbol{u})(0)={u}_0$
                  in $H $ and for all $\boldsymbol{\phi}\in
                  C^1_0(I,V)$ there holds 

		\begin{align*}
                    \int_I{\Big\langle\frac{d_e\boldsymbol{u}}{dt}(t),\boldsymbol{\phi}(t)\Big\rangle_V\,dt}+\int_I{\langle A(t)(\boldsymbol{u}(t)),\boldsymbol{\phi}(t)\rangle_V\,dt}&=\int_I{\langle \boldsymbol{f}(t),\boldsymbol{\phi}(t)\rangle_V\,dt}.
		\end{align*}
         \end{defn}
              
        Here, the initial condition is well-defined since due to
        Proposition \ref{2.16} (ii) there exists the unique continuous representation
        $\boldsymbol{j}_c\boldsymbol{u}\in C^0(\overline{I},H)$ of $\boldsymbol u \in \mathbfcal{W}^{1,p,q}_e(I,V,H)$.

	\section{Quasi non-conforming approximation}
	\label{sec:3}
	In this section we introduce the concept of quasi non-conforming approximations.
	
	\begin{defn}[Quasi non-conforming approximation]\label{3.1}
		Let $(V,H,j) $ and $(X,Y,j)$ be evolution triples such that 
		$V\subseteq X$ with $\|\cdot\|_V=\|\cdot\|_X$ in $V$
                and ${H\subseteq
                  Y}$~with~${(\cdot,\cdot)_H=(\cdot,\cdot)_Y}$~in~${H\times
                  H}$. Moreover, let $I:=\left(0,T\right)$,
                $T<\infty$, and let $1<p<\infty$. We call a sequence
                of closed subspaces $(V_n)_{n\in\mathbb{N}}$ of $X$ a \textbf{quasi non-conforming approximation of $V$ in $X$}, if the following properties are satisfied:
		\begin{description}[{(ii)}]
			\item[\textbf{(QNC.1)}] \hypertarget{QNC.1}{} There exists a dense subset $C\subseteq V$, such that for each $v\in C$ there exist elements $v_n\in V_n$, $n\in\mathbb{N}$, such that $v_n\to v $ in $X$~$(n\to\infty)$.
			\item[\textbf{(QNC.2)}] \hypertarget{QNC.2}{} For each sequence $\boldsymbol{u}_n\in L^p(I,V_{m_n})$, $n\in \mathbb{N}$, where ${(m_n)_{n\in\mathbb{N}}\subseteq\mathbb{N}}$ with $m_n\to \infty$ $(n\to\infty)$, from $\boldsymbol{u}_n\rightharpoonup \boldsymbol{u}$ in $L^p(I,X)$~$(n\to\infty)$, it follows that $\boldsymbol{u}\in L^p(I,V)$.
		\end{description}
	\end{defn}
	The next proposition shows that the notion of a quasi
        non-conforming approximation is indeed a generalization of the
        usual notion of a conforming approximation. In Section
        \ref{sec:7} we will show that our motivating example, namely
        the approximation of divergence-free Sobolev functions through
	discretely divergence-free finite element spaces, fits into the
        framework of quasi non-conforming approximations.

	\begin{prop}\label{ex} Let $(X,Y,j)$ and $(V,H,j)$ be as in Definition \ref{3.1}. Then, it holds:
		\begin{description}[{(ii)}]
			\item[(i)] The constant approximation $V_n=V$,
                          $n\in \mathbb{N}$, is a quasi non-conforming
                          approximation of $V$ in $X$.
			\item[(ii)] If $(V_n)_{n\in \mathbb{N}}$ is a conforming approximation 
 of $V$, i.e., $(V_n)_{n\in \mathbb{N}}$ satisfy (\hyperlink{C.1}{C.1}) and (\hyperlink{C.2}{C.2}), then $(V_n)_{n\in \mathbb{N}}$ is a quasi non-conforming approximation of~$V$~in~$X$.
		\end{description}
	\end{prop}

	\begin{proof} 
		\textbf{ad (i)} Follows right from the definition.
		
		\textbf{ad (ii)} We set $C:=\bigcup_{n\in \mathbb{N}}{V_n}$. Then, for each $v\in C$ there exists an integer $n_0\in \mathbb{N}$ such that $v\in V_n$ for every $n\ge n_0$. Therefore, the sequence $v_n\in V_n$, $n\in 
\mathbb{N}$, given via $v_n:=0$ if $n<n_0$ and $v_n:=v$ if $n\ge n_0$, satisfies $v_n\to v$ in $V$ $(n\to \infty)$, i.e., $(V_n)_{n\in \mathbb{N}}$ satisfies (\hyperlink{QNC.1}{QNC.1}). Apart from that, $(V_n)_{n\in 
\mathbb{N}}$ obviously fulfills (\hyperlink{QNC.2}{QNC.2}).\hfill$\qed$
	\end{proof}

        The following proposition will be crucial in verifying that the
induced operator $\mathbfcal A$ of a family of operators $(A(t))_{t\in
I}$ is quasi non-conforming Bochner pseudo-monotone (cf.~Definition~\ref{3.4}). 
	\begin{prop}\label{3.2}
          Let $(V,H,j) $ and $(X,Y,j)$ be as in Definition \ref{3.1}
          and let $(V_n)_{n\in\mathbb{N}}$ be a quasi non-conforming
          approximation of $V$ in $X$. Then, the following statements hold true:
		\begin{description}[{(iii)}]
			\item[(i)] For a sequence $v_n\in V_{m_n}$,
                          $n\in \mathbb{N}$, where
                          $(m_n)_{n\in\mathbb{N}}\subseteq\mathbb{N}$
                          with $m_n\to \infty$ $(n\to \infty)$, from
                          $v_n\rightharpoonup v$ in $X$ $(n\to
                          \infty)$, it follows that $v\in V$. 
			\item[(ii)] For a sequence $v_n\in V_{m_n}$, $n\in \mathbb{N}$, where $(m_n)_{n\in\mathbb{N}}\subseteq\mathbb{N}$ with $m_n\to \infty$ $(n\to \infty)$, with $\sup_{n\in \mathbb{N}}{\|v_n\|_X}<\infty$, and $v\in V$ the following statements are equivalent:
			\begin{description}
				\item[(a)]  $v_n\rightharpoonup v$ in $X$ $(n\to\infty)$.
				\item[(b)] $P_Hjv_n\rightharpoonup jv$ in $H$ $(n\to\infty)$, where $P_H\!:\!Y\!\to\! H$ is the orthogonal projection of~$Y$~into~$H$.
			\end{description}
			\item[(iii)] For each $\eta\in H$ there exists a sequence $v_n\in V_{m_n}$, $n\in \mathbb{N}$, where $(m_n)_{n\in\mathbb{N}}\subseteq\mathbb{N}$ with $m_n\to \infty$ $(n\to \infty)$, such that $jv_n\to \eta$ in $Y$ $(n\to \infty)$.
		\end{description}
	\end{prop}
	
	\begin{proof}
		\textbf{ad (i)} Immediate consequence of (\hyperlink{QNC.2}{QNC.2}).
		
		\textbf{ad (ii)} \textbf{(a) $\boldsymbol{\Rightarrow}$ (b)} 
		 Follows from the weak continuity of $j:X\to Y$ and
                 $P_H:Y\to H$.

                 \textbf{(b) $\boldsymbol{\Rightarrow}$  (a)} From the reflexivity of $X$, we obtain a subsequence $(v_n)_{n\in \Lambda}$, with $\Lambda\subseteq \mathbb{N}$, and an element $\tilde{v}\in X$, such that $v_n\rightharpoonup\tilde{v}$ in $X$ $(\Lambda\ni n\to \infty)$. Due to \textbf{(i)} we infer $\tilde{v}\in V$.  From the weak continuity of $j:X\to Y$ and $P_H:Y\to H$ we conclude $P_Hjv_n\rightharpoonup P_Hj\tilde{v}=j\tilde{v}$ in $H$  $(\Lambda\ni n\to \infty)$. In consequence, we have $j\tilde{v}=jv$ in $H$, which in virtue of the injectivity of $j:V\to H$ 
implies that $\tilde{v}=v$ in $V$, and therefore
		 \begin{align}
		 	v_n\;\;\rightharpoonup \;\;v\quad\text{ in }X\quad(\Lambda\ni n\to \infty).\label{eq:3.3}
		 \end{align}
    Since this argumentation remains valid for each subsequence of $(v_n)_{n\in \mathbb{N}}\subseteq X$, $v\in V$ is weak accumulation point of each 
subsequence of $(v_n)_{n\in \mathbb{N}}\subseteq X$. Therefore, the standard convergence principle (cf.~\cite[Kap. I, Lemma 5.4]{GGZ}) guarantees that \eqref{eq:3.3} remains true even if $\Lambda=\mathbb{N}$.
		 
   \textbf{ad (iii)} Since $(V,H,j)$ is an evolution triple, $j(V)$
    is dense in $H$. As a result, for fixed $\eta\in H$ there exists a
    sequence $(v_n)_{n\in \mathbb{N}}\subseteq V$, such that
    $\|\eta-jv_n\|_H\leq 2^{-n}$ for all $n\in \mathbb{N}$. Due to
    (\hyperlink{QNC.1}{QNC.1}) there exist a sequence $(w_n)_{n\in
      \mathbb N} \subseteq C$, such that $\|v_n-w_n\|_V\le 2^{-n-1}$ for
    all $n\in \mathbb{N}$  and a double sequence
    $(v_k^n)_{n,k\in \mathbb{N}}\subseteq X$, with $v_k^n\in V_k$ for
    all $k,n\in \mathbb{N}$, such that $v_k^n\to w_n$ in $X$
    $(k\to \infty)$ for all $n\in \mathbb{N}$. Thus, for each
    $n\in \mathbb{N}$ there exists $m_n\in \mathbb{N}$, such that
    $\|w_n-v_k^n\|_X\leq 2^{-n-1}$ for all $k\ge m_n$. Then, we have
    $v_{m_n}^n\in V_{m_n}$ for all $n\in \mathbb{N}$ and
    $\|\eta-jv_{m_n}^n\|_Y\leq (1+c)2^{-n}$ for all $n\in
    \mathbb{N}$, where $c>0$ is the embedding constant of $j$.
\hfill$\qed$
	\end{proof}

	\section{Quasi non-conforming Bochner pseudo-monotonicity}
	\label{sec:4}
	
	In this section we introduce an extended notion of Bochner
        pseudo-monotonicity (cf.~\cite{alex-rose-hirano}, \cite{K19}), which incorporates a given quasi non-conforming approximation $(V_n)_{n\in \mathbb{N}}$.
	\begin{defn}\label{3.4}
		Let $(X,Y,j)$ and $(V,H,j) $ be as in Definition
                \ref{3.1} and let $(V_n)_{n\in\mathbb{N}}$ be a quasi
                non-conforming approximation of $V$ in $X$, $I:=\left(0,T\right)$, with $0<T<\infty$, and $1<p\leq q<\infty$. An operator $\mathbfcal{A}:L^p(I,X)\cap_{\boldsymbol{j}} L^q(I,Y)\to (L^p(I,X)\cap_{\boldsymbol{j}} L^q(I,Y))^*$ is said to be \textbf{quasi non-conforming Bochner pseudo-monotone with respect to $(V_n)_{n\in\mathbb{N}}$} if for a sequence $\boldsymbol{u}_n\in L^\infty(I,V_{m_n})$, $n\in\mathbb{N}$, where $(m_n)_{n\in\mathbb{N}}\subseteq\mathbb{N}$ with $m_n\to \infty$ $(n\to\infty)$, from
		\begin{alignat}{2}
		\boldsymbol{u}_n&\;\;\rightharpoonup\;\;
		\boldsymbol{u}&&\quad\text{ in }L^p(I,X)\;\quad (n\rightarrow\infty)\label{eq:3.5},
		\\
		\boldsymbol{j}\boldsymbol{u}_n&\;\;\overset{\ast}{\rightharpoondown}\;\;
		\boldsymbol{j}\boldsymbol{u}&&\quad\text{ in } L^\infty(I,Y)\quad (n\rightarrow\infty),
		\label{eq:3.6}
		\\
		P_H(\boldsymbol{j}\boldsymbol{u}_n)(t)&\;\;\rightharpoonup\;\;
		(\boldsymbol{j}\boldsymbol{u})(t) &&\quad\text{ in }H\quad (n\rightarrow\infty)\quad\text{for a.e. }t\in I,\label{eq:3.7}
		\end{alignat}
		and
		\begin{align}
		\limsup_{n\rightarrow\infty}{\langle \mathbfcal{A}\boldsymbol{u}_n,\boldsymbol{u}_n-\boldsymbol{u}\rangle_{L^p(I,X)\cap_{\boldsymbol{j}} L^q(I,Y)}}\leq 0,\label{eq:3.8}
		\end{align}
		it follows for all $\boldsymbol{v}\in
                L^p(I,X)\cap_{\boldsymbol{j}} L^q(I,Y)$ that 
		\begin{align*}
			\langle \mathbfcal{A}\boldsymbol{u},\boldsymbol{u}-\boldsymbol{v}\rangle_{L^p(I,X)\cap_{\boldsymbol{j}} L^q(I,Y)}\leq	\liminf_{n\rightarrow\infty}{\langle \mathbfcal{A}\boldsymbol{u}_n,\boldsymbol{u}_n-\boldsymbol{v}\rangle_{L^p(I,X)\cap_{\boldsymbol{j}} L^q(I,Y)}}.
		\end{align*}
	\end{defn}
        Note that \eqref{eq:3.5} and \eqref{eq:3.6} guarantee that $\boldsymbol{u}\in L^p(I,V)\cap_{\boldsymbol{j}} L^\infty(I,H)$ due to Definition \ref{3.1}.
\medskip 
        
	The basic idea of quasi non-conforming Bochner pseudo-monotonicity, in comparison
        to the original notion of Bochner pseudo-monotonicity tracing back to
        \cite{alex-rose-hirano}, consists in incorporating the finite dimensional
        approximation $(V_n)_{n\in \mathbb{N}}$ into the definition. 
	We will see in the proof of Theorem~\ref{5.17} that
	\eqref{eq:3.5}--\eqref{eq:3.8} are natural properties of a sequence
	$\boldsymbol{u}_n\in
L^p(I,V_{m_n})$, $n\in \mathbb{N}$, coming from \eqref{eq:fully}
(which is a quasi non-conforming
Rothe--Galerkin approximation of \eqref{eq:1.1}), if $\mathbfcal{A}$
	satisfies appropriate additional assumptions. In fact,
	\eqref{eq:3.5} usually is a consequence of the coercivity of
	$\mathbfcal{A}$, \eqref{eq:3.6} stems from the time derivative, while
	\eqref{eq:3.7} and \eqref{eq:3.8} follow directly from the 
	approximative scheme.

	\begin{prop}\label{3.9}
		Let $(X,Y,j)$ and $(V,H,j) $ be as in Definition \ref{3.1} and let $(V_n)_{n\in\mathbb{N}}$ be a quasi
                non-conforming approximation of $V$ in $X$, $I:=\left(0,T\right)$,  $T<\infty$, and $1<p\leq q<\infty$. 
		Moreover, let ${A(t):X\to X^*}$, $t\in I$, be a family of operators  with the following properties:
		\begin{description}[{(A.3)}]
			\item[\textbf{(A.1)}] \hypertarget{A.1} $A(t):X\to X^*$ is pseudo-monotone for almost every $t\in I$.
			\item[\textbf{(A.2)}] \hypertarget{A.2} $A(\cdot)u:I\to X^*$ is Bochner measurable for every $u\in X$.
			\item[\textbf{(A.3)}] \hypertarget{A.3} For some constants $c_0>0$ and 
$c_1,c_2\ge 0$ holds for almost every $t\in I$ and every $u\in X$
			\begin{align*}
			\langle A(t)u,u\rangle_X\ge c_0\|u\|_X^p-c_1\|ju\|_Y^2-c_2.
			\end{align*}
			\item[\textbf{(A.4)}] \hypertarget{A.4} For constants  $\gamma\ge 0$ and $\lambda\in \left(0,c_0\right)$ holds	for almost every $t\in I$ and every $u,v\in X$
			\begin{align*}
			\vert\langle A(t)u,v\rangle_X\vert \leq \lambda\|u\|_X^p+\gamma\big(1+\|ju\|_Y^q+\|jv\|_Y^q+\|v\|_X^p\big).
			\end{align*}
		\end{description}
	Then, the induced operator $\mathbfcal{A}:L^p(I,X)\cap_{\boldsymbol{j}} L^q(I,Y)\to (L^p(I,X)\cap_{\boldsymbol{j}} L^q(I,Y))^*$,~given~via~\eqref{eq:induced},
	is well-defined, bounded and quasi non-conforming  Bochner
        pseudo-monotone~with~respect~to~the subspaces ${(V_n)_{n\in \mathbb{N}}}$.
	\end{prop}
	
	\begin{proof}
		
          \textbf{1. Well-definiteness:} For
          $\boldsymbol{u}_1,\boldsymbol{u}_2\in
          L^p(I,X)\cap_{\boldsymbol{j}} L^q(I,Y)$ there exists
          sequences of simple functions
          $(\boldsymbol{s}_n^m)_{n\in \mathbb{N}}\subseteq
          L^\infty(I,X)$, $m=1,2$, i.e., 
          $\boldsymbol{s}_n^m(t)=\sum_{i=1}^{k_n^m}{s_{n,i}^m\chi_{E_{n,i}^m}(t)}$
          for $t\in I$ and $m=1,2$, where $s_{n,i}^m\in X$,
          $k_n^m\in \mathbb{N}$ and $E_{n,i}^m\in \mathcal{L}^1(I) $
          with $\bigcup_{i=1}^{k_n^m}{E_{n,i}^m}=I$ and $E_{n,i}^m\cap E_{n,j}^m=\emptyset$ for $i\neq j$, such
          that $ \boldsymbol{s}_n^m(t)\to \boldsymbol{u}_m(t)$ in $X$
          for almost every $t\in I$ and $m=1,2$. Moreover, it follows
          from Lemma~\ref{2.9a} that $A(t):X\to X^*$ is for almost
          every $t\in I$ demi-continuous, since it is for almost every
          $t\in I$ pseudo-monotone (cf.~(\hyperlink{A.1}{A.1})) and
          bounded (cf.~(\hyperlink{A.4}{A.4})). This yields for
          almost~every~${t\in I}$
	\begin{align}
		\langle A(t)(\boldsymbol{s}_n^1(t)),\boldsymbol{s}_n^2(t)\rangle_{X}=\sum_{i=1}^{k_n^1}{\sum_{j=1}^{k_n^2}{\langle A(t)s_{n,i}^1,s_{n,j}^2\rangle_{X}\chi_{E_{n,i}^1\cap E_{n,j}^2}\!(t)}}\overset{n\to\infty}{\to }\langle A(t)(\boldsymbol{u}_1(t)),\boldsymbol{u}_2(t)\rangle_{X}.\label{eq:3.92}
	\end{align}
	Thus, since the functions $(t\mapsto\langle A(t)s_{n,i}^1,s_{n,j}^2\rangle_{X}:I\to \mathbb{R}$, $i=0,\dots,k_n^1$, $j=1,\dots,k_n^2$, $n\in \mathbb{N}$, are Lebesgue measurable due to (\hyperlink{A.2}{A.2}), we conclude from \eqref{eq:3.92}~that~${(t\mapsto\langle A(t)(\boldsymbol{u}_1(t)),\boldsymbol{u}_2(t)\rangle_{X})\!:\!I\to \mathbb{R}}$ is Lebesgue measurable. In addition, using (\hyperlink{A.4}{A.4}), we obtain
	\begin{align}
		\begin{split}
		\int_I{\langle A(t)(\boldsymbol{u}_1(t)),\boldsymbol{u}_2(t)\rangle_{X}\,dt}&\leq \lambda \|\boldsymbol{u}_1\|_{L^p(I,X)}^p\\&\quad+\gamma [T+\|\boldsymbol{j}\boldsymbol{u}_1\|_{L^q(I,Y)}^q+\|\boldsymbol{j}\boldsymbol{u}_2\|_{L^q(I,Y)}^q+\|\boldsymbol{u}_2\|_{L^p(I,X)}^p],
		\end{split}\label{eq:3.91}
	\end{align}
	i.e., $\mathbfcal{A}:L^p(I,X)\cap_{\boldsymbol{j}} L^q(I,Y)\to (L^p(I,X)\cap_{\boldsymbol{j}} L^q(I,Y))^*$ is well-defined. \\[-3mm]
	
	\textbf{2. Boundedness:}
	As $\|\boldsymbol{v}\|_{L^p(I,X)\cap_{\boldsymbol{j}} L^q(I,Y)}\leq 1$ implies that $\|\boldsymbol{v}\|_{L^p(I,X)}^p+\|\boldsymbol{j}\boldsymbol{v}\|_{L^q(I,Y)}^q\leq 2$ for every $\boldsymbol{v}\in L^p(I,X)\cap_{\boldsymbol{j}} L^q(I,Y)$, we infer from \eqref{eq:3.91} for every $\boldsymbol{u}\in L^p(I,X)\cap_{\boldsymbol{j}} L^\infty(I,Y)$  that
	\begin{align*}
	\begin{split}
		\|\mathbfcal{A}\boldsymbol{u}\|_{(L^p(I,X)\cap_{\boldsymbol{j}} L^q(I,Y))^*}&=\sup_{\|\boldsymbol{v}\|_{L^p(I,X)\cap_{\boldsymbol{j}} L^q(I,Y)}\leq 1}{\langle \mathbfcal{A}\boldsymbol{u},\boldsymbol{v}\rangle_{L^p(I,X)\cap_{\boldsymbol{j}} L^q(I,Y)}}
		\\&\leq \lambda \|\boldsymbol{u}\|_{L^p(I,X)}^p+\gamma\|\boldsymbol{j}\boldsymbol{u}\|_{L^q(I,Y)}^q+\gamma [T+2],
		\end{split}
	\end{align*}
	 i.e., $\mathbfcal{A}:L^p(I,X)\cap_{\boldsymbol{j}}
         L^q(I,Y)\to (L^p(I,X)\cap_{\boldsymbol{j}} L^q(I,Y))^*$ is
         bounded. \\[-3mm]
         
	 
	 \textbf{3. Quasi non-conforming Bochner pseudo-monotonicity with respect to $(V_n)_{n\in\mathbb{N}}$:}
	 In principle, we proceed analogously to \cite[Proposition~3.13]{alex-rose-hirano}. However, as we have solely almost everywhere weak convergence 
of the orthogonal projections available, i.e., \eqref{eq:3.7}, in the definition of quasi-nonconforming Bochner pseudo-monotonicity (cf.~Definition~\ref{3.4}), the arguments in \cite{alex-rose-hirano} ask for some slight modifications. In fact, in this context the properties of the quasi non-conforming approximation $(V_n)_{n\in \mathbb{N}}$ come into play. Especially the role of Proposition \ref{3.2} will be crucial. We split the proof of the quasi non-conforming Bochner pseudo-monotonicity into four steps: \\[-3mm]
	 
	\textbf{3.1. Collecting information:} Let $\boldsymbol{u}_n\in L^\infty(I,V_{m_n})$, $n\in\mathbb{N}$, where $(m_n)_{n\in\mathbb{N}}\!\subseteq\!\mathbb{N}$~with~${m_n\!\to\!\infty}$ $(n\to\infty)$, be a sequence satisfying \eqref{eq:3.5}--\eqref{eq:3.8}.
		We fix an arbitrary $\boldsymbol{v}\in L^p(I,X)\cap_{\boldsymbol{j}} L^q(I,Y)$, and choose a subsequence 
		$(\boldsymbol{u}_n)_{n\in\Lambda}$, with
		$\Lambda\subseteq\mathbb{N}$, such that
		\begin{align}
		\lim_{\substack{n\rightarrow\infty\\n\in\Lambda}}{
			\langle \mathbfcal{A}\boldsymbol{u}_n,\boldsymbol{u}_n-\boldsymbol{v}\rangle_{L^p(I,X)\cap_{\boldsymbol{j}} L^q(I,Y)}}=
		\liminf_{n\rightarrow\infty}{\langle\mathbfcal{A}
			\boldsymbol{u}_n,\boldsymbol{u}_n-\boldsymbol{v}
			\rangle_{L^p(I,X)\cap_{\boldsymbol{j}} L^q(I,Y)}}.\label{eq:3.10}
		\end{align}
		Due to \eqref{eq:3.7} there exists a subset
		$E\subseteq I$, with $I\setminus E$ a null set\footnote{A null set is a
                  set of zero Lebesgue measure.}, such that for all $t\in 
E$
		\begin{align}
		P_H(\boldsymbol{j}\boldsymbol{u}_n)(t)\;\;\rightharpoonup\;\;(\boldsymbol{j}\boldsymbol{u})(t)\quad\text{ in }H\quad(n\to\infty).\label{eq:3.11}
		\end{align}
		Using (\hyperlink{A.3}{A.3}) and
		(\hyperlink{A.4}{A.4}), we get for every $\boldsymbol{z}\in L^p(I,X)\cap_{\boldsymbol{j}} L^q(I,Y)$~and~almost~every~${t\in I}$
		\begin{align}
		\begin{split}
		\langle A&(t)(\boldsymbol{u}_n(t)),\boldsymbol{u}_n(t)
		-\boldsymbol{z}(t)\rangle_X
		\\
		&\ge c_0\|\boldsymbol{u}_n(t)\|_X^p-
		c_1\|j(\boldsymbol{u}_n(t))\|_Y^2-c_2
		-\langle
		A(t)(\boldsymbol{u}_n(t)),\boldsymbol{z}(t)\rangle_X
		\\
		&\ge 	(c_0-\lambda)\|\boldsymbol{u}_n(t)\|_X^p
		-c_1K^2-c_2-\gamma\big[1+K^q+\|(\boldsymbol{j}\boldsymbol{z})(t)\|_Y^q+\|\boldsymbol{z}(t)\|_X^p\big],
		\end{split}\label{eq:3.12}
		\end{align}
		where $K:=\sup_{n\in\mathbb{N}}{\|\boldsymbol{j}\boldsymbol{u}_n\|_{L^\infty(I,Y)}}<\infty$ (cf.~\eqref{eq:3.6}). 
		If we set $\mu_{\boldsymbol{z}}(t):=-c_1K^2-c_2-\gamma\big[1+K^q+\|(\boldsymbol{j}\boldsymbol{z})(t)\|_Y^q+\|\boldsymbol{z}(t)\|_X^p\big]\in L^1(I)$, then \eqref{eq:3.12} reads 
		\begin{align}
		\langle
		A(t)(\boldsymbol{u}_n(t)),\boldsymbol{u}_n(t)-
		\boldsymbol{z}(t)\rangle_X\ge
		(c_0-\lambda)\|\boldsymbol{u}_n(t)\|_X^p
		-\mu_{\boldsymbol{z}}(t), \tag*{$(\ast)_{\boldsymbol{z},n,t}$}
		\end{align}
		for almost every $t\in I$ and all $n\in \Lambda$. Next, we define
		\begin{align*}
		{E_1}:=
		\big \{t\in E \fdg A(t):X\rightarrow X^*\text{ is
			pseudo-monotone},
		\vert\mu_{\boldsymbol{u}}(t)\vert<\infty\text{
			and }(\ast)_{\boldsymbol{u},n,t}\text{ holds for all }n\in\Lambda\big \}. 
		\end{align*}
		From the defining properties of $E_1$ it follows
                directly that $I\setminus{E_1}$ is a null set. \\[-3mm]
		
		\textbf{3.2. Intermediate objective:} Our next
                objective is to verify that for all $t\in E_1$
                there holds
		\begin{align}
		\liminf_{\substack{n\rightarrow\infty\\n\in\Lambda}}
		{\langle A(t)(\boldsymbol{u}_n(t)),\boldsymbol{u}_n(t)
			-\boldsymbol{u}(t)\rangle_X}\ge 0. \tag*{$(\ast\ast )_{t}$}
		\end{align}
		To this end, let us fix an arbitrary
		$t\in{E_1}$ and define
		\begin{align*}
		\Lambda_t:=\{n\in\Lambda\fdg
		\langle A(t)(\boldsymbol{u}_n(t)),\boldsymbol{u}_n(t)
		-\boldsymbol{u}(t)\rangle_X< 0\}.
		\end{align*}
		We assume without loss of generality that $\Lambda_t$
		is not finite. Otherwise, $(\ast\ast )_{t}$ would already hold true
		for this specific $t\in E_1$ and nothing would be left to
		do. But if $\Lambda_t$ is not finite, then 
		\begin{align}
		\limsup_{\substack{n\rightarrow\infty\\n\in\Lambda_t}}
		{\langle A(t)(\boldsymbol{u}_n(t)),\boldsymbol{u}_n(t)-
			\boldsymbol{u}(t)\rangle_X}\leq 0.\label{eq:3.13}
		\end{align}
		The definition of $\Lambda_t$ and $(\ast)_{\boldsymbol{u},n,t}$ imply for all $n\in\Lambda_t$
		\begin{align}
		(c_0-\lambda)\|\boldsymbol{u}_n(t)\|_X^p\leq\langle A(t)
		(\boldsymbol{u}_n(t)),\boldsymbol{u}_n(t)-
		\boldsymbol{u}(t)\rangle_X+
		\vert\mu_{\boldsymbol{u}}(t)\vert<\vert\mu_{\boldsymbol{u}}(t)\vert
		<\infty. \label{eq:3.14}
		\end{align}
		This 
                and $\lambda <c_0$ 
                yield that the sequence
                $(\boldsymbol{u}_n(t))_{n\in\Lambda_t}$ is bounded in
                $X$. In view of 
                \eqref{eq:3.11}, Proposition \ref{3.2} (ii) implies that
		\begin{align}\label{eq:psmon}
			\boldsymbol{u}_n(t)\;\;\rightharpoonup\;\;\boldsymbol{u}(t)\quad\text{
                  in }X\quad( \Lambda_t\ni n\to \infty). 
		\end{align}
		Since $A(t):X\rightarrow X^*$ is pseudo-monotone, we
                get from \eqref{eq:psmon} and \eqref{eq:3.13} that  
		\begin{align*}
		\liminf_{\substack{n\rightarrow\infty\\n\in\Lambda_t}}
		{\langle A(t)(\boldsymbol{u}_n(t)),\boldsymbol{u}_n(t)
			-\boldsymbol{u}(t)\rangle_X}\ge 0.
		\end{align*}
		Due to
		$\langle A(t)(\boldsymbol{u}_n(t)), \boldsymbol{u}_n(t)
		-\boldsymbol{u}(t)\rangle_X\ge 0$ for all
		$n\in\Lambda\setminus\Lambda_t$, 
		$(\ast\ast)_t$
		holds for all $t\in E_1$. \\[-3mm]

		\textbf{3.3. Switching to the image space level:} In this passage we
		verify the existence of a subsequence
		$(\boldsymbol{u}_n)_{n\in\Lambda_0}\subseteq L^p(I,X)\cap_{\boldsymbol{j}} L^\infty(I,Y)$,
		with $\Lambda_0\subseteq\Lambda$, such that for almost every $t\in I$
		\begin{align}
		\begin{split}
		\boldsymbol{u}_n(t)\;\;\rightharpoonup\;\;
		\boldsymbol{u}(t)\quad\text{ in }X\quad(\Lambda_0\ni n\to \infty),\\
		\limsup_{\substack{n\rightarrow\infty\\n\in\Lambda_0}}
		{\langle A(t)(\boldsymbol{u}_n(t)),\boldsymbol{u}_n(t)
			-\boldsymbol{u}(t)\rangle_X}\leq 0.
		\end{split}
		\label{eq:3.15}
		\end{align}
		As a consequence, we are in a position to
		exploit the almost everywhere pseudo-monotonicity of the operator
		family.  Thanks to
		$\langle
		A(t)(\boldsymbol{u}_n(t)),\boldsymbol{u}_n(t)-\boldsymbol{u}(t)\rangle_X\ge
                -\mu_{\boldsymbol{u}}(t)$ for all $t\in E_1$ and
		$n\in\Lambda$ (cf. $(\ast)_{\boldsymbol{u},n,t}$), Fatou's lemma (cf. \cite[Theorem 1.18]{Rou05}) is
		applicable. It yields, also using \eqref{eq:3.8} 
		\begin{align}
                  \begin{aligned}
                    0&\leq
                    \int_I{\liminf_{\substack{n\rightarrow\infty\\n\in\Lambda}}
                      {\langle
                        A(s)(\boldsymbol{u}_n(s)),\boldsymbol{u}_n(s)-
                        \boldsymbol{u}(s)\rangle_X}\,ds}
                    \\
                    &\leq
                    \liminf_{\substack{n\rightarrow\infty\\n\in\Lambda}}{\int_I{\langle
                        A(s)(\boldsymbol{u}_n(s)),\boldsymbol{u}_n(s)-\boldsymbol{u}(s)\rangle_X\,ds}}\label{eq:3.16}
                    \\
                    &\leq \limsup_{n\rightarrow\infty}
                    {\langle\mathbfcal{A}\boldsymbol{u}_n,\boldsymbol{u}_n
                      -\boldsymbol{u}\rangle_{L^p(I,X)\cap_{\boldsymbol{j}}
                        L^q(I,Y)}}
                    \\
                    &\leq 0.
                  \end{aligned}
		\end{align}
		Let us define $g_n(t):=\langle
		A(t)(\boldsymbol{u}_n(t)),\boldsymbol{u}_n(t)-\boldsymbol{u}(t)\rangle_{X}$. Then, $(\ast\ast)_t$ and 
		\eqref{eq:3.16} read:
		\begin{align}
                  \liminf_{\substack{n\rightarrow\infty\\n\in\Lambda}}{g_n(t)}
                  &\ge 0\quad\text{ for all }t\in E_1.\label{eq:3.17}
                  \\
		\lim_{\substack{n\rightarrow\infty\\n\in\Lambda}}{\int_I{g_n(s)\,ds}}&=0.\label{eq:3.18}
		\end{align}
		As $s\mapsto s^-:=\min\{0,s\}$ is continuous and non-decreasing we
		deduce from \eqref{eq:3.17} that
		\begin{align*}
		0\ge\limsup_{\substack{n\rightarrow\infty\\n\in\Lambda}}
		{g_n(t)^-}\ge\liminf_{\substack{n\rightarrow\infty\\n\in\Lambda}}
		{g_n(t)^-}\ge \min\Big\{0,
		\liminf_{\substack{n\rightarrow\infty\\n\in\Lambda}}{g_n(t)}\Big\}=0,
		\end{align*}
		i.e.,~$g_n(t)^-\to 0$ $(\Lambda\ni n\to\infty)$ for all $t\in E_1$. Since 
		$0\ge g_n(t)^-\ge -\mu_{\boldsymbol{u}}(t)$ for all
                $t\in E_1$ and
		$n\in\Lambda$, Vitali's theorem yields
		$g_n^-\to 0$ in $L^1(I)$ $(\Lambda\ni n\to\infty)$.  From
		the latter, $\vert g_n\vert=g_n-2g_n^-$ and \eqref{eq:3.18}, we
		conclude that $g_n\to0$ in
		$L^1(I)$ $(\Lambda\ni n\to\infty)$.  This provides a further subsequence
		$(\boldsymbol{u}_n)_{n\in\Lambda_0}$
		with $\Lambda_0\subseteq\Lambda$ and a subset $F\subseteq I$ such
		that $ I\setminus F$ is a null set and for all $t\in F$
		\begin{align}
		\lim_{\substack{n\rightarrow\infty\\n\in\Lambda_0}}{\langle
                  A(t)(\boldsymbol{u}_n(t)),\boldsymbol{u}_n(t)-\boldsymbol{u}(t)\rangle_X}=
                  0. \label{eq:3.19}
		\end{align}
	   This and \eqref{eq:3.14} implies for all $t\in
           {E_1}\cap F$ that 
		\begin{align*}
		\limsup_{\substack{n\rightarrow\infty\\n\in\Lambda_0}}
		{(c_0-\lambda)\|\boldsymbol{u}_n(t)\|_X^p}\leq
		\limsup_{\substack{n\rightarrow\infty\\n\in\Lambda_0}}
		{\langle A(t)(\boldsymbol{u}_n(t)),\boldsymbol{u}_n(t)
			-\boldsymbol{u}(t)\rangle_X+
			\vert\mu_{\boldsymbol{u}}(t)\vert}
		=\vert\mu_{\boldsymbol{u}}(t)\vert<\infty,
		\end{align*}
		i.e., $(\boldsymbol{u}_n(t))_{n\in \Lambda_0}$ is
                bounded in $X$ for all $t\in
                {E_1}\cap F$. Thus, \eqref{eq:3.11}
                and Proposition \ref{3.2} (ii) yield for all $t\in {E_1}\cap F$
		\begin{align}
		\boldsymbol{u}_n(t)\;\;\rightharpoonup\;\;\boldsymbol{u}(t)\quad\text{ in }
		X\quad(\Lambda_0\ni n \to \infty).\label{eq:3.20}
		\end{align} 
		The relations 
		\eqref{eq:3.19} and \eqref{eq:3.20} are just \eqref{eq:3.15}. \\[-3mm]

		\textbf{3.4. Switching to the Bochner-Lebesgue level:} From the
		pseudo-monotonicity of the operators $A(t):X\rightarrow X^*$ for
		all $t\in {E_1}\cap F$ we obtain almost every $t\in I$
		\begin{align*}
		\langle A(t)(\boldsymbol{u}(t)),\boldsymbol{u}(t)-
		\boldsymbol{v}(t)\rangle_X\leq
		\liminf_{\substack{n\rightarrow\infty\\n\in\Lambda_0}}
		{\langle A(t)(\boldsymbol{u}_n(t)),\boldsymbol{u}_n(t)
			-\boldsymbol{v}(t)\rangle_X}.
		\end{align*}
		Due to $(\ast)_{\boldsymbol v,n,t}$, we have
                $\langle A(t)(\boldsymbol{u}_n(t)),\boldsymbol{u}_n(t)
                -\boldsymbol{v}(t)\rangle_X\ge-\mu_{\boldsymbol{v}}(t)$
                for almost every $t\in
                I$~and~all~${n\in\Lambda_0}$. Thus, using the
                definition of the induced operator \eqref{eq:induced},
                Fatou's lemma and \eqref{eq:3.10} we deduce 
		\begin{align*}
		\langle \mathbfcal{A}\boldsymbol{u},\boldsymbol{u}
		-\boldsymbol{v}\rangle_{L^p(I,X)\cap_{\boldsymbol{j}} L^q(I,Y)}
		&\leq
		\int_I{\liminf_{\substack{n\rightarrow\infty\\n\in\Lambda_0}}
			{\langle A(s)(\boldsymbol{u}_n(s)),\boldsymbol{u}_n(s)
				-\boldsymbol{v}(s)\rangle_X}\,ds}
		\\
		&\leq \liminf_{\substack{n\rightarrow\infty\\n\in\Lambda_0}}
		{\int_I{\langle A(s)(\boldsymbol{u}_n(s)),\boldsymbol{u}_n(s)
				-\boldsymbol{v}(s)\rangle_X\,ds}}
		\\
		&=\lim_{\substack{n\rightarrow\infty\\n\in\Lambda}}
		{\langle \mathbfcal{A}\boldsymbol{u}_n,\boldsymbol{u}_n
			-\boldsymbol{v}\rangle_{L^p(I,X)\cap_{\boldsymbol{j}} L^q(I,Y)}}
		\\
		&= \liminf_{n\rightarrow\infty}{\langle \mathbfcal{A}
			\boldsymbol{u}_n,\boldsymbol{u}_n-\boldsymbol{v}\rangle_{L^p(I,X)\cap_{\boldsymbol{j}} L^q(I,Y)}}.
		\end{align*}
		As $\boldsymbol{v}\in L^p(I,X)\cap_{\boldsymbol{j}} L^q(I,Y)$ was chosen arbitrary, this completes the proof
		of Proposition \ref{3.9}.~\hfill$\qed$
	\end{proof}

	\section{Rothe scheme}		\label{sec:5}
	
	
        Let $X$ be a Banach space, and $I:=\left(0,T\right)$,
        ${T<\infty}$, be a finite time interval. For $K\in\mathbb{N}$
        we set ${\tau:=\frac{T}{K}}$,
        $I_k^\tau:=\left((k-1)\tau,k\tau\right]$, $k=1,\dots,K$, and
        $\mathcal{I}_\tau :=\{I_k^\tau\}_{k=1,\dots,K}$. Moreover, we
        denote by
		\begin{align*}
			\mathbfcal{S}^0(\mathcal{I}_\tau,X):=\{\boldsymbol{u}:I\to X\fdg \boldsymbol{u}(s)=\boldsymbol{u}(t)\text{ in }X\text{ for all }t,s\in I_k^\tau,k=1,\dots,K\}\subset L^\infty(I,X)
		\end{align*}
		 the \textbf{space of piece-wise constant functions
                   with respect to $\mathcal{I}_\tau$}. For a given
                 finite sequence $(u^k)_{k=0,\dots,K}\subseteq X$ the
                 \textbf{backward difference quotient} operator is
                 defined via 
		\begin{align*}
		d_\tau u^k:=\frac{1}{\tau}(u^k-u^{k-1})\quad\text{ in }X,\quad k=1,\dots,K.
		\end{align*}
		 Furthermore, we denote for a given finite sequence $(u^k)_{k=0,\dots,K}\subseteq X $ by
		 $\overline{\boldsymbol{u}}^\tau\in \mathbfcal{S}^0(\mathcal{I}_\tau,X)$ the \textbf{piece-wise constant interpolant}, and by $\hat{\boldsymbol{u}}^\tau\in W^{1,\infty}(I,X)$ the \textbf{piece-wise affine interpolant},  for every $t\in I_k^\tau$ and $k=1,\dots,K$ given via
		 \begin{align}\label{eq:polant}
		 	\overline{\boldsymbol{u}}^\tau(t):=u^k,\qquad\hat{\boldsymbol{u}}^\tau(t):=\Big(\frac{t}{\tau}-(k-1)\Big)u^k+\Big(k-\frac{t}{\tau}\Big)u^{k-1}\quad\text{ in }X.
		 \end{align}
		In addition, if $(X,Y,j)$ is an evolution triple and $(u^k)_{k=0,\dots,K}\subseteq X$ a finite sequence, then it holds for $k,l=0,\dots,K$ the \textbf{discrete integration by parts formula} 
		\begin{align}
			\int_{k\tau}^{l\tau}{\Big\langle \frac{d_e\hat{\boldsymbol{u}}^\tau}{dt}(t),\overline{\boldsymbol{u}}^\tau(t)\Big\rangle_X\,dt}\ge \frac{1}{2}\|ju^l\|_Y^2-\frac{1}{2}\|ju^k\|_Y^2,\label{eq:4.2}
		\end{align}
		which is an immediate consequence of the identity
		$\langle d_\tau eu^k,u^k\rangle_X=\frac{1}{2}d_\tau\|ju^k\|_Y^2+\frac{\tau}{2}\|d_\tau ju^k\|_Y^2$ for every $k=1,\dots,K$.

	For the discretization of the right-hand side in
        \eqref{eq:1.1} we use the following construction. 
          Let $X$ be a Banach space, $I=\left(0,T\right)$, $T<\infty$, $K \in
        \mathbb N$, $\tau:= \frac{T}{K}>0$
          and $1<p<\infty$. The \textbf{Clem\'ent
            $0$-order quasi-interpolation operator}
          ${\mathscr{J}_\tau:L^p(I,X)\to
            \mathbfcal{S}^0(\mathcal{I}_\tau,X)}$ is defined for every
          $\boldsymbol{u}\in L^p(I,X)$ via
		\begin{align*}
		\mathscr{J}_\tau[\boldsymbol{u}]:=\sum_{k=1}^K{[\boldsymbol{u}]_k^\tau\chi_{I_k^\tau}}\quad\text{ in }\mathbfcal{S}^0(\mathcal{I}_\tau,X),\qquad[\boldsymbol{u}]_k^\tau:=\fint_{I_k^\tau}{\boldsymbol{u}(s)\,ds}\in 
X.
		\end{align*}
	\begin{prop}\label{4.4}
	For every $\boldsymbol{u}\in L^p(I,X)$  it holds:
		\begin{description}[{(iii)}]
			\item[(i)] $\mathscr{J}_\tau[\boldsymbol{u}]\to\boldsymbol{u}$ in $L^p(I,X)$ $(\tau\to 0)$, i.e., $\bigcup_{\tau>0}\mathbfcal{S}^0(\mathcal{I}_\tau,X)$ is dense in $L^p(I,X)$.
			\item[(ii)] $\sup_{\tau>0}{\|\mathscr{J}_\tau[\boldsymbol{u}]\|_{L^p(I,X)}}\leq \|\boldsymbol{u}\|_{L^p(I,X)}$.
		\end{description}
	\end{prop}

	\begin{proof}
		See \cite[Remark 8.15]{Rou05}.\hfill$\qed$
	\end{proof}

	Since we treat non-autonomous evolution equations we also need
        to discretize the time dependent family of operators in
        \eqref{eq:1.1}.  This will also be obtained by means of the
        Clem\'ent $0$-order quasi-interpolant.  Let $(X,Y,j)$ be an
        evolution triple, $I:=\left(0,T\right)$, $T<\infty$, $K \in
        \mathbb N$, $\tau= \frac{T}{K}>0$ and
        $1<p\leq q<\infty$. Let $A(t):X\to X^*$, $t\in I$,
        be a family of operators satisfying the conditions 
        (\hyperlink{A.1}{A.1})--(\hyperlink{A.4}{A.4}), and denote by
        $\mathbfcal{A}:L^p(I,X)\cap_{\boldsymbol{j}}L^q(I,Y)\to
        (L^p(I,X)\cap_{\boldsymbol{j}}L^q(I,Y))^*$ the induced
        operator (cf.~\eqref{eq:induced}). The
        \textbf{k-th temporal mean}
        $[A]^\tau_k:X\to X^*$, $k=1,\dots,K$,
        of $A(t):X\to X^*$, $t\in I$, is defined for every $u\in X$ via
        \begin{align*}
          [A]^\tau_k u:=\fint_{I_k^\tau}{A(s)u\,ds}\quad\text{ in }X^*.
        \end{align*} 
        The \textbf{Clement $0$-order quasi-interpolant}
        $\mathscr{J}_\tau[A](t):X\to X^*$, $t\in I$, of
        $A(t):X\to X^*$, $t\in I$, is defined for almost every
        $t\in I$ and $u\in X$ via
		\begin{align*}
		\mathscr{J}_\tau[A](t)u:=\sum_{k=1}^{K}{\chi_{I_k^\tau}(t)[A]^\tau_ku}\quad\text{ in }X^*.
		\end{align*} 
 The \textbf{Clement $0$-order quasi-interpolant}
 $\mathscr{J}_\tau[\mathbfcal{A}]:L^p(I,X)\cap_{\boldsymbol{j}}L^q(I,Y)\to
 (L^p(I,X)\cap_{\boldsymbol{j}}L^q(I,Y))^*$, of
 ${\mathbfcal{A}:L^p(I,X)\cap_{\boldsymbol{j}}L^q(I,Y)\to
   (L^p(I,X)\cap_{\boldsymbol{j}}L^q(I,Y))^*}$  is for all $\boldsymbol{u},\boldsymbol{v}\in L^p(I,X)\cap_{\boldsymbol{j}}L^q(I,Y)$ defined via
		\begin{align*}
		\langle \mathscr{J}_\tau[\mathbfcal{A}]\boldsymbol{u},\boldsymbol{v}\rangle_{L^p(I,X)\cap_{\boldsymbol{j}}L^q(I,Y)}:=\int_I{\langle \mathscr{J}_\tau[A] (t)(\boldsymbol{u}(t)),\boldsymbol{v}(t)\rangle_X\,dt}.
		\end{align*} 
Note that $\mathscr{J}_\tau[\mathbfcal{A}]$ is the induced operator of
the family of operators $\mathscr{J}_\tau[A](t):X\to X^*$, $t \in I$.
	\begin{prop}[Clement $0$-order quasi-interpolant for induced
          operators]\label{4.6} \newline
 Let $A(t):X\to X^*$, $t\in I$,
        be a family of operators satisfying the conditions 
        (\hyperlink{A.1}{A.1})--(\hyperlink{A.4}{A.4}), and denote by
        $\mathbfcal{A}:L^p(I,X)\cap_{\boldsymbol{j}}L^q(I,Y)\to
        (L^p(I,X)\cap_{\boldsymbol{j}}L^q(I,Y))^*$ the induced
        operator (cf.~\eqref{eq:induced}). Then, there holds: 
		\begin{description}[{(iii)}]
			\item[(i)] $[A]^\tau_k:X\to X^*$ is well-defined, bounded,  pseudo-monotone, and satisfies: \\[-3mm]
			\begin{description}[{(a)}]
				\item[(i.a)] $\langle [A]^\tau_ku,v\rangle_X\leq \lambda\|u\|_X^p+\gamma[1+\|ju\|_Y^q+\|jv\|_Y^q+\|v\|_X^p]$ for all $u,v\in X$. \\[-3mm]
				\item[(i.b)] $\langle [A]^\tau_ku,u\rangle_X\ge c_0\|u\|_X^p-c_1\|ju\|_Y^2-c_2$ for all $u\in X$. \\[-3mm]
			\end{description}
			\item[(ii)]  $\mathscr{J}_\tau[A](t):X\to X^*$, $t\in I$, satisfies the conditions (\hyperlink{A.1}{A.1})--(\hyperlink{A.4}{A.4}). \\[-3mm]
			
		\item[(iii)] $\mathscr{J}_\tau[\mathbfcal{A}]\!:\!L^p(I,X)\cap_{\boldsymbol{j}}\!L^q(I,Y)\!\to\!(L^p(I,X)\cap_{\boldsymbol{j}}\!L^q(I,Y))^*$~is~well-defined,~bounded~and~satisfies: \\[-3mm]
				\begin{description}[{(a)}]
				\item[(iii.a)] For all $\boldsymbol{u}_\tau\in \mathbfcal{S}^0(\mathcal{I}_\tau,X)$, $\boldsymbol{v}\in L^p(I,X)\cap_{\boldsymbol{j}}L^q(I,Y)$ 
holds
				\begin{align*}
					\langle \mathscr{J}_\tau[\mathbfcal{A}]\boldsymbol{u}_\tau,\boldsymbol{v}\rangle_{L^p(I,X)\cap_{\boldsymbol{j}}L^q(I,Y)}=\langle \mathbfcal{A}\boldsymbol{u}_\tau, \mathscr{J}_\tau[\boldsymbol{v}]\rangle_{L^p(I,X)\cap_{\boldsymbol{j}}L^q(I,Y)}.
				\end{align*}
				\item[(iii.b)] If the functions $\boldsymbol{u}_\tau\in\mathbfcal{S}^0(\mathcal{I}_\tau,X)$, $\tau>0$, are bounded in $L^p(I,X)\cap_{\boldsymbol{j}}L^q(I,Y)$, then
				\begin{align*}
						\mathbfcal{A}\boldsymbol{u}_\tau-\mathscr{J}_\tau[\mathbfcal{A}]\boldsymbol{u}_\tau\;\;\rightharpoonup\;\;\boldsymbol{0}\quad\text{ in }(L^p(I,X)\cap_{\boldsymbol{j}}L^q(I,Y))^*\quad(\tau\to 0).
				\end{align*}
				\item[(iii.c)] 
                                  \begin{gather*}
                                    \|\mathscr{J}_\tau[\mathbfcal{A}]\boldsymbol{u}_\tau\|_{(L^p(I,X)\cap_{\boldsymbol{j}}L^q(I,Y))^*}\leq
                                    \|\mathbfcal{A}\boldsymbol{u}_\tau\|_{(L^p(I,X)\cap_{\boldsymbol{j}}L^q(I,Y))^*}.
                                  \end{gather*}
			\end{description}
		\end{description}
	\end{prop}

	\begin{proof}
		\textbf{ad (i)} Let $u\in X$. Due to
                (\hyperlink{A.2}{A.2}) the function $A(\cdot)u:I\to
                X^*$ is Bochner measurable.  (\hyperlink{A.4}{A.4})
                guarantees that $\|A(\cdot)u\|_{X^*}\in L^1(I)$, and
                thus the Bochner integrability~of $A(\cdot)u:I\to
                X^*$. As a result, the Bochner integral
                $[A]^\tau_ku=\fint_{I^\tau_k}{A(s)u\,ds}\in X^*$
                exists, i.e.,  $[A]^\tau_k:X\to X^*$ is
                well-defined. The inequalities \textbf{(i.a)} and
                \textbf{(i.b)} are obvious. In particular, we gain
                from inequality \textbf{(i.a)}  the boundedness of
                $[A]^\tau_k:X\to X^*$. So, it is left to show the
                pseudo-monotonicity. Therefore, let $(u_n)_{n\in
                  \mathbb{N}}\subseteq X$ be a sequence such that 
		\begin{align}
			u_n\;\;\rightharpoonup\;\; u\quad\text{ in }X\quad(n\to\infty),\label{eq:4.7}\\
			\limsup_{n\to\infty}{\langle [A]^\tau_ku_n,u_n-u\rangle_X}\leq 0.\label{eq:4.8}
		\end{align}
		If we set $\boldsymbol{u}_n:=u_n\chi_{I_k^\tau}\in
                L^\infty(I,X)$, $n\in \mathbb{N}$, and
                $\boldsymbol{u}:=u\chi_{I_k^\tau}\in L^\infty(I,X)$,
                then  \eqref{eq:4.7}, the Lebesgue theorem on
                dominated convergence and the properties of the
                induced embedding $\boldsymbol{j}$ imply 
		\begin{alignat}{2}
		\boldsymbol{u}_n&\;\;\rightharpoonup\;\; \boldsymbol{u}&&\quad\text{ in 
}L^p(I,X)\;\quad(n\to \infty),\label{eq:4.9}\\
		\boldsymbol{j}\boldsymbol{u}_n&\;\;\overset{\ast}{\rightharpoondown}\;\;\boldsymbol{j}\boldsymbol{u}&&\quad\text{ in }L^\infty(I,Y)\quad(n\to\infty),\label{eq:4.10}\\
		(\boldsymbol{j}\boldsymbol{u}_n)(t)&\overset{n\to\infty}{\rightharpoonup}(\boldsymbol{j}\boldsymbol{u})(t)&&\quad\text{ in }Y\quad\text{for a.e. 
}t\in I.\label{eq:4.11}
		\end{alignat}
		In addition, from \eqref{eq:4.8} we infer
		\begin{align}
			\limsup_{n\to\infty}{\langle \mathbfcal{A}\boldsymbol{u}_n,\boldsymbol{u}_n-\boldsymbol{u}\rangle_{L^p(I,X)\cap_{\boldsymbol{j}}L^q(I,Y)}}=\tau\limsup_{n\to\infty}{\langle [A]^\tau_ku_n,u_n-u\rangle_{X}}\leq 0.\label{eq:4.12}
		\end{align}
		Note that the constant approximation $V_n=X$, $n\in
                \mathbb{N}$, is trivially a quasi non-conforming
                approximation of $X$ in $X$ (cf.~Remark \ref{rem:6.5} (i)). Thus,
                Proposition~\ref{3.9} yields that the induced operator 
                $\mathbfcal{A} :L^p(I,X)\cap_{\boldsymbol{j}}L^q(I,Y)\to (L^p(I,X)\cap_{\boldsymbol{j}}L^q(I,Y))^*$
                is quasi
                non-conforming Bochner pseudo-monotone with respect to
                $V_n=X$, $n\in \mathbb{N}$. Consequently, we obtain from
                \eqref{eq:4.9}--\eqref{eq:4.12} that for all
                $\boldsymbol{v}\in
                L^p(I,X)\cap_{\boldsymbol{j}}L^q(I,Y)$ there holds 
		\begin{align}
		\langle \mathbfcal{A}\boldsymbol{u},\boldsymbol{u}-\boldsymbol{v}\rangle_{L^p(I,X)\cap_{\boldsymbol{j}}L^q(I,Y)}\leq\liminf_{n\to\infty}{\langle 
\mathbfcal{A}\boldsymbol{u}_n,\boldsymbol{u}_n-\boldsymbol{v}\rangle_{L^p(I,X)\cap_{\boldsymbol{j}}L^q(I,Y)}}.\label{eq:4.13}
		\end{align}
		If we choose in \eqref{eq:4.13} $\boldsymbol{v}:=v\chi_{I_k^\tau}\in
                L^\infty(I,X)$ with $v\in X$ and divide by $\tau>0$, we conclude
		\begin{align*}
			\langle [A]^\tau_ku,u-v\rangle_{X }\leq\liminf_{n\to\infty}{\langle [A]^\tau_ku_n,u_n-v\rangle_{X}}.
		\end{align*}
                In other words, $[A]^\tau_k:X\to X^*$ is pseudo-monotone. 
\\[-3mm]
		
		\textbf{ad (ii)} The assertion follows immediately from \textbf{(i)} and the definition of $\mathscr{J}_\tau[A](t)$, $t\in I$. \\[-3mm]
		
		\textbf{ad (iii)} Since $\mathscr{J}_\tau[\mathbfcal{A}]$ is the induced operator of
the family of operators $\mathscr{J}_\tau[A](t)$, $t \in I$,  the well-definiteness and boundedness
                of $\mathscr{J}_\tau[\mathbfcal{A}]$ results from
                \textbf{(ii)} in conjunction with
                Proposition~\ref{3.9} applied again in the trivial
                setting of the constant approximation $V_n=X$, $n\in
                \mathbb{N}$. \\[-3mm]

                \allowdisplaybreaks		
		\textbf{ad (iii.a)} Let $\boldsymbol{u}_\tau\in \mathbfcal{S}^0(\mathcal{I}_\tau,X)$ and $\boldsymbol{v}\in L^p(I,X)\cap_{\boldsymbol{j}}L^q(I,Y)$. Then, using for every $t,s\in I_k^\tau$, $k=1,\dots,K$, that $\langle 
A(s)(\boldsymbol{u}_\tau(t)),\boldsymbol{v}(t)\rangle_X=\langle A(s)(\boldsymbol{u}_\tau(s)),\boldsymbol{v}(t)\rangle_X$ and Fubini's theorem, we infer 
		\begin{align*}
				\langle \mathscr{J}_\tau[\mathbfcal{A}]\boldsymbol{u}_\tau,\boldsymbol{v}\rangle_{L^p(I,X)\cap_{\boldsymbol{j}}L^q(I,Y)}&=\int_I{\langle \mathscr{J}_\tau[A](t)(\boldsymbol{u}(t)),\boldsymbol{v}(t)\rangle_X\,dt}\\&=\sum_{k=1}^{K}{\int_{I_k^\tau}{\Big\langle\fint_{I_k^\tau}{ A(s)(\boldsymbol{u}_\tau(t))\,ds},\boldsymbol{v}(t)\Big\rangle_{\!\!X}dt}}
				\\&=\sum_{k=1}^{K}{\int_{I_k^\tau}{\Big\langle A(s)(\boldsymbol{u}_\tau(s)),\fint_{I_k^\tau}{\boldsymbol{v}(t)\,dt}\Big\rangle_{\!\!X}\,ds}}\\&=
				\int_I{\langle A(s)(\boldsymbol{u}(s)), \mathscr{J}_\tau[\boldsymbol{v}](s)\rangle_{X}\,ds}
				\\&=\langle \mathbfcal{A}\boldsymbol{u}_\tau,\mathscr{J}_\tau[\boldsymbol{v}]\rangle_{L^p(I,X)\cap_{\boldsymbol{j}}L^q(I,Y)}.
		\end{align*}
		
		\textbf{ad (iii.b)} Let the family $\boldsymbol{u}_\tau\in\mathbfcal{S}^0(\mathcal{I}_\tau,X)$, $\tau>0$, be bounded in $L^p(I,X)\cap_{\boldsymbol{j}}L^q(I,Y)$. Then, by  the boundedness of $\mathbfcal{A}:L^p(I,X)\cap_{\boldsymbol{j}}L^q(I,Y)\to (L^p(I,X)\cap_{\boldsymbol{j}}L^q(I,Y))^*$ (cf.~Proposition~\ref{3.9}), the family $(\mathbfcal{A}\boldsymbol{u}_\tau)_{\tau>0}\subseteq (L^p(I,X)\cap_{\boldsymbol{j}}L^q(I,Y))^*$ is bounded 
as well. Therefore, also using \textbf{(iii.a)}, we conclude for every $\boldsymbol{v}\in L^p(I,X)\cap_{\boldsymbol{j}}L^q(I,Y)$ that
		\begin{align*}
			\langle \mathbfcal{A}\boldsymbol{u}_\tau-\mathscr{J}_\tau[\mathbfcal{A}]\boldsymbol{u}_\tau,\boldsymbol{v}\rangle_{L^p(I,X)\cap_{\boldsymbol{j}}L^q(I,Y)}=	\langle \mathbfcal{A}\boldsymbol{u}_\tau,\boldsymbol{v}-\mathscr{J}_\tau[\boldsymbol{v}]\rangle_{L^p(I,X)\cap_{\boldsymbol{j}}L^q(I,Y)}\;\;\to\;\; 0\quad(\tau\to 0),
		\end{align*}
		where we also used Proposition \ref{4.4} (i). \\[-3mm]
		
		\textbf{ad (iii.c)} Using \textbf{(iii.a)} and
                Proposition \ref{4.4} (ii), we deduce
		\begin{align*}
		\|\mathscr{J}_\tau[\mathbfcal{A}]\boldsymbol{u}_\tau\|_{(L^p(I,X)\cap_{\boldsymbol{j}}L^q(I,Y))^*}&=\sup_{\|\boldsymbol{v}\|_{L^p(I,X)\cap_{\boldsymbol{j}}L^q(I,Y)}\leq 1}{\langle \mathscr{J}_\tau[\mathbfcal{A}]\boldsymbol{u}_\tau,\boldsymbol{v}\rangle_{L^p(I,X)\cap_{\boldsymbol{j}}L^q(I,Y)}}\\&=\sup_{\|\boldsymbol{v}\|_{L^p(I,X)\cap_{\boldsymbol{j}}L^q(I,Y)}\leq 1}{\langle \mathbfcal{A}\boldsymbol{u}_\tau,\mathscr{J}_\tau[\boldsymbol{v}]\rangle_{L^p(I,X)\cap_{\boldsymbol{j}}L^q(I,Y)}}\\&\leq 	\|\mathbfcal{A}\boldsymbol{u}_\tau\|_{(L^p(I,X)\cap_{\boldsymbol{j}}L^q(I,Y))^*}.\tag*{$\qed$}
		\end{align*}
	\end{proof}

	\section{Fully discrete, quasi non-conforming approximation}
	\label{sec:6}
	
	In this section we formulate the exact framework of a  quasi
        non-conforming Rothe--Galerkin approximation, prove its
        well-posedness, i.e., the existence of iterates, and its
        stability, i.e., the boundedness of the corresponding double
        sequence of piece-wise constant interpolants. Moreover, we
        prove the main result of this paper, Theorem \ref{5.17}, which
        shows the weak convergence of a diagonal subsequence towards a
        weak solution of problem \eqref{eq:1.1}.

	\begin{asum}\label{asum}
		Let  $I\!:=\!\left(0,T\right)$, $T\!<\!\infty$ and $1\!<\!p\!\leq \!q\!<\!\infty$. We make the following assumptions:
		\begin{description}[{(iii)}]
                \item[(i)] \textbf{Spaces:} $(V,H,j)$ and
                  $(X,Y,j)$ are as in Definition \ref{3.1} and 
                  $(V_n)_{n\in\mathbb{N}}$ is a quasi non-conforming approximation of $V$ in $X$.
                \item[(ii)] \textbf{Initial data:} ${u}_0\!\in\!H$ and there is a sequence $u_n^0\!\in\! V_n$, $n\in\mathbb{N}$, such that ${u_n^0\to {u}_0}$~in~$Y$~${(n\!\to\! \infty)}$ and $\sup_{n\in \mathbb{N}}{\|ju_n^0\|_Y}\leq 
\|{u}_0\|_H$.\footnote{For a quasi non-conforming approximation Proposition \ref{3.2} guarantees the existence of such a sequence.}
			\item[(iii)] \textbf{Right-hand side}: $\boldsymbol{f}\in L^{p'}(I,X^*)$.
			\item[(iv)] \textbf{Operators}: $A(t):X\to X^*$, $t\in I$, is a family 
of operators satisfying (\hyperlink{A.1}{A.1})--(\hyperlink{A.4}{A.4}) and $\mathbfcal{A}:L^p(I,X)\cap_{\boldsymbol{j}}L^\infty(I,Y)\to (L^p(I,X)\cap_{\boldsymbol{j}}L^q(I,Y))^*$ the corresponding induced operator.
		\end{description}
	\end{asum}
	
	 Furthermore, we set $H_n:=j(V_n)\subseteq Y$ equipped with $(\cdot,\cdot)_Y$, denote by $j_n:V_n\to H_n$ the restriction of $j$ to $V_n$ and by $R_n:H_n\to H_n^*$ the corresponding Riesz isomorphism with respect to $(\cdot,\cdot)_Y$. As $j_n$ is an isomorphism, the triple $(V_n,H_n,j_n)$ 
is an evolution triple with canonical embedding $e_n:=j_n^*R_nj_n:V_n\to V_n^*$, which satisfies
	\begin{align}
	\langle e_nv_n,w_n\rangle_{V_n}=(j_nv_n,j_nw_n)_Y\quad\text{ for
          all }v_n,w_n\in V_n. \label{eq:iden}
	\end{align}
	Putting all together leads us to the following algorithm:

	\begin{alg}[Quasi non-conforming Rothe--Galerkin scheme] 
          Let Assumption \eqref{asum} be satisfied.
          For given $K,n\in \mathbb{N}$ 
          the sequence of iterates 
          $(u_n^{k})_{k=0,\dots,K}\subseteq V_n$ is  given solving the
          implicit Rothe--Galerkin scheme for $\tau=\frac{T}{K}$ and $k=1,\dots,K$
          \begin{align}
            (d_\tau ju_n^k,jv_n)_Y+\langle [A]^\tau_k u_n^k,v_n\rangle_X= 
\langle [\boldsymbol{f}]_k^\tau,v_n\rangle_X\quad\text{ for all }v_n\in V_n.\label{eq:4.15}
          \end{align}
	\end{alg}

		\begin{rmk}\label{rem:6.5}
			Note that the Rothe--Galerkin scheme \eqref{eq:4.15} also covers pure Rothe schemes, i.e., without spatial approximation, and fully discrete conforming approximations:
			\begin{description}[{(ii)}]
				\item[(i)] If $X=V$, $Y=H$, and
                                  $V_n=X$, $n\in \mathbb{N}$, then \eqref{eq:4.15} is  a pure Rothe scheme.
				\item[(ii)] If $X=V$, $Y=H$, and the
                                  closed subspaces $(V_n)_{n\in
                                    \mathbb{N}}$ satisfy
                                  (\hyperlink{C.1}{C.1})--(\hyperlink{C.2}{C.2}),
                                  then \eqref{eq:4.15} is a
                                  conforming Rothe--Galerkin scheme. 
			\end{description}
		\end{rmk}
	
                \begin{prop}[Well-posedness of \eqref{eq:4.15}]\label{5.1}
		Let Assumption \eqref{asum} be satisfied and set
                \linebreak ${\tau_0:=\frac{1}{4c_1}}$. Then, for all $K,n\in \mathbb{N}$ with $\tau=\frac{T}{K}<\tau_0$ there exist iterates $(u_n^k)_{k=1,\dots,K}\subseteq V_n$, solving~\eqref{eq:4.15}.
	\end{prop}
	
	\begin{proof}
		Using \eqref{eq:iden} and the identity mapping $\textup{id}_{V_n}\colon 
V_n
                \to X$, we see that \eqref{eq:4.15} is equivalent to 
		\begin{align}
		(\textup{id}_{V_n})^*\big ([\boldsymbol{f}]^\tau_k\big
                  )+\frac{1}{\tau}e_nu_n^{k-1} \in  R\Big(\frac{1}{\tau}e_n+(\textup{id}_{V_n})^*\circ [A]^\tau_k \circ
                  \textup{id}_{V_n}\Big),\quad\text{for all }k=1,\dots,K.\label{eq:5.2}
		\end{align}
		We fix an arbitrary $k=1,\dots,K$. Apparently,
                $\frac{1}{\tau}e_n:V_n\to V_n^*$ is linear and
                continuous. Using \eqref{eq:iden}, we infer that
                $\langle
                \frac{1}{\tau}e_{n}u,u\rangle_{V_n}=\frac{1}{\tau}\|j_nu\|_Y^2\ge
                0$ for all $u\in V_n$, i.e.,
                $\frac{1}{\tau}e_n:V_n\to V_n^*$ is positive definite, and
                thus monotone. In consequence,
                $\frac{1}{\tau}e_n:V_n\to V_n^*$ is pseudo-monotone. Since
                the conditions
                (\hyperlink{A.1}{A.1})--(\hyperlink{A.4}{A.4}) are
                inherited from $A\colon X\to X^*$ to
                $(\textup{id}_{V_n})^*\circ A \circ
                \textup{id}_{V_n}\colon V_n \to V_n^*$ and since
                $(\textup{id}_{V_n})^*\circ [A]^\tau_k\circ
                \textup{id}_{V_n}=[(\textup{id}_{V_n})^*\circ A \circ
                \textup{id}_{V_n}]^\tau_k$, Proposition \ref{4.6} (i)
                guarantees that the operator
                $(\textup{id}_{V_n})^*\circ [A]_\tau^k \circ
                \textup{id}_{V_n}:V_n\to V_n^*$ is bounded and
                pseudo-monotone. Altogether, we conclude that the sum
                $\frac{1}{\tau}e_n+(\textup{id}_{V_n})^*\circ
                [A]_\tau^k\circ \textup{id}_{V_n}:V_n\to V_n^*$ is
                bounded and pseudo-monotone.  In addition, as
                $\tau<\frac{1}{2c_1}$, combining \eqref{eq:iden} and
                Proposition~\ref{4.6}~(i.b), provides for all
                $u\in V_n$
                \begin{align*}
                \Big\langle \Big(\frac{1}{\tau}e_n+(\textup{id}_{V_n})^* \circ
                  [A]^\tau_k\circ \textup{id}_{V_n}\Big)u,u\Big\rangle_{V_n}\ge
                3c_1\|j_nu\|_Y^2+c_0\|u\|_X^p-c_2,
                \end{align*}
                i.e.,  $\frac{1}{\tau}e_n+(\textup{id}_{V_n})^*\circ
                [A]^\tau_k\circ \textup{id}_{V_n}:V_n\to V_n^*$ is coercive. Hence, Proposition~\ref{2.9} proves \eqref{eq:5.2}.\hfill$\qed$
	\end{proof}
	
	\begin{prop}[Stability of \eqref{eq:4.15}]\label{apriori}
	Let Assumption \eqref{asum} be satisfied and set
        $\tau_0:=\frac{1}{4c_1}$. Then, there exists a constant $M>0$
        (not depending on $K,n\in \mathbb{N}$), such that the
        piece-wise constant interpolants
        $\overline{\boldsymbol{u}}_n^\tau\in
        \mathbfcal{S}^0(\mathcal{I}_\tau,X)$, $K,n\in \mathbb{N}$ with
        $\tau=\frac{T}{K}\in \left(0,\tau_0\right)$, and piece-wise affine
        interpolants $\hat{\boldsymbol{u}}_n^\tau\in
        W^{1,\infty}(I,X)$, $n\in \mathbb{N}$, $\tau\in (0,\tau_0)$ (cf.~\eqref{eq:polant}) generated by iterates~${(u_n^k)_{k=0,\dots,K}\subseteq 
V_n}$, $K,n\in \mathbb{N}$ with
        $\tau=\frac{T}{K}\in \left(0,\tau_0\right)$, solving \eqref{eq:4.15}, satisfy the following estimates:
		\begin{align}
		\|\overline{\boldsymbol{u}}_n^\tau\|_{L^p(I,X)\cap_{\boldsymbol{j}}L^\infty(I,Y)}&\leq M,\label{eq:5.4}\\
		\|\boldsymbol{j}\hat{\boldsymbol{u}}_n^\tau\|_{L^\infty(I,Y)}&\leq M,\label{eq:5.5}\\
		\|\mathbfcal{A}\overline{\boldsymbol{u}}_n^\tau\|_{(L^p(I,X)\cap_{\boldsymbol{j}}L^q(I,Y))^*}&\leq M,\label{eq:5.6}\\
		\|e_n(\hat{\boldsymbol{u}}_n^\tau-\overline{\boldsymbol{u}}_n^\tau)\|_{L^{q'}(I,V_n^*)}&\leq\tau  \big(\|\boldsymbol{f}\|_{L^{p'}(I,X^*)}+M\big).\label{eq:5.7}
		\end{align}
	\end{prop}
	
	\begin{proof}
		We use $v_n=u_n^k\in V_n$,  $k=1,\dots,l$, for arbitrary
                $l=1,\dots,K$ in \eqref{eq:4.15}, multiply~by~${\tau\in (0,\tau_0)}$, sum with respect to $k=1,\dots,l$, use \eqref{eq:4.2} and $\sup_{n\in \mathbb{N}}{\|ju^0_n\|_Y\leq \|{u}_0\|_H}$, to obtain~for~every~${l=1,\dots,K}$
		\begin{align}\begin{split}
		\frac{1}{2}\|j u_n^l\|_Y^2+\sum_{k=1}^l{\tau\langle[A]^k_\tau u_n^k,u_n^k\rangle_X}\leq \frac{1}{2}\|{u}_0\|_H^2+
		\sum_{k=1}^l{\tau\langle[\boldsymbol{f}]^k_\tau ,u_n^k\rangle_X}.
		\end{split}
		\label{eq:5.8}
		\end{align}
		Applying the weighted $\varepsilon$-Young inequality with constant $c(\varepsilon):=(p\varepsilon)^{1-p'}/p'$ for all $\varepsilon>0$, using $\|\mathscr{J}_\tau[\boldsymbol{f}]\|_{L^{p'}(I,X^*)}\leq \|\boldsymbol{f}\|_{L^{p'}(I,X^*)}$ (cf.~Proposition~\ref{4.4} (ii)), we deduce for every 
$l=1,\dots,K$
		\begin{align*}
		\begin{split}
		\sum_{k=1}^l{\tau\langle[\boldsymbol{f}]^k_\tau ,u_n^k\rangle_X}=\langle\mathscr{J}_\tau[\boldsymbol{f}],\overline{\boldsymbol{u}}_n^\tau\chi_{\left[0,l\tau\right]}\rangle_{L^p(I,X)}\leq c(\varepsilon)\|\boldsymbol{f}\|_{L^{p'}(I,X^*)}^{p'}+\varepsilon\int _0^{l\tau}{\|\overline{\boldsymbol{u}}^\tau_n(s)\|_X^p\,ds}.
		\end{split}
		\end{align*}
		In addition, using Proposition \ref{4.6} (i.b), we obtain for every $l=1,\dots,K$
		\begin{align}
		\sum_{k=1}^l{\tau\langle[A]^k_\tau u_n^k,u_n^k\rangle_X}\ge
		c_0\int_0^{l\tau}{\|\overline{\boldsymbol{u}}^\tau_n(s)\|_X^p\,ds}-\tau 
c_1\|ju_n^l\|_Y^2-\sum_{k=1}^{l-1}{\tau c_1\|ju_n^k\|_Y^2}-c_2T.\label{eq:5.10}
		\end{align}
		We set $\varepsilon:= \frac{c_0}{2}$,
                $\alpha:=\frac{1}{2}\|{u}_0\|_H^2+
                c(\varepsilon)\|\boldsymbol{f}\|_{L^{p'}(I,X^*)}^{p'}+c_2T$,
                $\beta:=4\tau c_1<1$ and
                ${y^k_n:=\frac{1}{4}\|ju^k_n\|_Y^2}$ for $k=1,\dots,K$. 
Thus, we
                infer for every $l=1,\dots,K$ from \eqref{eq:5.8}, \eqref{eq:5.10} that
		\begin{align}
		y^l_n+\frac{c_0}{2}\int_0^{l\tau}{\|\overline{\boldsymbol{u}}^\tau_n(s)\|_X^p\,ds}\leq \alpha+\beta\sum_{k=1}^{l-1}{y^k_n}.\label{eq:5.11}
		\end{align}
		The discrete Gronwall inequality applied on \eqref{eq:5.11} yields
		\begin{align*}
		\frac 14\|\boldsymbol{j}\overline{\boldsymbol{u}}_n^\tau\|^2_{L^\infty(I,Y)}
                 + \frac {c_0}2 \|\overline{\boldsymbol{u}}_n^\tau\|^p_{L^p(I,X)}
                  \le \alpha\exp(K\beta)=\alpha\exp(4Tc_1)=:C_0,	
         \end{align*}
                which proves \eqref{eq:5.4}. From  the boundedness of $\mathbfcal{A}:L^p(I,X)\cap_{\boldsymbol{j}}L^q(I,Y)\to (L^p(I,X)\cap_{\boldsymbol{j}}L^q(I,Y))^*$ (cf.~Proposition \ref{3.9}) and \eqref{eq:5.4} we infer $\|\mathbfcal{A}\overline{\boldsymbol{u}}_n^\tau\|_{(L^p(I,X)\cap_{\boldsymbol{j}}L^q(I,Y))^*}\leq C_1$ for some $C_1>0$, i.e., \eqref{eq:5.6}.  In addition, it holds ${\|\boldsymbol{j}\hat{\boldsymbol{u}}_n^\tau\|^2_{L^\infty(I,Y)}\leq \|\boldsymbol{j}\overline{\boldsymbol{u}}_n^\tau\|^2_{L^\infty(I,Y)}\leq 4C_0}$ for every $n\in \mathbb{N}$ and $\tau\in(0,\tau_0)$, i.e., \eqref{eq:5.5}. Moreover, since $e_n\big(\hat{\boldsymbol{u}}_n^\tau(t)-\overline{\boldsymbol{u}}_n^\tau(t)\big)=(t-k\tau) d_\tau e_n \hat{\boldsymbol{u}}_n^\tau(t)=(t-k\tau)\frac{d_{e_n}\hat{\boldsymbol{u}}_n^\tau}{dt}(t)$ in $V_n^*$ and $\vert t-k\tau\vert \leq \tau$ for every $t\in I_k^\tau$, $k=1,\dots,K$, $K,n\in\mathbb{N}$, there holds for every $n\in\mathbb{N}$ and $\tau\in(0,\tau_0)$
		\begin{align*}
		\begin{split}
		\big\|e_n\big(\hat{\boldsymbol{u}}_n^\tau-\overline{\boldsymbol{u}}_n^\tau\big)\big\|_{L^{q'}(I,V_n^*)}&\leq\tau\left\|\frac{d_{e_n}\hat{\boldsymbol{u}}_n^\tau}{dt}\right\|_{L^{q'}(I,V_n^*)}\\&= \tau\big\|(\textup{id}_{L^q(I,V_n)})^*\big(\mathscr{J}_\tau[\boldsymbol{f}]-\!\mathscr{J}_\tau[\mathbfcal{A}]\overline{\boldsymbol{u}}_n^\tau\big)\big\|_{L^{q'}(I,V_n^*)}\leq \tau\big(\|\boldsymbol{f}\|_{L^{p'}(I,X^*)}\!+C_1\big),
		\end{split}
		\end{align*}
		i.e., the estimate \eqref{eq:5.7}, where we used Proposition \ref{4.4} (ii) and Proposition \ref{4.6} (iii.c). \hfill$\qed$
	\end{proof}

We can now prove the abstract convergence result, which is the
main result of this paper. 
	\begin{thm}
\label{5.17}
		Let Assumption \eqref{asum} be satisfied~and~set ${\tau_0\!:=\!\frac{1}{4c_1}}$.
		If $(\overline{\boldsymbol{u}}_n)_{n\in
                  \mathbb{N}}:=(\overline{\boldsymbol{u}}^{\tau_n}_{m_n})_{n\in
                  \mathbb{N}}\subseteq L^\infty(I,X)$, where
                $\tau_n=\frac{T}{K_n}$ and $K_n,m_n\to\infty$  $(n\to\infty)$, is
                an arbitrary diagonal sequence of  piece-wise
                constant interpolants
                $\overline{\boldsymbol{u}}_n^\tau\in
                \mathbfcal{S}^0(\mathcal{I}_{\tau},X)$, $K,n\in
                \mathbb{N}$ with $\tau=\frac TK\in \left(0,\tau_0\right)$, from
                Proposition~\ref{apriori}. Then, there exists a not relabelled subsequence and a weak limit  $\overline{\boldsymbol{u}}\in L^p(I,V)\cap_{\boldsymbol{j}} L^\infty(I,H)$ such that
		\begin{align*}
		\begin{alignedat}{2}
			\overline{\boldsymbol{u}}_n\;\;&\rightharpoonup\;\;\overline{\boldsymbol{u}}&&\quad\text{ in }L^p(I,X),\\
			\overline{\boldsymbol{u}}_n\;\;&\overset{\ast}{\rightharpoondown}\;\;\overline{\boldsymbol{u}}&&\quad\text{ in }L^\infty(I,Y),
			\end{alignedat}\begin{aligned}
			\qquad (n\to \infty).
			\end{aligned}
		\end{align*}
		Furthermore, it follows that
                $\overline{\boldsymbol{u}}\in
                \mathbfcal{W}^{1,p,q}_e(I,V,H)$ is a weak solution of the
                initial value~problem~\eqref{eq:1.1}.
	\end{thm}

	\begin{proof} We split the proof into four steps:\\[-3mm]
	
	\textbf{1. Convergences:} From the estimates \eqref{eq:5.4}--\eqref{eq:5.7}, the reflexivity of ${L^p(I,X)\cap_{\boldsymbol{j}} L^q(I,Y)}$, also using Proposition \ref{4.6} (iii.b), we obtain not relabelled
	subsequences $(\overline{\boldsymbol{u}}_n)_{n\in \mathbb{N}},(\hat{\boldsymbol{u}}_n)_{n\in \mathbb{N}}\subseteq
	L^p(I,X)\cap_{\boldsymbol{j}} L^\infty(I,Y)$, where  $\hat{\boldsymbol{u}}_n:=\hat{\boldsymbol{u}}_{m_n}^{\tau_n}$ for all $n\in \mathbb{N}$, as well as $\overline{\boldsymbol{u}}\in L^p(I,X)\cap_{\boldsymbol{j}} L^\infty(I,Y)$, $\boldsymbol{j}\hat{\boldsymbol{u}}\in L^\infty(I,Y)$ and
	${{\boldsymbol{\chi}}\in (L^p(I,X)\cap_{\boldsymbol{j}} L^q(I,Y))^*}$ such that
	\begin{align}
	\begin{alignedat}{2}
	\overline{\boldsymbol{u}}_n&\;\;\rightharpoonup\;\;\overline{\boldsymbol{u}}&\quad
	&\text{ in }L^p(I,X)\;\quad(n\to\infty),\\
	\boldsymbol{j}\overline{\boldsymbol{u}}_n&\;\;\overset{\ast}{\rightharpoondown}\;\;\boldsymbol{j}\overline{\boldsymbol{u}}&&\text{
		in }L^\infty(I,Y)\quad (n\rightarrow\infty),\\
	\boldsymbol{j}\hat{\boldsymbol{u}}_n&\;\; \overset{\ast}{\rightharpoondown}\;\;\boldsymbol{j}\hat{\boldsymbol{u}}&&\text{
		in }L^\infty(I,Y)\quad (n\rightarrow\infty),\\
	\mathscr{J}_{\tau_n}[\mathbfcal{A}]\overline{\boldsymbol{u}}_n&\;\;\rightharpoonup\;\;{\boldsymbol{\chi}}&&\text{
		in }(L^p(I,X)\cap_{\boldsymbol{j}} L^q(I,Y))^*\quad(n\to \infty),
              \\
              \mathbfcal{A}\overline{\boldsymbol{u}}_n&\;\;\rightharpoonup\;\;{\boldsymbol{\chi}}&&\text{
		in }(L^p(I,X)\cap_{\boldsymbol{j}} L^q(I,Y))^*\quad(n\to \infty).
	\end{alignedat}\label{eq:5.18}
	\end{align}
	From (\hyperlink{QNC.2}{QNC.2}) we immediately obtain that $\overline{\boldsymbol{u}}\in L^p(I,V)\cap_{\boldsymbol{j}} L^\infty(I,H)$.
	In particular, there exists $\boldsymbol{g}\in L^{p'}(I,V^*)+\boldsymbol{j}^*(L^{q'}(I,H^*))$ (cf.~\eqref{eq:dual}), such that for every ${\boldsymbol{v}\in L^p(I,V)\cap_{\boldsymbol{j}} L^q(I,H)}$ 
	\begin{align}
	\langle {\boldsymbol{\chi}},\boldsymbol{v}\rangle_{L^p(I,X)\cap_{\boldsymbol{j}} L^q(I,Y)}=\int_I{\langle \boldsymbol{g}(t),\boldsymbol{v}(t)\rangle_V\,dt}.\label{eq:rep}
	\end{align}
	 Due to $\eqref{eq:5.7}$ there exists a subset $E\subset I$, with $I\setminus E$ a null set, such that for every $t\in E$
	\begin{align}
		\|e_{m_n}(\hat{\boldsymbol{u}}_n(t)-\overline{\boldsymbol{u}}_n(t))\|_{V_{m_n}^*}\;\;\to \;\;0\quad(n\to\infty).\label{eq:5.19}
	\end{align}
	Owing to (\hyperlink{QNC.1}{QNC.1})  we can choose
        for every element $v$ of the dense subset $C\subseteq V$ a
        sequence $v_n\in V_{m_n}$,~${n\in\mathbb{N}}$, such that
        $v_n\to v$ in $X$ $(n\to\infty)$. Then, using the definition
        of $P_H$, \eqref{eq:iden}, \eqref{eq:5.4}, \eqref{eq:5.5} and \eqref{eq:5.19}, we infer for~every~${t\in E}$~that
	\begin{align}
	\begin{split}
		\vert(P_H[(\boldsymbol{j}&\hat{\boldsymbol{u}}_n)(t)-(\boldsymbol{j}\overline{\boldsymbol{u}}_n)(t)],jv)_H\vert= \vert((\boldsymbol{j}\hat{\boldsymbol{u}}_n)(t)-(\boldsymbol{j}\overline{\boldsymbol{u}}_n)(t),jv)_Y\vert \\&\leq \vert\langle e_{m_n}[(\boldsymbol{j}\hat{\boldsymbol{u}}_n)(t)-(\boldsymbol{j}\overline{\boldsymbol{u}}_n)(t)],v_n\rangle_{V_{m_n}}\vert+ \vert\left((\boldsymbol{j}\hat{\boldsymbol{u}}_n)(t)-(\boldsymbol{j}\overline{\boldsymbol{u}}_n)(t),jv-jv_n\right)_Y\vert
		\\&\leq 	\|e_{m_n}[\hat{\boldsymbol{u}}_n(t)-\overline{\boldsymbol{u}}_n(t)]\|_{V_{m_n}^*}\|v_n\|_X+ \|(\boldsymbol{j}\hat{\boldsymbol{u}}_n)(t)-(\boldsymbol{j}\overline{\boldsymbol{u}}_n)(t)\|_Y\|jv-jv_n\|_Y
		\\&\leq \|e_{m_n}[\hat{\boldsymbol{u}}_n(t)-\overline{\boldsymbol{u}}_n(t)]\|_{V_{m_n}^*}\|v_n\|_X+ 2M\|jv-jv_n\|_Y\;\;\to\;\; 0\quad(n\to\infty).
	\end{split}\label{eq:5.20}
	\end{align}
Since $C$ is dense in $V$ and $j(V)$ is dense in $H$, we conclude from \eqref{eq:5.20} for every $t\in E$ that
	\begin{align}
		P_H[(\boldsymbol{j}\hat{\boldsymbol{u}}_n)(t)-(\boldsymbol{j}\overline{\boldsymbol{u}}_n)(t)]\;\;\rightharpoonup \;\;0\quad\text{ in }H\quad(n\to\infty).\label{eq:5.21}
	\end{align}
Since the sequences
$(P_H\boldsymbol{j}\overline{\boldsymbol{u}}_n)_{n\in\mathbb{N}},(P_H\boldsymbol{j}\hat{\boldsymbol{u}})_{n\in\mathbb{N}}\subseteq
L^\infty(I,H)$ are bounded (cf. \eqref{eq:5.4} and \eqref{eq:5.5}), 
\cite[Proposition~2.15]{alex-rose-hirano} yields, due to \eqref{eq:5.21}, that
$P_H(\boldsymbol{j}\hat{\boldsymbol{u}}_n-\boldsymbol{j}\overline{\boldsymbol{u}}_n)\rightharpoonup
\boldsymbol 0$ in $L^q(I,H)$  $(n \to \infty)$. 
From \eqref{eq:5.18}$_{2,3}$ we easily deduce that
$P_H(\boldsymbol{j}\hat{\boldsymbol{u}}_n-\boldsymbol{j}\overline{\boldsymbol{u}}_n)\rightharpoonup
P_H(\boldsymbol{j}\hat{\boldsymbol{u}}-\boldsymbol{j}\overline{\boldsymbol{u}})$ in $L^q(I,H)$  $(n \to \infty)$.
Thus,
$P_H(\boldsymbol{j}\hat{\boldsymbol{u}})=P_H(\boldsymbol{j}\overline{\boldsymbol{u}})=\boldsymbol{j}\overline{\boldsymbol{u}}$
in $L^\infty (I,H)$, where we used that  $\overline{\boldsymbol{u}}\in L^p(I,V)\cap_{\boldsymbol{j}} L^\infty(I,H)$.\\[-3mm]

	\textbf{2. Regularity and trace of the weak limit:}
	\hypertarget{3.2}{} 
	Let $v\in C$ and  $v_n\in V_{m_n}$, $n\in\mathbb{N}$, be a sequence such 
that $v_n\to v$ in $X$ $(n\to\infty)$. Testing
	\eqref{eq:4.15} for $n\in \mathbb{N}$ by
	$v_n\in V_{m_n}$, multiplication by $\varphi\in C^\infty(\overline{I})$ with $\varphi(T)=0$, integration over $I$, and integration by parts yields for every $n\in \mathbb{N}$
	\begin{align}
	\begin{split}
	\langle \mathscr{J}_{\tau_n}[\mathbfcal{A}]\overline{\boldsymbol{u}}_n,v_n\varphi\rangle_{L^p(I,X)\cap_{\boldsymbol{j}}L^q(I,Y)}&-\int_{I}{\langle\mathscr{J}_{\tau_n}[\boldsymbol{f}](s),v_n\rangle_{X}\varphi(s)\,ds} \\&=\int_I{((\boldsymbol{j}\hat{\boldsymbol{u}}_n)(s),jv_n)_Y\varphi^\prime(s)\,ds}
	+(u_{m_n}^0,jv_n)_Y\varphi(0).
	\end{split}\label{eq:5.22}
	\end{align}
	By passing in \eqref{eq:5.22} for $n\to \infty$, using
        \eqref{eq:5.18}, \eqref{eq:rep}, Proposition \ref{4.4} (i), $P_H(\boldsymbol{j}\hat{\boldsymbol{u}})=\boldsymbol{j}\overline{\boldsymbol{u}}$ in $L^{\infty}(I,H)$,
	$u_{m_n}^0\to {u}_0$ in
	$Y$ $(n\to\infty)$ and the density of $C$ in $V$, we obtain
        that for all $v\in V$ and
	$\varphi\in C^\infty(\overline{I})$ with $\varphi(T)=0$ there holds
	\begin{align}
	\begin{split}
	\int_{I}{\langle\boldsymbol{g}(s)-\boldsymbol{f}(s),v\rangle_V\varphi(s)\,ds}
        &=\int_I{((\boldsymbol{j}\hat{\boldsymbol{u}})(s),jv)_Y\varphi^\prime(s)\,ds}
	+({u}_0,jv)_Y\varphi(0)
	\\&=\int_I{((\boldsymbol{j}\overline{\boldsymbol{u}})(s),jv)_H\varphi^\prime(s)\,ds}
	+({u}_0,jv)_H\varphi(0).
	\end{split}\label{eq:5.23} 
	\end{align}
	In the case
	$\varphi\in C_0^\infty(I)$ in \eqref{eq:5.23}, recalling Definition \ref{2.15}, we conclude that $\overline{\boldsymbol{u}}\in \mathbfcal{W}_e^{1,p,q}(I,V,H)$ with continuous representation $\boldsymbol{j}_c\overline{\boldsymbol{u}}\in C^0(\overline{I},H)$ and 
	\begin{align}
	\frac{d_e\overline{\boldsymbol{u}}}{dt}=\boldsymbol{f}-\boldsymbol{g}\quad\text{ in }
	L^{p'}(I,V^*)+ \boldsymbol{j}^*(L^{q'}(I,H^*)).\label{eq:5.24}
	\end{align}
	Thus, we are able to apply the generalized integration by parts formula in $\mathbfcal{W}_e^{1,p,q}(I,V,H)$ (cf.~Proposition~\ref{2.16}) in
	\eqref{eq:5.23} in the case $\varphi\in C^\infty(\overline{I})$ with
	$\varphi(T)=0$ and $\varphi(0)=1$, which yields for all $v\in
	V$ 
	\begin{align*}
	((\boldsymbol{j}_c\overline{\boldsymbol{u}})(0)-{u}_0,jv)_H=0. 
	\end{align*}
	As $j(V)$ is dense in $H$ and $(\boldsymbol{j}_c\overline{\boldsymbol{u}})(0)\in H$, we deduce from \eqref{eq:5.24} that $(\boldsymbol{j}_c\overline{\boldsymbol{u}})(0)={u}_0$ in $H$. \\[-3mm]

	\textbf{3. Pointwise weak convergence:}
	Next, we show that
	$ P_H(\boldsymbol{j}\hat{\boldsymbol{u}}_n)(t)\rightharpoonup (\boldsymbol{j}_c
	\overline{\boldsymbol{u}})(t)$ in $H$ ${(n\to\infty)}$ for all $t \in \overline{I}$, which due to \eqref{eq:5.21} in turn yields
	that $ P_H(\boldsymbol{j}\overline{\boldsymbol{u}}_n)(t)
	\rightharpoonup (\boldsymbol j
	\overline{\boldsymbol{u}})(t)$ in $H$ $(n\to\infty)$ for almost every $t 
\in \overline{I}$. To this end, let us fix an arbitrary
	$t\in I$. From the a-priori estimate
	$\|(\boldsymbol{j}\hat{\boldsymbol{u}}_n)(t)\|_{Y}\leq M$ for all
	$t\in \overline{I}$ and $n\in\mathbb{N}$ (cf.~\eqref{eq:5.5}) we
	obtain a subsequence
	$((\boldsymbol{j}\hat{\boldsymbol{u}}_n)(t))_{n\in\Lambda_t}\subseteq
	Y$ with $\Lambda_t\subseteq\mathbb{N}$, initially
	depending on this fixed $t$, and an element
	$\hat{\boldsymbol{u}}_{\Lambda_t}\in Y$ such that
	\begin{align}
	(\boldsymbol{j}\hat{\boldsymbol{u}}_n)(t)
	\;\;\rightharpoonup\;\;
	\hat{\boldsymbol{u}}_{\Lambda_t}\quad\text{ in }Y\quad(
	\Lambda_t\ni n\to \infty).\label{eq:5.27}  
	\end{align}
	Let $v\in C$ and $v_n\in V_{m_n}$, $n\in\mathbb{N}$, be such that $v_n\to v$ in $X$ $(n\to\infty)$. Then, we test \eqref{eq:4.15} for $n\in \Lambda_t$ by
	$v_n\in V_{m_n}$, multiply by $\varphi\in C^\infty(\overline{I})$ with $\varphi(0)=0$ and
	$\varphi(t)=1$, integrate over $\left[0,t\right]$ and integrate by parts, to obtain for all $n\in \Lambda_t$ 
	\begin{align}
          \begin{split}
            &\big \langle
            \mathscr{J}_{\tau_n}[\mathbfcal{A}]\overline{\boldsymbol{u}}_n,v_n\varphi\chi_{\left[0,t\right]}
            \big \rangle_{L^p(I,X)\cap_{\boldsymbol{j}}L^q(I,Y)}-\int_0^t{\big \langle\mathscr{J}_{\tau_n}[\boldsymbol{f}](s),v_n \big \rangle_{X}\varphi(s)\,ds} \label{eq:5.28}
            \\
            &=\int_0^t{((\boldsymbol{j}\hat{\boldsymbol{u}}_n)(s),jv_n)_Y\varphi^\prime(s)\,ds}
            -((\boldsymbol{j}\hat{\boldsymbol{u}}_n)(t),jv_n)_Y.
          \end{split}
	\end{align}
	By passing in \eqref{eq:5.28} for  $n\in \Lambda_t$ to
	infinity, using \eqref{eq:5.18}, \eqref{eq:rep}, Proposition \ref{4.4} (i), \eqref{eq:5.27} and the density of $C$ in $V$,~we~obtain for all $v\in V$
	\begin{align}
	\begin{split}	\int_0^t{\langle\boldsymbol{g}(s)-\boldsymbol{f}(s),v\rangle_V\varphi(s)\,ds}
=\int_0^t{((\boldsymbol{j}\overline{\boldsymbol{u}})(s),jv)_H\varphi^\prime(s)\,ds}
-(\hat{\boldsymbol{u}}_{\Lambda_t},jv)_Y.
\end{split} \label{eq:5.29b}
	\end{align}
	From \eqref{eq:5.24},
	\eqref{eq:5.29b}, the integration by parts formula in
        $\mathbfcal{W}_e^{1,p,q}(I,V,H)$ and the properties of $P_H$ we 
	obtain 
	\begin{align}
	0=((\boldsymbol{j}_c\overline{\boldsymbol{u}})(t)-
	\hat{\boldsymbol{u}}_{\Lambda_t},jv)_Y=((\boldsymbol{j}_c\overline{\boldsymbol{u}})(t)-
	P_H\hat{\boldsymbol{u}}_{\Lambda_t},jv)_H,\label{eq:5.29} 
	\end{align}
	for all $v\in V$. Thanks to the density of $j(V)$ in $H$, \eqref{eq:5.29} yields
	$(\boldsymbol{j}_c\overline{\boldsymbol{u}})(t)=P_H\hat{\boldsymbol{u}}_{\Lambda_t}$
	in $H$, i.e.,
	\begin{align}
	P_H(\boldsymbol{j}\hat{\boldsymbol{u}}_n)(t)\;\;
	\rightharpoonup\;\;(\boldsymbol{j}_c\overline{\boldsymbol{u}})(t)\quad\text{
		in }H\quad(\Lambda_t\ni n \to \infty).\label{eq:5.30} 
	\end{align}
	As this argumentation remains valid for each subsequence of
	${(P_H(\boldsymbol{j}\hat{\boldsymbol{u}}_n)(t) )_{n\in\mathbb{N}}\subseteq H}$,
	the element ${(\boldsymbol{j}_c\overline{\boldsymbol{u}})(t)\in H}$ is a weak accumulation point of
	each subsequence of
	$(P_H(\boldsymbol{j}\hat{\boldsymbol{u}}_n)(t) )_{n\in\mathbb{N}}\subseteq
	H$. The standard convergence principle (cf.~\cite[Kap. I, Lemma 5.4]{GGZ}) yields
	$\Lambda_t=\mathbb{N}$ in \eqref{eq:5.30}. Therefore, using \eqref{eq:5.21} and that $(\boldsymbol{j}_c\overline{\boldsymbol{u}})(t)=(\boldsymbol{j}\boldsymbol{u})(t)$ in $H$ for almost every $t\in I$, there holds for almost every $t\in I$ 
	\begin{align}\label{eq:5.31}
	P_H(\boldsymbol{j}\overline{\boldsymbol{u}}_n)(t)\;\;\rightharpoonup\;\;(\boldsymbol{j}\overline{\boldsymbol{u}})(t)\quad\text{
		in }H\quad(n\to\infty).\\[-3mm]\notag
	\end{align}
	
	\textbf{4. Identification of
          $\mathbfcal{A}\overline{\boldsymbol{u}}$ and
          ${\boldsymbol{\chi}}$}: Inequality \eqref{eq:5.8} in
        the case $\tau=\tau_n$, $n=m_n$ and $l=K_n$, using Proposition
        \ref{4.6} (iii.a),
        $(\boldsymbol{j}_c\overline{\boldsymbol{u}})(0)\!=\!{u}_0$ in
        $H$,
        ${\|P_H(\boldsymbol{j}\hat{\boldsymbol{u}}_n)(T)\|_H\!\leq\!\|(\boldsymbol{j}\hat{\boldsymbol{u}}_n)(T)\|_Y\!=\!\|(\boldsymbol{j}\overline{\boldsymbol{u}}_n)(T)\|_Y}$
        and
        $\langle\mathscr{J}_{\tau_n}[\boldsymbol{f}],\overline{\boldsymbol{u}}_n\rangle_{L^p(I,X)}=\langle\boldsymbol{f},\overline{\boldsymbol{u}}_n\rangle_{L^p(I,X)}$
        for all $n\in \mathbb{N}$, yields for all $n\in \mathbb{N}$
	\begin{align}
	\langle
	\mathbfcal{A}\overline{\boldsymbol{u}}_n,\overline{\boldsymbol{u}}_n\rangle_{L^p(I,X)\cap_{\boldsymbol{j}}L^q(I,Y)}\leq 
	-\frac{1}{2}\|P_H(\boldsymbol{j}\hat{\boldsymbol{u}}_n)(T)\|_H^2+\frac{1}{2}\|(\boldsymbol{j}_c\overline{\boldsymbol{u}})(0)\|_H^2+\langle\boldsymbol{f},\overline{\boldsymbol{u}}_n\rangle_{L^p(I,X)}.\label{eq:5.32}
	\end{align}
	Thus, the limit superior with respect to $n\in\mathbb{N}$ on both sides in \eqref{eq:5.32},
	\eqref{eq:5.18}, \eqref{eq:rep}, \eqref{eq:5.30} with
	$\Lambda_t=\mathbb{N}$ in the case $t=T$, the weak lower
	semi-continuity of $\|\cdot\|_H$, the integration by parts
	formula in $\mathbfcal{W}_e^{1,p,q}(I,V,H)$ and \eqref{eq:5.24} yield
	\begin{align}\begin{split}
	\limsup_{n\rightarrow\infty}{ \langle
		\mathbfcal{A}\overline{\boldsymbol{u}}_n,\overline{\boldsymbol{u}}_n-\overline{\boldsymbol{u}}\rangle_{L^p(I,X)\cap_{\boldsymbol{j}}L^q(I,Y)}}
	&\leq
	-\frac{1}{2}\|(\boldsymbol{j}_c\overline{\boldsymbol{u}})(T)\|_H^2+\frac{1}{2}\|(\boldsymbol{j}_c\overline{\boldsymbol{u}})(0)\|_H^2
        \\
        &\qquad+
	\int_I{\langle\boldsymbol{f}(t)-\boldsymbol{g}(t),\overline{\boldsymbol{u}}(t)\rangle_V\,dt}\\&=-\int_I{\Big\langle
	\frac{d_e\overline{\boldsymbol{u}}}{dt}(t)+\boldsymbol{f}(s)-\boldsymbol{g}(s),\overline{\boldsymbol{u}}(t)\Big\rangle_V\,dt}=0.
	\end{split}\label{eq:5.33}
	\end{align}
	As a result of \eqref{eq:5.18}, \eqref{eq:5.31},
	\eqref{eq:5.33} and the quasi non-conforming Bochner pseudo-monotonicity 
of
	$\mathbfcal{A}:L^p(I,X)\cap_{\boldsymbol{j}}L^q(I,Y)\rightarrow
	({L^p(I,X)\cap_{\boldsymbol{j}}L^q(I,Y)})^*$ (cf.~Proposition~\ref{3.9}), there holds 
	\begin{align*}
		\langle \mathbfcal{A}\overline{\boldsymbol{u}},\overline{\boldsymbol{u}}-\boldsymbol{v}\rangle_{L^p(I,X)\cap_{\boldsymbol{j}}L^q(I,Y)}&\leq\liminf_{n\to\infty}{\langle \mathbfcal{A}\overline{\boldsymbol{u}}_n,\overline{\boldsymbol{u}}_n-\boldsymbol{v}\rangle_{{L^p(I,X)\cap_{\boldsymbol{j}}L^q(I,Y)}}}\\&\leq\langle{\boldsymbol{\chi}},\overline{\boldsymbol{u}}-\boldsymbol{v}\rangle_{{L^p(I,X)\cap_{\boldsymbol{j}}L^q(I,Y)}},
	\end{align*}
	for any $\boldsymbol{v}\in {L^p(I,X)\cap_{\boldsymbol{j}}L^q(I,Y)}$, which in turn implies that $\mathbfcal{A}\overline{\boldsymbol{u}}={\boldsymbol{\chi}}$ in $({L^p(I,X)\cap_{\boldsymbol{j}}L^q(I,Y)})^*$. 
	This completes the proof of Theorem \ref{5.17}.\hfill$\qed$
	\end{proof}

	\section{Applications} 
        \label{sec:7}

        In this section, we apply the abstract theory developed in the
        previous sections to two problems stemming from incompressible
        non-Newtonian fluid flows. In particular, we treat the motion
        of micropolar electrorheological fluids and a variant of the
        Smagorinsky model in turbulence. In both cases we show that
        solutions of a fully discrete implicit Rothe--Galerkin scheme
        converge to a weak solution of the corresponding
        problem. These results are new and, to the best of the
        authors' knowledge, can not be found in the literature. We restrict
        our treatment to the three-dimensional case. All
        results have an analogue in the two-dimensional setting (for a
        formulation of \eqref{eq:erf} in two dimensions
        see~\eqref{eq:erf-2d}  in Section~\ref{sec:8}).

In this section we always assume that we have given a family of shape regular triangulations
(cf.~\cite{BS08}) $(\mathcal{T}_h)_{h>0}$ of a bounded polygonal Lipschitz
domain $\Omega \subset \mathbb R^3$, and that $I=(0,T)$ is a finite time interval.

In order to formulate the problems we need some additional notation.
We denote by $\levy$ the
  anti-symmetric Levi--Civita symbol. 
  For a tensor $P$ and a vector $\omega$, resp., we denote by
  $\levy : P$ the vector having the components
  $\varepsilon_{ijk}P_{jk}$, $i=1,2,3$ and by $\levy\cdot \omega $ the
  tensor with components $\varepsilon_{ijk}\omega_k$, $i,j=1,2,3$,
  respectively.  In both cases the summation convention over repeated
  indices is used. For a tensor $P$ and a vector $\omega$ the symbol
  ${P} \,\omega \in \mathbb{R}^3$ denotes the matrix-vector product, i.e.,
  $({P}\, \omega)_i={{P}_{ij} \omega_j}$ for $i=1,2,3$.
  $\mathbb{M}^{3\times 3}_{\text{sym}}$ is the
  vector space of all symmetric $3\times 3$ tensors ${P} $ and
  $\mathbb{M}^{3\times 3}_{\text{skew}}$ is the
  vector space of all skew-symmetric $3\times 3$ tensors ${P} $.
  We equip the vector space $\mathbb{M}^{3\times 3}$ of all $3\times 3$ tensors ${P} $ with the scalar product
  ${P}:{Q}:={P_{ij}Q_{ij}}$ and the norm
  $\vert{P}\vert:=({P}:{P})^{\frac{1}{2}}$.
  Moreover, we denote the symmetric and the
  skew-symmetric part, resp., of a tensor ${P} \in
  \mathbb{M}^{3\times 3}$ by
  ${{P}^{\sym}:=\frac{1}{2}({P}+{P}^\top)}$ and
  ${P}^{\anti}:=\frac{1}{2}({P} -{P}^\top)$,
  respectively. The particular case of the 
  skew-symmetric part of the velocity gradient is denoted by
  ${W}u:=\frac{1}{2}(\nabla u -\nabla u^\top)$.

\subsection{Micropolar electrorheological fluids}
We consider the following system describing the motion of micropolar
electrorheological fluids 
\begin{align}
  \begin{split}
    \begin{alignedat}{2}
      \partial_t \boldsymbol{u}-\divo \bS+\text{div}(\boldsymbol{u}\otimes
      \boldsymbol{u}) +\nabla\boldsymbol{q}&=\boldsymbol{f}&&\quad\text{ in }I\times\Omega,
      \\
      \divo \boldsymbol{u}&=0&&\quad\text{ in }I\times\Omega,
      \\
      \partial_t \boldsymbol \omega -\divo \bN +\text{div}(\boldsymbol{\omega}\otimes
      \boldsymbol{u}) &=\boldsymbol{\ell} - \levy:\bS &&\quad\text{ in }I\times\Omega,
      \\
      \boldsymbol{u}=\mathbf{0},\qquad \boldsymbol{\omega}&=\boldsymbol{0}&&\quad\text{ on }I\times \partial\Omega,
      \\
  \boldsymbol{u}(0)={u}_0,\qquad           \boldsymbol{\omega}(0)&={\omega}_0&&\quad\text{ in }\Omega.
    \end{alignedat}
  \end{split}\label{eq:erf}
\end{align}
In these equations $\boldsymbol u$ denotes the velocity, $\boldsymbol \omega$ the
micro-rotation, $\boldsymbol{q}$ the pressure, $\bS$ the mechanical extra stress
tensor, $\bN$ the couple stress tensor, $\boldsymbol \ell$ the electromagnetic
couple force, $\boldsymbol f=\,\hat {\!\boldsymbol {f}} + \chi^E\divo
(\boldsymbol E \otimes \boldsymbol E)$ the
body force, where $\hat {\!\boldsymbol {f}}$ is the mechanical body force, $\chi^E$
the dielectric susceptibility and $\boldsymbol E$ the electric
field. The electric field solves the quasi-static Maxwell's equations
\begin{gather}
   \begin{gathered}\label{maxwell}
     \divo \boldsymbol{E}=0 \qquad    \curl \boldsymbol {E}=\mathbf 0 
\qquad \text{in}\
     I\times \Omega\,,
     \\
     \boldsymbol {E}\cdot \boldsymbol{n}=\boldsymbol{E}_0\cdot \boldsymbol{n}
     \qquad\text{on}\ I \times \partial\Omega,
   \end{gathered}
\end{gather}
where $\boldsymbol{n}$ is the outer unit normal vector of the boundary
$\partial \Omega$ and $\boldsymbol{E}_0$ is a given time-dependent
electric field. 

This model was developed in \cite{win-r} to obtain a more realistic
description for the motion of electrorheological fluids.  A
representative example for a constitutive relation for the stress
tensors in \eqref{eq:erf} reads (cf.~\cite{win-r},
\cite{rubo}) 
\begin{align}
  \hspace*{-1mm}  
  \begin{aligned}\label{eq:SN-ex}
    \textrm{S}&=(\alpha_{31}+\alpha_{33}\vert {E}\vert^2)
    (1+\vert{D}\vert)^{p-2}{D}+ \alpha_{51}
    (1+\vert{D}\vert)^{p-2}\big ({D} \bE \otimes \bE +
    \bE \otimes D\bE \big)\hspace*{-5mm}
    \\
    &\quad +
    \alpha_{71}\vert{E}\vert^2(1+\vert{R}\vert)^{p-2}
    {R} + \alpha_{91} (1+\vert{R}\vert)^{p-2}\big
    ({R} \bE \otimes \bE + \bE \otimes R\bE \big)\,,
    \\
    {N}&=(\beta_{31}+\beta_{33}\vert{E}\vert^2)
    (1+\vert L\vert)^{p-2}L
    + \beta_{51}(1+\vert L\vert)^{p-2}\big
    (L
    \bE \otimes \bE + \bE \otimes L\bE \big)\,,
  \end{aligned}
\end{align}
with {$p \in (1,\infty)$} and constants
$\alpha_{31},\alpha_{33},\alpha_{71},\beta_{31},\beta_{33}\ge 0$.  The
constants $\alpha_{51}, \alpha_{91}, \beta_{51}$ have to satisfy
certain restrictions (cf.~\cite{win-r}, \cite{rubo}), which ensure the
validity of the second law of thermodynamics. In \eqref{eq:SN-ex} we
used the notation 
$D = \bD u $, $R=R( u, \omega):= W u +\levy \cdot \omega $ and
$L =\nabla \omega$.

This model for micropolar electrorheological fluids is rather general
and contains as special cases the models for generalized Newtonian fluids 
($\bE
=0$, $\alpha_{31}=0$), electrorheological fluids with constant
exponents
($\alpha_{71}=\alpha_{91}=\beta_{31}=\beta_{33}=\beta_{51}=0$) and
micropolar fluids ($\bE =\textrm{const.}$,
$\alpha_{51}=\alpha_{91}=\beta_{33}=\beta_{51}=0$). Consequently the
results presented in this section also apply to these models either
directly or by an easy adaptation.
  
We refrain from considering concrete constitutive relations for the
stress tensors, but we make general assumptions covering prototypical
situations: 

The continuous mapping
$\textrm{S}:\mathbb{M}^{3\times
  3}_{\text{sym}}\times\mathbb{M}^{3\times 3}_{\text{skew}}\times
\setR^3 \to \mathbb{M}^{3\times 3}$ satisfies for some
$p\in \left(1,\infty\right)$ and $\kappa \ge 0$ and all 
$D,{P}\in \mathbb{M}^{3\times 3}_{\text{sym}}$,
$R,{Q}\in \mathbb{M}^{3\times 3}_{\text{skew}}$,
$E \in \setR^3$ the following properties:
\begin{description}[{\textbf{(S.3)})}]
\item[\textbf{(S.1)}]\hypertarget{S.1}{}$\vert\textrm{S}^\sym ({D},\bR,\bE)\vert \leq
  \alpha_0(1+\abs{\bE}^2)\,(\kappa+\vert
  {D}\vert)^{p-2}\vert{D}\vert+\alpha_1$ for
  $(\alpha_0>0,\,\alpha_1\ge 0)$; \newline
  $\vert\textrm{S}^\anti ({D},\bR,\bE)\vert \leq
  \beta _0 \abs{\bE}^2(\kappa+\vert
  {R}\vert)^{p-2}\vert{R}\vert+\beta_1$ for
  $(\beta_0>0,\,\beta_1\ge 0)$.\vspace*{1mm}
\item[\textbf{(S.2)}]\hypertarget{S.2}{}$\textrm{S}({D},\bR,\bE):{D}\ge
  c_0(1+\abs{\bE}^2)\,(\kappa+\vert {D}\vert)^{p-2}\vert{D}\vert^2-c_1$
  for ${(c_0>0,\,c_1\ge 0)}$; \newline
  $\textrm{S}(D,{R},\bE):{R}\ge
  c_2\abs{\bE}^2(\kappa+\vert {R}\vert)^{p-2}\vert{R}\vert^2-c_3$
  for ${(c_2>0,\,c_3\ge 0)}$.\vspace*{1mm}
\item[\textbf{(S.3)}]\hypertarget{S.3}{}
  $(\textrm{S}({D},\bR,\bE)-\textrm{S}({P},  Q,
  \bE)): (D- {P} + \bR-{Q})\ge 0$.
\end{description}
The continuous mapping
$\textrm{N}:\mathbb{M}^{3\times
  3}\times \setR^3 \to \mathbb{M}^{3\times 3}$ satisfies for the same
$p\in \left(1,\infty\right)$ and $\kappa \ge 0$ and all 
${L},  K \in \mathbb{M}^{3\times 3}$,
$\bE \in \setR^3$ the following properties:
\begin{description}[{\textbf{(N.3)})}]
\item[\textbf{(N.1)}]\hypertarget{N.1}{}$\vert\textrm{N}({L},\bE)\vert \leq
  \gamma_0(1+\abs{\bE}^2)\,(\kappa+\vert
  {L}\vert)^{p-2}\vert{L}\vert+\gamma_1$ for
  $(\gamma_0>0,\,\gamma_1\ge 0)$.\vspace*{1mm}
\item[\textbf{(N.2)}]\hypertarget{N.2}{}$\textrm{N}({L},\bE):{L}\ge
  c_3(1+\abs{\bE}^2)\,(\kappa+\vert {L}\vert)^{p-2}\vert{L}\vert^2-c_4$
  for ${(c_3>0,\,c_4\ge 0)}$.\vspace*{1mm}
\item[\textbf{(N.3)}]\hypertarget{N.3}{}
  $(\textrm{N}({L},\bE)-\textrm{N}({K},
  \bE)): ({L} -{K})\ge 0$.
\end{description}
Concerning the electric field $\boldsymbol E$ solving \eqref{maxwell} we assume
that the boundary data $\bE_0$ are regular enough to ensure that
\begin{description}[{\textbf{(E.1)})}]
\item[\textbf{(E.1)}]\hypertarget{E.1}{} $\boldsymbol E$ belongs to the space
  $L^\infty(I,L^\infty(\Omega))$ and a.e.~in $I\times\Omega$ there holds
  $\abs{\boldsymbol E}>0$.
\end{description}

To treat problem \eqref{eq:erf} we define for
$p>\frac{6}5$ the function spaces
\begin{align*}
  \begin{alignedat}{2}
    X&:=W^{1,p}_0(\Omega)^3 \times W^{1,p}_0(\Omega)^3,
    &\qquad Y&:=L^2(\Omega)^3\times L^2(\Omega)^3,
    \\
    V&:=W^{1,p}_{0,\divo }(\Omega)\times W^{1,p}_0(\Omega)^3,
    &H&:=L^2_{\divo }(\Omega)\times L^2(\Omega)^3,
  \end{alignedat}
\end{align*}
and the families of operators ${S(t),N(t),B:X\to X^*}$
for all ${(u,\omega) ^\top}, (v,\eta) ^\top\in X$ via\footnote{To formulate the
  operator $S$ we used that $  \bS:\nabla u+(\levy:\bS)\cdot \omega
  =\bS:\big(\bD u +  \bR(u ,\omega) \big)$.
}
\begin{gather*}
  \begin{aligned}
    \langle S(t)({u},{\omega}) ^\top, (v,\eta) ^\top\rangle_X&:= \int_\Omega \bS(\bD
    u,\bR(u,\omega),\boldsymbol E(t)) :(\bD v + \bR (v,\eta) )\,dx\,,
    \\
    \langle N(t)({u},{\omega}) ^\top, (v,\eta) ^\top \rangle_X&:= \int_\Omega
    \bN(\nabla \omega,\boldsymbol E(t)) :\nabla \eta\,dx\,,
    \\
    \langle B({u},{\omega}) ^\top, (v,\eta) ^\top \rangle_X&:= \int_\Omega
    u \otimes u : \nabla v  + \omega \otimes u : \nabla \eta \,dx\,,
  \end{aligned}
\end{gather*}
and set $A(t):=S(t) +N(t) +B:X\to X^*$, $t \in I$.
Then, \eqref{eq:erf} for $\mathcal U_0:=({u}_0, \omega_0) ^\top\in H$ and
$\mathbfcal F :=(\boldsymbol f, \boldsymbol \ell) ^\top \in L^{p'}(I,X^*)$ can be re-written as the abstract 
evolution equation for $\mathbfcal U:=(\boldsymbol u, \boldsymbol \omega) ^\top$
\begin{align*}
  \begin{split}
    \begin{alignedat}{2}
      \frac{d\,\mathbfcal{U}}{dt}(t)+A(t)(\mathbfcal{U}(t))&=\mathbfcal{F}(t)&&\quad\text{
        in }V^*,
      \\
      \mathbfcal{U}(0)&=\mathcal {U}_0&&\quad\text{ in }H.
    \end{alignedat}
  \end{split}
\end{align*}
In \cite{br-mperf} it is shown that, under appropriate assumption on
the data, there exists a
weak solution of the problem \eqref{eq:erf} for  $p>\frac 65$ provided
$\bS$ satisfies (\hyperlink{S.1}{S.1})--(\hyperlink{S.3}{S.3}), $\bN$
satisfies (\hyperlink{N.1}{N.1})--(\hyperlink{N.3}{N.3}), and $\bE$
satisfies (\hyperlink{E.1}{E.1}). 

Due to the presence of $B$ in the definition of the operator family
$A(t):X\to X^*$, $t \in I$, the condition  (\hyperlink{A.3}{A.3}) is 
not satisfied. Thus, we modify this family and define the
operator family $\hat{A}(t):X\to X^*$ via
$\hat A(t):=S(t) +N(t) +\hat B$, $t \in I$, where $\hat B$ is given 
for all ${(u,\omega) ^\top}, (v,\eta) ^\top\in X$ via
\begin{gather*}
  \langle \hat B({u},{\omega}) ^\top, (v,\eta) ^\top \rangle_X:= \frac 12 \int_\Omega v
  \otimes u : \nabla u - {{u}\otimes
  {u}:\nabla {v}} +{{\eta}\otimes
  {u}:\nabla {\omega}}- \omega \otimes u : \nabla \eta \,dx\,.
\end{gather*}
The operator $\hat{B}$ is a symmetrized extension of $B$, as
$\langle\hat{B}({u},\omega) ^\top,({v},\eta) ^\top\rangle_X=\langle B({u},\omega) ^\top,({v},\eta) ^\top\rangle_X$ for all
${(u,\omega) ^\top}$, $(v,\eta) ^\top\in X$, which in contrast to $B$ fulfils
$\langle\hat{B}({u},\omega) ^\top,({u},\omega) ^\top\rangle_X=0$ for all
$({u},\omega) ^\top\in X$. Thus, we have the following result:
\begin{prop}\label{4.7} For $p>\frac{11}{5}$ the operator family
  $\hat A(t):X\to X^*$, $t\in I$, satisfies
  (\hyperlink{A.1}{A.1})--(\hyperlink{A.4}{A.4}).
\end{prop}

\begin{proof}
  The assertion is essentially proved in \cite{br-mperf}. For the sake
  of completeness we sketch the main arguments. Let us first consider
  $S(t)+N(t):X\to X^*$, $t\in I$. From (\hyperlink{S.1}{S.1}),
  (\hyperlink{N.1}{N.1}), (\hyperlink{S.2}{S.2}), (\hyperlink{N.2}{N.2}) and
  (\hyperlink{E.1}{E.1}) in conjunction with the standard theory of
  Nemytski\u{\i} operators {(cf.~\cite[Theorem 1.43]{Rou05})} we
  deduce for almost every $t\in I$ the well-definiteness and
  continuity of $S(t)+N(t):X\to X^*$, as well as  condition
  (\hyperlink{A.3}{A.3}). Condition (\hyperlink{A.2}{A.2}) follows in
  a standard way by using (\hyperlink{S.1}{S.1}),
  (\hyperlink{N.1}{N.1}), 
  (\hyperlink{E.1}{E.1}), Pettis' and Fubini's theorem, since $X$ is
  separable.  (\hyperlink{S.3}{S.3})  and (\hyperlink{N.3}{N.3}) certainly imply for
  almost every $t\in I$ the monotonicity of $S(t)+N(t):X\to X^*$. Hence,
  $S(t)+N(t):X\to X^*$ is for almost every $t\in I$ pseudo-monotone, i.e.,
  condition (\hyperlink{A.1}{A.1}) is satisfied.  Eventually, it can
  readily be seen by exploiting (\hyperlink{S.1}{S.1}) and (\hyperlink{N.1}{N.1}) that
  $S(t):X\to X^*$, $t\in I$, satisfies (\hyperlink{A.4}{A.4}).
		
  Next, we treat the more delicate part $\hat{B}:X\to X^*$. One can verify
  by the standard theory of Nemytski\u{\i} operators and Rellich's
  compactness theorem that $\hat{B}:X\to X^*$ is bounded and
  pseudo-monotone, i.e., satisfies (\hyperlink{A.1}{A.1}) and
  (\hyperlink{A.2}{A.2}). As already pointed out above we have 
  $\langle \hat{B}({u},\omega),({u},\omega)\rangle_X=0$ for all
  $({u},\omega)\in X$, i.e., $\hat{B}$ satisfies (\hyperlink{A.3}{A.3}).
  To verify that $\hat B$ satisfies (\hyperlink{A.4}{A.4}) we note
  that it is sufficient to treat the last two terms, since the same
  estimates apply for the first two if one replaces $\omega$ and $\eta$
  by $u$ and $v$, respectively. Thus, we define for fixed $u \in
  W^{1,p}_0(\Omega)^3$ the operator  $\tilde B$ via $\langle
  \tilde{B}\omega,\eta\rangle_{W^{1,p}_0} := \frac 12 \int_\Omega {{\eta}\otimes
  {u}:\nabla {\omega}}- \omega \otimes u : \nabla \eta \,dx$. By Hölder's inequality there
  holds for every $\eta,\omega \in W^{1,p}_0(\Omega)^3$ 
  \begin{align}
    \vert\langle \tilde{B} \omega,\eta\rangle\vert\leq
    \|{u}\|_{L^{2p'}}    \|{\eta}\|_{L^{2p'}}\|\omega\|_{W^{1,p}_0}+
    \|{\omega}\|_{L^{2p'}}    \|{u}\|_{L^{2p'}}\|\eta\|_{W^{1,p}_0}\,. \label{eq:4.7.1} 
  \end{align}
  For $p\ge 3$, the Sobolev embedding provides $\|v\|_{L^{2p'}}\leq c\|v\|_{W^{1,p}_0}$ for every
  ${v}\in  W^{1,p}_0(\Omega)^3$.  Thus, a twofold application of the $\varepsilon$--Young inequality 
  yields for every
  $\varepsilon>0$
  \begin{align*}
    \vert\langle \tilde{B} \omega,\eta\rangle \vert&\leq
    c\,     \|{u}\|_{W^{1,p}_0}    \|{\omega}\|_{W^{1,p}_0}
    \|{\eta}\|_{W^{1,p}_0} \\& \le \varepsilon  \|{u}\|^p_{W^{1,p}_0}  +
    c_{\varepsilon}  (\|{\omega}\|_{W^{1,p}_0}
    \|{\eta}\|_{W^{1,p}_0})^{p'}
    \\& \le \varepsilon  \|{u}\|^p_{W^{1,p}_0}  +
   \varepsilon \|{\omega}\|_{W^{1,p}_0}^p +
   c_{\varepsilon} \|{\eta}\|_{W^{1,p}_0}^{\frac{p}{p-2}}
   \\& \le \varepsilon  \|{u}\|^p_{W^{1,p}_0}  +
   \varepsilon \|{\omega}\|_{W^{1,p}_0}^p +
   c_{\varepsilon}(1+\|{\eta}\|_{W^{1,p}_0}^p),
  \end{align*}
	where we exploited for the last inequality that $a^{\frac p{p-2}}\leq (1+a)^{\frac p{p-2}}\leq (1+a)^{ p}\leq
	c(p)(1+a^{p})$, valid for all $a\ge 0$ and $p\ge 3$. Therefore,
  (\hyperlink{A.4}{A.4}) is satisfied for $\varepsilon>0$ sufficiently small.

  If $p\in (\frac{11}{5},3)$, then by interpolation with
  $\frac{1}{\rho}=\frac{1-\theta}{p^*}+\frac{\theta}{2}$, where
  $\rho=p\frac{5}{3}$, $\theta=\frac{2}{5}$ and
  $p^*=\frac{3p}{3-p}$, we obtain for all ${v}\in W^{1,p}_0(\Omega)^3$
  \begin{align}
    \|{v}\|_{L^{\rho}}\leq
    \|{v}\|_{L^2}^{\frac{2}{5}}\|v\|_{L^{p^*}}^{\frac{3}{5}}\leq
    \|{v}\|_{L^2}^{\frac{2}{5}}\|v\|_{W^{1,p}_0}^{\frac{3}{5}}.
    \label{eq:4.7.2} 
  \end{align}
  Hence, since  $\rho\ge 2p'$, we further conclude from
  \eqref{eq:4.7.2} in \eqref{eq:4.7.1} that for all
  ${\omega},{\eta} \in W^{1,p}_0(\Omega)^3 $
  \begin{align*}
    \vert\langle \tilde{B} \omega,\eta\rangle\vert
    &\leq    c\,     \|{u}\|_{L^{2}}^{\frac{2}{5}}
      \|u \|_{W^{1,p}_0}^{\frac{3}{5}}    \|{\eta}\|_{L^{2}}^{\frac{2}{5}}
      \|{\eta}\|_{W^{1,p}_0}^{\frac{3}{5}}\|\omega\|_{W^{1,p}_0} +  c\,     \|{u}\|_{L^{2}}^{\frac{2}{5}}
      \|u \|_{W^{1,p}_0}^{\frac{3}{5}}    \|\omega\|_{L^{2}}^{\frac{2}{5}}
      \|{\omega}\|_{W^{1,p}_0}^{\frac{3}{5}}\|\eta\|_{W^{1,p}_0} 
    \\
    &\leq  \varepsilon  \|{u}\|^p_{W^{1,p}_0} \! +
    \varepsilon   \|{\omega}\|_{W^{1,p}_0}^p \!+ c_\varepsilon
    \|{\eta}\|^p_{W^{1,p}_0} \!+  c_\varepsilon
      \|{u}\|_{L^2}^{\frac{4p}{5p-11}} +c_\varepsilon
      \|{\omega}\|_{L^2}^{\frac{4p}{5p-11}} +c_\varepsilon
      \|{\eta}\|_{L^2}^{\frac{4p}{5p-11}} , 
  \end{align*}
  where we applied an appropriate version of the $\varepsilon$-Young's
  inequality with exponents $\frac{10(p-1)}{5p-11}$,
  $\frac{5(p-1)}{3}$, $\frac{10(p-1)}{5p-11}$, $\frac{5(p-1)}{3}$ and
  $p$. Thus, (\hyperlink{A.4}{A.4}) is satisfied for $\varepsilon>0$ sufficiently small.

  Altogether, $\hat{B}:X\to X^*$ satisfies (\hyperlink{A.4}{A.4}) with
  $q= \frac{4p}{5p-11}$. As
  a result, $\hat A(t):X\to X^*$, $t\in I$, satisfies
  (\hyperlink{A.1}{A.1})--(\hyperlink{A.4}{A.4}).
  \hfill$\qed$
\end{proof}

Let us now define the discrete setup. For given
$m,\ell,k \in \mathbb N_0$, we define $X_h:=X_h^u \times X_h^\omega$,
where ${X^u_h\subset\mathcal{P}_m(\mathcal{T}_h)^3\cap
W^{1,p}_0(\Omega)^3}$, ${X^\omega_h\subset\mathcal{P}_k(\mathcal{T}_h)^3\cap
W^{1,p}_0(\Omega)^3}$ are appropriate finite element spaces, both equipped with the $W^{1,p}_0(\Omega)^3$--norm. In addition,
we define for $h>0$ and~an~\mbox{appropriate} finite element space
$Z_h\subset \mathcal{P}_\ell(\mathcal{T}_h)\cap Z$ equipped with the $Z$--norm, where $Z:=L^{p'}(\Omega)$, the
space  
\begin{align*}
  V_h:=\Big\{({u}_h,\omega_h) ^\top\in X_h \,\Big|\, \int_\Omega\divo {u}_h\,\eta_h\, dx=0\text{ for all }\eta_h\in Z_h\Big\}.
\end{align*}
For a null sequence
$(h_n)_{n\in \mathbb{N}}\subseteq\left(0,\infty\right)$, we set 
$V_n:=V_{h_n}$, $n\in \mathbb{N}$. To ensure that the spaces $(V_n)_{n\in \mathbb{N}}$ are
a quasi non-conforming approximation of $V$ in $X$ we make the
following assumption:
\begin{asum}[Projection operators]\label{proj}
  For every $h>0$ it holds \linebreak
  ${\mathcal{P}_1(\mathcal{T}_h)^3\subset X_h^u}$,   ${\mathcal{P}_1(\mathcal{T}_h)^3\subset X_h^\omega}$,
  $\mathbb{R}\subset Z_h$, and that there exist linear interpolation 
  operators $\Uppi_h^{\divo }: W^{1,p}_0(\Omega)^3\to X_h^u$ ,
  $\Uppi_h^{\omega}: W^{1,p}_0(\Omega)^3 \to X_h^\omega$ and
  $\Uppi_h^{Z}:Z \to Z_h$
  with the following properties:
  \begin{description}[{(iii)}]
  \item[(i)] \textbf{Divergence preservation of $\Uppi_h^{\divo }$ in
      $Z_h^*$:} It holds for all ${u}\in W^{1,p}_0(\Omega)^3$ and $\eta_h\in Z_h$
    \begin{align*}
      \int_\Omega \divo {u}\,\eta_h\, dx=\int_\Omega \divo
      \Uppi_h^{\divo }{u}\,\eta_h\, dx\,.
    \end{align*}
  \item[(ii)] \textbf{local $W^{1,1}$-stability of
      $\Uppi_h^{\divo }$:} There exists a constant $c>0$, independent
    of $h>0$, such that for every ${u}\in W^{1,p}_0(\Omega)^3$ and
    $K\in \mathcal{T}_h$\footnote{The neighbourhood $S_K$ of a simplex
      $K \in \mathcal T_h$ is defined via
      $S_K := \text{interior } \bigcup _{\{\overline K' \in \mathcal
        T_h \fdg \overline K'\cap \overline K \neq \emptyset\}}
      \overline K'$.}
    \begin{align*}
      \Vert \Uppi_h^{\divo }{u}\Vert _{L^1(K)}\leq
      c\,\Vert{u}\Vert _{L^1(S_K)}+c\, h_K\, \Vert \nabla{u}\Vert_{L^1(S_K)}.
    \end{align*}
  \item[(iii)] \textbf{local $L^1$-stability of $\Uppi_h^Z$:} There exists a constant
    $c>0$, independent of $h>0$, such that for every $\eta\in Z$ and $K\in \mathcal{T}_h$ 
    \begin{align*}
      \Vert \Uppi_h^{Z }{\eta}\Vert _{L^1(K)}\leq
      c\,\Vert{\eta}\Vert _{L^1(S_K)}. 
    \end{align*}
  \item[(iv)] \textbf{local $W^{1,1}$-stability of
      $\Uppi_h^{\omega }$:} There exists a constant $c>0$, independent
    of $h>0$, such that for every ${\omega}\in W^{1,p}_0(\Omega)^3$ and
    $K\in \mathcal{T}_h$
    \begin{align*}
      \Vert \Uppi_h^{\omega }{\omega}\Vert _{L^1(K)}\leq
      c\,\Vert{\omega}\Vert _{L^1(S_K)}+c\, h_K\, \Vert \nabla{\omega}\Vert_{L^1(S_K)}.
    \end{align*}
  \end{description}
\end{asum}
Certainly, the existence of such operators depends on the choice of
$X_h$ and $Z_h$. It is shown in \cite{BF91}, \cite{GL01}, \cite{GS03},
\cite{BBDR12}, \cite{GNS15}, \cite{DST2021}, \cite{Scott21} that $\Uppi_h^{\divo }$ exists for a
variety of spaces $X_h$ and $Z_h$, which, e.g., include the
Taylor-Hood, the spatially conforming Crouzeix--Raviart, and the MINI element in
dimension two and three.  Projection operators $\Uppi_h^Z$ satisfying
Assumption \ref{proj} (iii) are e.g. the Cl\'ement interpolation
operator (cf.~\cite{clement}) and a version of the Scott-Zhang
interpolation operator (cf.~\cite{zhang-scott}). The standard 
Scott-Zhang interpolation operator (cf.~\cite{zhang-scott}) satisfies
Assumption \ref{proj} (iv).  The abstract assumptions allow for an
easy extension of our results to other choices of $X_h$ and $Z_h$.

The next proposition shows that  the
approximation of divergence-free Sobolev functions through discretely
divergence-free finite element spaces perfectly fits into the
framework of quasi non-conforming approximations.

\begin{prop}\label{ex:3.5}
  Let $p > \frac 65$ and let Assumption~\ref{proj} be satisfied. Then,
  the sequence $(V_n)_{n\in \mathbb{N}}$ forms a quasi non-conforming
  approximation of $V$ in $X$.
\end{prop}
	
\begin{proof}
  Clearly, $(V,H,\textup{id}_V)$ and $(X,Y,\textup{id}_X)$ form
  evolution triples satisfying $V\subseteq X$ with \linebreak
  ${\|\!\cdot\!\|_V=\|\!\cdot\!\|_X}$ in $V$ and $H\subseteq Y$ with
  $(\cdot,\cdot)_H=(\cdot,\cdot)_Y$ in $H\times H$. So, let us verify
  that $(V_n)_{n\in\mathbb{N}}$ satisfies (\hyperlink{QNC.1}{QNC.1})
  and (\hyperlink{QNC.2}{QNC.2}):
	
  \textbf{ad (\hyperlink{QNC.1}{QNC.1})} Due to their finite
  dimensionality, the spaces $(V_n)_{n\in \mathbb{N}}$ are closed. We
  set  ${C:=\mathcal{V} \times C_0^\infty(\Omega)^3}$, where $\mathcal V=\{{v}\in C_0^\infty(\Omega)^3\fdg\divo
  {v}=0\}$. Let $({u},\omega)\in C$. Then, owing to standard
  estimates for polynomial projection operators (cf.~\cite[Lemma
  2.25]{tscherpel-phd}), the sequence
  $({u}_n, \omega_n) ^\top:=(\Uppi_{h_n}^{\divo }{u}, \Uppi_{h_n}^{\omega}{\omega}) ^\top\in V_n$,
  $n\in \mathbb{N}$, satisfies
  \begin{align*}
    \|({u}, \eta) ^\top-({u}_n, \omega_n) ^\top\|_X\leq
    c\,h_n\,\|({u}, \omega) ^\top\|_{W^{2,p
    }(\Omega)^3}\;\;\to\;\; 0\quad(n\to \infty).
  \end{align*}
		
  \textbf{ad (\hyperlink{QNC.2}{QNC.2})} Let
  $(\boldsymbol{u}_n, \boldsymbol{\omega}_n) ^\top\in L^p(I,V_{m_n})$, $n\in \mathbb{N}$, where
  $(m_n)_{n\in\mathbb{N}}\subseteq\mathbb{N}$ with $m_n\to\infty$
  $(n\to\infty)$, be such that
  $(\boldsymbol{u}_n,\boldsymbol{\omega}_n) ^\top\rightharpoonup (\boldsymbol{u},\boldsymbol{\omega}) ^\top$ in $L^p(I,X)$
  $(n\to\infty)$. Since the second component of $V$ and $X$ coincide
  we only have to show that $\boldsymbol u \in
  L^p(I,W^{1,p}_{0,\divo}(\Omega)^3)$. Let $\eta\in C^\infty_0(\Omega)$ and
  $\varphi\in C_0^\infty(I)$. As in the previous step we infer that
  the sequence $\eta_n:=\Uppi_{m_n}^Z\eta\in Z_{h_{m_n}}$,
  $n\in\mathbb{N}$, satisfies $\eta_n\to\eta$ in $Z$
  $(n\to\infty)$. On the other hand, in view of the definition of
  $V_{m_n}$ there holds 
  $\langle\divo \boldsymbol{u}_n(t),\eta_n\rangle_Z=0$ for almost
  every $t\in I$ and $n\in \mathbb{N}$. Thus, for every $n\in
  \mathbb{N}$ we have 
  \begin{align}
    \int_I{\int_\Omega\divo \boldsymbol{u}_n(s)\,\eta_n\,dx \,\varphi(s)\,ds}=0.\label{eq:3.4}
  \end{align}
  By passing in \eqref{eq:3.4} for $n\to \infty$, we obtain for every $\eta \in C_0^\infty(\Omega)$ and $\varphi\in C_0^\infty(I)$
  \begin{align*}
    \int_I{\int_\Omega \divo \boldsymbol{u}(s)\,\eta\, dx\,\varphi(s)\,ds}=0,
  \end{align*}
  i.e.,  $\boldsymbol{u}\in L^p(I,W^{1,p}_{0,\divo}(\Omega)^3) $.
  \hfill$\qed$
\end{proof}

\begin{rmk}
  From the proof of Proposition \ref{ex:3.5} it is clear that instead
  of Assumption \ref{proj} it is sufficient to require that there
  exist dense subsets $\mathcal X$ of $W^{1,p}_0(\Omega)^3$ and
  $\mathcal Z$ of $L^{p'}(\Omega)$ and linear interpolation operators
  $\Uppi_h^{\divo }: \mathcal X \to X_h^u$ ,
  $\Uppi_h^{\omega}: \mathcal X \to X_h^\omega$ and
  $\Uppi_h^{Z}: \mathcal Z\to Z_h$ which have a global approximation
  property, i.e.,  for every $({u},\omega)\in \mathcal X \times \mathcal X$
  there holds $\| ({u}, \omega) ^\top-(\Uppi_{h}^{\divo }{u},
  \Uppi_{h}^{\omega}{\omega}) ^\top\|_{X} \to 0$ for $h\to 0$, as well as for
  every $\eta \in \mathcal Z$
  there holds $\| \eta - \Uppi_{h}^{Z}\eta\|_{Z} \to 0$ for $h\to 0$.
\end{rmk}
Let us summarize  our setup for the treatment of problem
\eqref{eq:erf} describing the motion of micropolar electrorheological
fluids.

\begin{asum}\label{asumex}
  Let $\Omega\subseteq \mathbb{R}^3$ be a bounded polygonal
  Lipschitz domain, $I:=\left(0,T\right)$, $T<\infty$, and
  $p>\frac{11}{5}$. We make the following assumptions:
  \begin{description}[{(iii)}]
  \item[(i)] The stress tensors and the electric field satisfy
    (\hyperlink{S.1}{S.1})--(\hyperlink{S.3}{S.3}),
    (\hyperlink{N.1}{N.1})--(\hyperlink{N.3}{N.3}), and
    (\hyperlink{E.1}{E.1}).
  \item[(ii)] $(V,H,\text{id})$, $(X,Y,\text{id})$ and
    $(V_n)_{n\in \mathbb{N}}$ are defined as in Proposition
    \ref{ex:3.5}.
  \item[(iii)] $\mathcal U_0:=({u}_0, \omega_0) ^\top\in H$ and $\mathcal U^0_n:=({u}_n^0,\omega^0_n)^\top\in V_n$, $n\in
    \mathbb{N}$, such that $\mathcal {U}_n^0\to \mathcal{U}_0$ in $Y$
    $(n\to \infty)$ and $\sup_{n\in
      \mathbb{N}}{\|\mathcal{U}_n^0\|_Y}\leq \|\mathcal{U}_0\|_H$. 
  \item[(iv)] 
    $\mathbfcal F:=(\boldsymbol{f},\boldsymbol \ell)^\top\in L^{p'}(I,X^*)$.
  \item[(v)] $\widehat A(t):X\to X^*$, $t\in I$, is defined as in Proposition \ref{4.7}.
  \end{description}
\end{asum}
Furthermore, we denote by $e:=(\textup{id}_V)^*R_H:V\to V^*$ the
canonical embedding with respect to the evolution triple
$(V,H,\textup{id})$. Thus, the quasi non-conforming
Rothe--Galerkin scheme in this setup reads:

\begin{alg}
  Let Assumption \ref{asumex} be satisfied. For given
  $K,n\in \mathbb{N}$ the sequence of iterates
  ${\mathcal U_n^k:=({u}_n^{k}, {\omega}_n^{k})^\top\in V_n}$,
  ${k=0,\dots,K}$ is given solving the implicit Rothe--Galerkin scheme for
  ${\tau=\frac{T}{K}}$ and $k=1,\dots,K$
  \begin{align}
    (d_\tau \,\mathcal {U}^k_n,\mathcal {W}_n)_Y+\langle [\widehat A]^\tau_k\,
    \mathcal {U}^k_n,\mathcal {W}_n\rangle_X=\langle
    [\mathbfcal{F}]^\tau_k,\mathcal {W}_n\rangle_X\quad\text{ for all
    }\mathcal {W}_n\in V_n.\label{eq:erfnon}
  \end{align}
\end{alg}

By means of Proposition \ref{5.1}, Proposition \ref{apriori}, Theorem
\ref{5.17} and the observations already made in Proposition
\ref{4.7} and Proposition \ref{ex:3.5}, we can immediately conclude the following results.
		
\begin{thm}[Well-posedness, stability and weak convergence of
  \eqref{eq:erfnon}]\label{rem:7.4}\newline
  Let Assumption \ref{asumex} be satisfied. Then, it holds:
  \begin{description}[{(III)}]
  \item[(I)] \textbf{Well-posedness:} For every $K,n\in \mathbb{N}$
    there exist iterates $(\mathcal{U}_n^k)_{k=0,\dots,K}\subseteq V_n$,
    solving \eqref{eq:erfnon}, without any restrictions on the
    step-size. 
  \item[(II)] \textbf{Stability:} The
    corresponding piece-wise constant
    interpolants   $\overline{\mathbfcal{U}}_n^\tau\in
    \mathbfcal{S}^0(\mathcal{I}_\tau,X)$, $K,n\in\mathbb{N}$ with $\tau=\frac{T}{K}$, are bounded in $L^p(I,X)\cap L^\infty(I,Y)$.
  \item[(III)] \textbf{Weak convergence:} If
    $(\overline{\mathbfcal{U}}_n)_{n\in\mathbb{N}}:=(\overline{\mathbfcal{U}}_{m_n}^{\tau_n})_{n\in\mathbb{N}}$,
    where $\tau_n=\frac{T}{K_n}$ and
    ${K_n,m_n\to \infty }$~${(n\to\infty)}$, is an arbitrary diagonal
    sequence of the piece-wise constant interpolants
    ${\overline{\mathbfcal{U}}_n^\tau\in
      \mathbfcal{S}^0(\mathcal{I}_\tau,X)}$, ${K,n\in\mathbb{N}}$ with
    ${\tau=\frac{T}{K}}$, then there exists a not relabelled
    subsequence and a weak limit
    ${\overline{\mathbfcal{U}}\in L^p(I,V)\cap L^\infty(I,H)}$ such
    that
    \begin{align*}
      \begin{alignedat}{2}
        \overline{\mathbfcal{U}}_n&\;\;\rightharpoonup\;\;\overline{\mathbfcal{U}}&&\quad\text{ in }L^p(I,X),\\
        \overline{\mathbfcal{U}}_n&\;\;\overset{\ast}{\rightharpoondown}\;\;\overline{\mathbfcal{U}}&&\quad\text{
          in }L^\infty(I,Y),
      \end{alignedat}
         \begin{aligned}
           \quad(n\to\infty).
         \end{aligned}
    \end{align*}
    Furthermore, it follows that $\overline{\mathbfcal{U}}\in
    \mathbfcal{W}_e^{1,p,p}(I,V,H)\cap L^\infty(I,H)$ satisfies
    $\overline{\mathbfcal{U}}(0)=\mathcal{U}_0$ in $H$ and for all
    ${\boldsymbol{\phi}\in L^p(I,V)}$ 
    \begin{align*}
      \int_I{\Big\langle\frac{d_e\overline{\mathbfcal{U}}}{dt}(t),\boldsymbol{\phi}(t)\Big\rangle_V\,dt}+\int_I{\langle
      A(t)(\overline{\mathbfcal{U}}(t)),\boldsymbol{\phi}(t)\rangle_X\,dt}=\int_I{\langle\mathbfcal{F}(t),\boldsymbol{\phi}(t)\rangle_X\,dt}. 
    \end{align*}
  \end{description}
\end{thm}
		
\begin{proof}
  \textbf{ad (I)/(II)} The assertions follow immediately from
  Proposition \ref{5.1} and Proposition \ref{apriori}, since the
  operator family $\hat A(t):=S(t)+N(t)+\hat{B}:X\to X^*$, $t\in I$, satisfies
  (\hyperlink{A.1}{A.1})--(\hyperlink{A.4}{A.4}) with $c_1=0$ due to
  Proposition \ref{4.7}.
			
  \textbf{ad (III)} The assertions follow from Theorem \ref{5.17}. To
  be more precise, Theorem \ref{5.17} initially yields that
  $\overline{\mathbfcal{U}}\in \mathbfcal{W}_e^{1,p,q}(I,V,H)$, where
  $q>1$ is specified in the proof of Proposition \ref{4.7}, satisfies
  $\overline{\mathbfcal{U}}(0)=\mathcal{U}_0$ in $H$ and for all
  $\boldsymbol{\phi}\in C^1_0(I,V)$
  \begin{align*}
      \int_I{\Big\langle\frac{d_e\overline{\mathbfcal{U}}}{dt}(t),\boldsymbol{\phi}(t)\Big\rangle_V\,dt}+\int_I{\langle
    \hat A(t)(\overline{\mathbfcal{U}}(t)),\boldsymbol{\phi}(t)\rangle_X\,dt}=\int_I{\langle\mathbfcal{F}(t),\boldsymbol{\phi}(t)\rangle_X\,dt}
  \end{align*}
  Since
  $\langle\hat{B}(\overline{\mathbfcal{U}}(t)),\boldsymbol{\phi}(t)\rangle_X=\langle
  B(\overline{\mathbfcal{U}}(t)),\boldsymbol{\phi}(t)\rangle_X$ for
  almost every $t\in I$ and all $\boldsymbol{\phi}\in C^1_0(I,V)$ as
  well as $B(\overline{\mathbfcal{U}}(\cdot))\in L^{p'}(I,V^*)$
  (cf.~\cite[Theorem~2.30]{br-mperf}), as
  $\overline{\mathbfcal{U}}(t)\in L^p(I,V)\cap L^\infty(I,H)$, we
  actually proved that
  $\overline{\mathbfcal{U}}\in \mathbfcal{W}_e^{1,p,p}(I,V,H)\cap
  L^\infty(I,H)$, is such that for all $\boldsymbol{\phi}\in C^1_0(I,V)$
  \begin{align*}
      \int_I{\Big\langle\frac{d_e\overline{\mathbfcal{U}}}{dt}(t),\boldsymbol{\phi}(t)\Big\rangle_V\,dt}+\int_I{\langle
    A(t)(\overline{\mathbfcal{U}}(t)),\boldsymbol{\phi}(t)\rangle_X\,dt}=\int_I{\langle\mathbfcal{F}(t),\boldsymbol{\phi}(t)\rangle_X\,dt}.
    \tag*{$\qed$}
  \end{align*}
\end{proof}
	
\begin{rmk}\label{rem:ex}
  The results in Theorem~\ref{rem:7.4} for micropolar
  electrorheological fluids are completely new. There are some
  previous results for the various subcases. Let us mention some of
  them. For the special
  subcase of generalized Newtonian fluids the result in Theorem
  \ref{rem:7.4} is, among others, already contained in
  \cite{tscherpel-phd} (cf.~\cite{sueli-tscherpel}). 
  Theorem \ref{rem:7.4} extends the convergence of a conforming
  implicit fully discrete Rothe--Galerkin scheme of an evolution
  problem with Bochner pseudo-monotone operators, proved in
  \cite{BR19}, to the quasi non-conforming setting. Convergence
  results with optimal rates for the unsteady $p$-Navier-Stokes
  equations and related problems can be found e.g.~in \cite{bdr-3-2},
  \cite{sarah-phd}, \cite{breit-mensah} and~\cite{br-parabolic}. For
  the special subcase of micropolar fluids with $p=2$ optimal
  convergence rates for strong solutions are proved in \cite{NSAT14},
  \cite{Ra16}. The convergence of a fully discrete approximation
  towards a mollified problem for electrorheological fluids with
  variable exponents is proved in \cite {CHP10}.  
\end{rmk}

\subsection {A modified Smagorinsky model}

We consider the following modified\footnote{Models of this type are
  known in literature  as improved eddy
viscosity models.}  version of the Smagorinsky model
for turbulent flows 
\begin{align}
  \begin{split}
    \begin{alignedat}{2}
      \partial_t \boldsymbol{u}-\divo \big ( \delta\, \abs{\bD \boldsymbol
        u} \bD \boldsymbol u\big )+\text{div}(\boldsymbol{u}\otimes
      \boldsymbol{u}) +\nabla\boldsymbol{q}&=\boldsymbol{f}&&\quad\text{ in }I\times\Omega,
      \\
      \divo \boldsymbol{u}&=0&&\quad\text{ in }I\times\Omega,
      \\
      \boldsymbol{u}&=\mathbf{0}&&\quad\text{ on }I\times \partial\Omega,
      \\
  \boldsymbol{u}(0)&={u}_0&&\quad\text{ in }\Omega.
    \end{alignedat}
  \end{split}\label{eq:lomax}
\end{align}
In these equations $\boldsymbol u$ denotes the velocity,
$\boldsymbol{q}$ is the pressure, $ {\boldsymbol {f}}$ is the
mechanical body force and
$ \delta (\cdot):=\dist (\cdot, \partial \Omega)$ the distance from
the boundary $\partial \Omega$. This modification (with a
position-dependent eddy viscosity) is intended to improve some of the
weakness of the original Smagorinsky model, which is considered to be
too dissipative in laminar regimes and close to walls, and thus does
not work satisfactorily for the computation of boundary layers and of
the transition to turbulence (cf.~\cite{CRL14}, \cite{Sag02},
\cite{ABLN20}). The introduction of models similar to \eqref{eq:lomax}
dates back to Obukhov and van Driest, at least in the case of a
channel flow, see~\cite{CRL14}. Improved (as well as adaptive and
selective) eddy viscosity models are mostly used in numerical
computations, while the basic analytical problems are still mainly
open. 
%

For the functional setting we make use of the standard theory of
weighted Lebesgue spaces $L^p(\Omega; \sigma)$ and Sobolev spaces
$W^{1,p}(\Omega;\sigma) $ (cf.~\cite{heinonen},
\cite{kufner_opic}) with a weight $\sigma$ belonging to the Muckenhoupt
class $\mathcal A_p$. The norm in
$L^p(\Omega; \sigma)$ is defined as $\|v\|_{L^p(\Omega; \sigma)}:=
\big (\int _\Omega \abs{v}^p \, \sigma\, dx\big )^{\frac 1p}$.
The dual space $(L^p(\Omega; \sigma))^*$ can be
identified with 
$L^{p'}(\Omega;\sigma')$, where $\sigma':=\sigma^{\frac{-1}{p-1}}$. In
particular we have
\begin{equation}  \label{eq:hoelder}
  \Bigabs{\int_\Omega{f}{g}\, dx }\le \|{f}\|_{L^p(\Omega;\sigma)} \|{g}\|_{L^{p'}(\Omega;\sigma')},
\end{equation}
if $f \in L^{p}(\Omega;\sigma)$,  $g \in L^{p'}(\Omega;\sigma')$. Note
that $\sigma \in \mathcal A_p$ iff $\sigma' \in \mathcal A_{p'}$. The
space  $W^{1,p}_0(\Omega;\sigma) $ is defined as the completion of
$C^\infty_0(\Omega)$ with respect to the norm $\|\cdot \|_{W^{1,p}(\Omega;
  \sigma)}:=\|\cdot \|_{L^p(\Omega;
  \sigma)} +\|\nabla \cdot \|_{L^p(\Omega; \sigma)}$. The norm
$\|\nabla \cdot \|_{L^p(\Omega; \sigma)}$ is an equivalent norm on ${W^{1,p}_0(\Omega;
  \sigma)}$. 
Let us summarize some
facts needed for our analysis. The distance function $ \delta$ belongs to 
the
Muckenhoupt class $\mathcal A_p$ for $p>2$ (cf.~\cite{dyda}).  The
embedding $W^{1,3}_0(\Omega;\delta) $ into $L^{q}(\Omega;\delta)$ is compact for
all $q \in [1,6)$ (cf.~\cite[Theorem~2.6]{Froe}). Moreover,
$W^{1,3}_0(\Omega;\delta) $ embeds into $L^{3}(\Omega;\delta^{\beta})$ if 
$\beta
\ge -2$ (cf.~\cite{dyda}).

To treat problem \eqref{eq:lomax} we define the function spaces
\begin{align*}
  \begin{alignedat}{2}
    X&:=W^{1,3}_0(\Omega;\delta)^3,   &\qquad Y&:=L^2(\Omega)^3,
    \\
    V&:=W^{1,p}_{0,\divo }(\Omega;\delta),    &H&:=L^2_{\divo }(\Omega),
  \end{alignedat}
\end{align*}
and the operators ${S,B:X\to X^*}$
for all $u, v \in X$ via
\begin{gather*}
  \langle S{u},v\rangle_X:= \int_\Omega \delta\, \abs{\bD  u}
  \bD  u : \bD v \,dx\,,\qquad \langle B{u}, v \rangle_X:=
  \int_\Omega u \otimes u : \nabla v \,dx\,,
\end{gather*}
and set $A(t):=S +B:X\to X^*$, $t \in I$. Then, \eqref{eq:lomax} for ${u}_0\in H$ and
$\boldsymbol f \in L^{p'}(I,X^*)$ can be re-written as the abstract 
evolution equation
\begin{align*}
  \begin{split}
    \begin{alignedat}{2}
      \frac{d\boldsymbol{u}}{dt}(t)+A(t)(\boldsymbol{u}(t))&=\boldsymbol{f}(t)&&\quad\text{
        in }V^*,
      \\
      \boldsymbol{u}(0)&={u}_0&&\quad\text{ in }H.
    \end{alignedat}
  \end{split}
\end{align*}
As in the previous section we modify the operator family $A(t)$, $t
\in I$, and define  $\hat{A}(t):X\to X^*$ via
$\hat A(t):=S +\hat B$, $t \in I$, where $\hat B$ is given 
for all $u,v\in X$ via
\begin{gather*}
  \langle \hat B {u},v \rangle_X:= \frac 12 \int_\Omega v
  \otimes u : \nabla u - {{u}\otimes
  {u}:\nabla {v}} \,dx\,.
\end{gather*}
The operator $\hat{B}$ is a symmetrized extension of $B$, as
$\langle\hat{B}{u},{v}\rangle_X=\langle B{u},{v}\rangle_X$ for all
$u,v\in X$ and fulfils
$\langle\hat{B}{u},{v}\rangle_X=0$ for all
${u}\in X$. Thus, we have the following result:
\begin{prop}\label{4.7-lo}  The operator family
  $\hat A(t):X\to X^*$, $t\in I$, satisfies
  (\hyperlink{A.1}{A.1})--(\hyperlink{A.4}{A.4}) with $p=3$. 
\end{prop}

\begin{proof}
  Let us first consider $S:X\to X^*$. From a straightforward
  modification of the  theory of
  Nemytski\u{\i} operators to weighted spaces  we
  deduce for almost every $t\in I$ the well-definiteness and
  continuity of $S:X\to X^*$, as well as $
     \langle S{u},u\rangle_X 
     =\|\bD u \|_{L^3(\Omega;\delta)}^3$,
  which is condition (\hyperlink{A.3}{A.3}). Condition
  (\hyperlink{A.2}{A.2}) is obviously satisfied, since the operator
  $S$ does not depend on time. The monotonicity of $S:X\to X^*$
  follows in the same way as for the operator $\langle \tilde S
  u,v\rangle _{W^{1,3}_0(\Omega)}:= \int_\Omega \abs{\bD  u}
  \bD  u: \bD v \,dx $ as the arguments to prove this are
  pointwise. Hence, $S:X\to X^*$ is pseudo-monotone, i.e., condition
  (\hyperlink{A.1}{A.1}) is satisfied.  Condition
  (\hyperlink{A.4}{A.4}) follows from the pointwise
 $\varepsilon$-Young inequality. 
		
  Next, we treat $\hat{B}:X\to X^*$. Again a
  straightforward modification of theory of Nemytski\u{\i} operators
  to weighted spaces and the above compact embedding yield that
  $\hat{B}:X\to X^*$ is bounded and pseudo-monotone, i.e., satisfies
  (\hyperlink{A.1}{A.1}). Condition
  (\hyperlink{A.2}{A.2}) is obvious, since the operator
  $\hat B$ does not depend on time.   As already
  pointed out above we have $\langle \hat{B}{u},{u}\rangle_X=0$ for
  all ${u}\in X$, i.e., $\hat{B}$ satisfies (\hyperlink{A.3}{A.3}).
  We use H\"older's inequality \eqref{eq:hoelder}, the above stated
  embeddings and the $\varepsilon$-Young inequality to verify that for every $u, v \in X$ there holds 
  \begin{align*}
    \vert\langle \hat {B} u,v \rangle\vert
    &\leq \|{v}\|_{L^{3}(\Omega;\delta^{-\frac 12})}
    \|{u}\|_{L^{3}(\Omega;\delta^{-\frac 12})}\|\nabla u\|_{L^{3}(\Omega;\delta)}+
        \|{u}\|_{L^{3}(\Omega;\delta^{-\frac 12})}^2
    \|\nabla v\|_{L^{3}(\Omega;\delta)}
    \\
    &\leq \varepsilon \|\nabla u\|^3_{L^{3}(\Omega;\delta)}+
    c_\varepsilon \|\nabla v\|^3_{L^{3}(\Omega;\delta)},
  \end{align*}
  i.e., (\hyperlink{A.4}{A.4}) for $\varepsilon>0$ sufficiently small.

  Altogether, $\hat A(t):X\to X^*$, $t\in I$, satisfies
  (\hyperlink{A.1}{A.1})--(\hyperlink{A.4}{A.4}).
  \hfill$\qed$
\end{proof}

Let us now define the discrete setup. For $p\in (1,\infty)$ and a
Muckenhoupt weight $\sigma \in \mathcal A_p$ we define 
$Z:=L^{p'}(\Omega;\sigma')$. For given
$m,\ell \in \mathbb N_0$ we denote by 
${X_h\subset\mathcal{P}_m(\mathcal{T}_h)^3\cap W^{1,p}_0(\Omega;\sigma)^3}$
and $Z_h\subset \mathcal{P}_\ell(\mathcal{T}_h)\cap Z$ suitable
finite element spaces equipped with the
$W^{1,p}_0(\Omega;\sigma)^3$-norm~and~\mbox{$Z$-norm},
respectively. In addition, we define for $h>0$ the space 
\begin{align*}
  V_h:=\Big\{{u}_h\in X_h \,\Big|\, \int_\Omega\divo {u}_h\,\eta_h\, dx=0\text{ for all }\eta_h\in Z_h\Big\}.
\end{align*}
For a null sequence
$(h_n)_{n\in \mathbb{N}}\subseteq\left(0,\infty\right)$ we set 
$V_n:=V_{h_n}$, $n\in \mathbb{N}$. To ensure that the spaces $(V_n)_{n\in \mathbb{N}}$ are
a quasi non-conforming approximation of $V$ in $X$ we make the
following assumption:
\begin{asum}[Projection operators]\label{proj:lomax}
  For every $h>0$ 
  there exist linear interpolation 
  operators $\Uppi_h^{\divo }: W^{1,p}_0(\Omega;\sigma)^3\to X_h$  and $\Uppi_h^{Z}:Z\to Z_h$
  with the following properties:
  \begin{description}[{(iii)}]
  \item[(i)] \textbf{Divergence preservation of $\Uppi_h^{\divo }$ in
      $Z_h^*$:} It holds for all ${u}\in W^{1,p}_0(\Omega;\sigma)^3$ and $\eta_h\in Z_h$
    \begin{align*}
      \int_\Omega \divo {u}\,\eta_h\,dx=\int_\Omega \divo
      \Uppi_h^{\divo }{u}\,\eta_h\, dx\,.
    \end{align*}
  \item[(ii)] \textbf{global $W^{1,p}(\Omega;\sigma)$-approximability
      of $\Uppi_h^{\divo }$:} There exists a constant $c>0$,
    independent of $h>0$, such that for every
    ${u}\in W^{1,p}_0(\Omega;\sigma)^3\cap W^{2,p}(\Omega;\sigma)^3$ 
    \begin{align*}
      \Vert  u - \Uppi_h^{\divo }{u}\Vert _{L^p(\Omega;\sigma)}+
      h\, \Vert \nabla u -\nabla \Uppi_h^{\divo }{u}\Vert _{L^p(\Omega;\sigma)} \leq
      c\,h^2\,\Vert{\nabla ^2u}\Vert _{L^p(\Omega;\sigma)}. 
\end{align*}
  \item[(iii)] \textbf{global $L^{p'}(\Omega;\sigma')$-approximability of 
$\Uppi_h^Z$:} There exists a constant
    $c>0$, independent of $h>0$, such that for every $\eta\in Z\cap
    W^{1,p'}(\Omega;\sigma')$ 
    \begin{align*}
      \Vert \eta - \Uppi_h^{Z }{\eta}\Vert _{L^{p'}(\Omega;\sigma')}\leq
      c\,h\,\Vert \nabla{\eta}\Vert _{L^{p'}(\Omega;\sigma')}. 
    \end{align*}
  \end{description}
\end{asum}
\begin{rmk}\label{rem:lo}
  Since interpolation operators in weighted spaces are not so common
  in the literature, we discuss them in some detail.
  
  (i) The Cl\'ement
  interpolation operator (cf.~\cite{clement}) satisfies Assumption
  \ref{proj:lomax} (iii). Indeed, the proof of \cite[Theorem
  4.2]{Ba16}, using a duality argument and a local Poincar\'e
  inequality, also works in the setting of weighted spaces in view of
  the local Poincar\'e inequality in weighted spaces
  (cf.~\cite[Theorem 5.1]{john}).

  (ii) The existence of an operator $\Uppi_h^{\divo}$ satisfying
  Assumptions \ref{proj:lomax} (i), (ii) depends on the choice of
  $X_h$ and $Z_h$. The general strategy from \cite[Section~VI.4]{BF91} can be adapted
  to the weighted setting (cf.~\cite{DOS20} for a similar approach). To do so one needs a projection operator
  $\Uppi_h: W^{1,p}_0(\Omega,\sigma)^3\to \mathcal P_k(\mathcal
  T_h)^3\cap W^{1,1}_0(\Omega)^3$, where $k\in \mathbb N$ is such that
  $\mathcal P_k(\mathcal T_h)^3 \subset X_h$, which satisfies a local
  approximation property, i.e., for every
  ${v}\in W^{1,p}_0(\Omega,\sigma)^3\cap W^{2,p}(\Omega;\sigma)^3$ and
  $K\in \mathcal{T}_h$
  \begin{align}\label{eq:est2}
    \Vert  v - \Uppi_h{v}\Vert _{L^p(K;\sigma)}+
    h\, \Vert \nabla v -\nabla \Uppi_h{v}\Vert _{L^p(K;\sigma)} \leq
    c\,h^2\,\Vert{\nabla ^2v}\Vert _{L^p(S_K;\sigma)}. 
  \end{align}
  The existence of such an operator is proved in \cite[Theorem~5.2,
  5.3]{NOS16}. Moreover, one needs a correction operator
  $\Uppi_h^{\textrm{cor}}:W^{1,1}_0(\Omega)^3 \to X_h$ which is locally
  $W^{1,1}$-stable, i.e., there exists a constant $c>0$, independent
  of $h>0$, such that for every ${v}\in W^{1,1}_0(\Omega)^3$ and
  $K\in \mathcal{T}_h$
  \begin{align*}
    \Vert \Uppi_h^{\textrm{cor}}{v}\Vert _{L^1(K)}\leq
    c\,\Vert{v}\Vert _{L^1(S_K)}+c\, h_K\, \Vert \nabla{v}\Vert_{L^1(S_K)}.
  \end{align*}
  This inequality implies that there exists a constant $c>0$,
  independent of $h>0$, such that for every
  ${v}\in W^{1,p}_0(\Omega;\sigma)^3$ and $K\in \mathcal{T}_h$
  \begin{align}\label{eq:est3}
    \Vert \Uppi_h^{\textrm{cor}}{v}\Vert _{L^p(K;\sigma)}\leq
    c\,\Vert{v}\Vert _{L^p(S_K;\sigma)}+c\, h_K\, \Vert \nabla{v}\Vert_{L^p(S_K;\sigma)}.
  \end{align}
  The proof of this assertion just uses H\"older's inequality, the
  equivalence $\|g\|_{L^\infty(K)} \sim \dashint_K\abs{g}\, dx$ valid
  for all polynomials $g \in \mathcal P_m(K)$, and  $\sigma \in \mathcal A_p$.
  From \eqref{eq:est2}, \eqref{eq:est3} one easily deduces that
$$
\Uppi ^{\divo} _h(v):= \Uppi _h(v)+ \Uppi ^{\textrm{cor}}_h(v -\Uppi _h v)
$$
satisfies Assumptions \ref{proj:lomax} (i), (ii). Consequently, at
least for the MINI element we proved that Assumptions \ref{proj:lomax}
is satisfied.  The abstract assumptions allow for an easy
extension of our results to other choices of $X_h$ and $Z_h$.
\end{rmk}
Proceeding in the same way as in the proof of Proposition \ref{ex:3.5}
one can show:
\begin{prop}\label{ex:3.5-lomax}
  Let Assumption~\ref{proj:lomax} be satisfied for $p=3$ and $\sigma= 
\delta$. Then,
  the sequence $(V_n)_{n\in \mathbb{N}}$ forms a quasi non-conforming
  approximation of $V$ in $X$.
\end{prop}

Let us summarize  our setup for the treatment of problem
\eqref{eq:lomax}.

\begin{asum}\label{asumex-lo}
  Let $\Omega\subseteq \mathbb{R}^3$ be  a bounded polygonal
  Lipschitz domain, $I:=\left(0,T\right)$, and $T<\infty$. We make the following assumptions:
  \begin{description}[{(iii)}]
  \item[(i)] $(V,H,\text{id})$, $(X,Y,\text{id})$ and
    $(V_n)_{n\in \mathbb{N}}$ are defined as in Proposition
    \ref{ex:3.5-lomax}.
  \item[(ii)] ${u}_0\in H$ and ${u}_n^0\in V_n$, $n\in
    \mathbb{N}$, are such that $ {u}_n^0\!\to\! {u}_0$ in $Y$
    $(n\!\to\! \infty)$ and ${\sup_{n\in
      \mathbb{N}}{\|{u}_n^0\|_Y}\!\leq \!\|{u}_0\|_H}$. 
  \item[(iii)] 
    $\boldsymbol{f}\in L^{p'}(I,X^*)$.
  \item[(iv)] $\hat A(t):X\to X^*$, $t\in I$, is defined as in Proposition \ref{4.7-lo}.
  \end{description}
\end{asum}
Furthermore, we denote by $e:=(\textup{id}_V)^*R_H:V\to V^*$ the
canonical embedding with respect to the evolution triple
$(V,H,\textup{id})$.  Thus, the quasi non-conforming
Rothe--Galerkin scheme in this setup reads:

\begin{alg}
  Let Assumption \ref{asumex-lo} be satisfied. For given
  $K,n\in \mathbb{N}$ the sequence of iterates
  ${{u}_n^{k}\in V_n}$,
  ${k=0,\dots,K}$ is given solving the implicit Rothe--Galerkin scheme for
  $\tau=\frac{T}{K}$ and ${k=1,\dots,K}$
  \begin{align}
    (d_\tau {u}^k_n, {v}_n)_Y+\langle [\widehat A]^\tau_k\,
    {u}^k_n, {v}_n\rangle_X=\langle
    [\boldsymbol{f}]^\tau_k, {v}_n\rangle_X\quad\text{ for all
    }{v}_n\in V_n.\label{eq:lomaxnon}
  \end{align}
\end{alg}

By means of Proposition \ref{5.1}, Proposition \ref{apriori}, Theorem
\ref{5.17} and the observations already made in Proposition
\ref{4.7-lo} and Proposition \ref{ex:3.5-lomax}, we can conclude in the same
way as in Theorem~\ref{rem:7.4} the following results: 
		
\begin{thm}[Well-posedness, stability and weak convergence of
  \eqref{eq:lomaxnon}]\label{rem:7.4-lo}\newline
  Let Assumption \ref{asumex-lo} be satisfied. Then, it holds:
  \begin{description}[{(III)}]
  \item[(I)] \textbf{Well-posedness:} For every $K,n\in \mathbb{N}$
    there exist iterates $({u}_n^k)_{k=0,\dots,K}\subseteq V_n$,
    solving \eqref{eq:lomaxnon}, without any restrictions on the
    step-size. 
  \item[(II)] \textbf{Stability:} The
    corresponding piece-wise constant
    interpolants   $\overline{\boldsymbol{u}}_n^\tau\in
    \mathbfcal{S}^0(\mathcal{I}_\tau,X)$, $K,n\in\mathbb{N}$ with $\tau=\frac{T}{K}$, are bounded in $L^p(I,X)\cap L^\infty(I,Y)$.
  \item[(III)] \textbf{Weak convergence:} If
    $(\overline{\boldsymbol{u}}_n)_{n\in\mathbb{N}}:=(\overline{\boldsymbol{u}}_{m_n}^{\tau_n})_{n\in\mathbb{N}}$,
    where $\tau_n=\frac{T}{K_n}$ and
    ${K_n,m_n\to \infty }$~${(n\to\infty)}$, is an arbitrary diagonal
    sequence of the piece-wise constant interpolants
    ${\overline{\boldsymbol{u}}_n^\tau\in
      \mathbfcal{S}^0(\mathcal{I}_\tau,X)}$, ${K,n\in\mathbb{N}}$ with
    ${\tau=\frac{T}{K}}$, then there exists a not relabelled
    subsequence and a weak limit
    ${\overline{\boldsymbol{u}}\in L^p(I,V)\cap L^\infty(I,H)}$ such
    that
    \begin{align*}
      \begin{alignedat}{2}
        \overline{\boldsymbol{u}}_n&\;\;\rightharpoonup\;\;\overline{\boldsymbol{u}}&&\quad\text{ in }L^p(I,X),\\
        \overline{\boldsymbol{u}}_n&\;\;\overset{\ast}{\rightharpoondown}\;\;\overline{\boldsymbol{u}}&&\quad\text{
          in }L^\infty(I,Y),
      \end{alignedat}
         \begin{aligned}
           \quad(n\to\infty).
         \end{aligned}
    \end{align*}
    Furthermore, it follows that $\overline{\boldsymbol{u}}\in
    \mathbfcal{W}_e^{1,p,p}(I,V,H)\cap L^\infty(I,H)$ satisfies
    $\overline{\boldsymbol{u}}(0)={u}_0$ in $H$ and for all
    ${\boldsymbol{\phi}\in L^p(I,V)}$ 
    \begin{align*}
      \int_I{\Big\langle\frac{d_e\overline{\boldsymbol{u}}}{dt}(t),\boldsymbol{\phi}(t)\Big\rangle_V\,dt}+\int_I{\langle
      A(t)(\overline{\boldsymbol{u}}(t)),\boldsymbol{\phi}(t)\rangle_X\,dt}=\int_I{\langle\boldsymbol{f}(t),\boldsymbol{\phi}(t)\rangle_X\,dt}. 
    \end{align*}
  \end{description}
\end{thm}
		
This result is to the best of the authors' knowledge the first one 
proving the convergence of a fully discrete approximation of problem
\eqref{eq:lomax}. Moreover, it is even the first existence proof of weak
solutions for the problem \eqref{eq:lomax} at all.

\newcommand{\squeezeup}{\vspace{-2.5mm}}
        
\section{Numerical experiments}
\label{sec:8}
To conclude, we want to present some numerical experiments with data
having low regularity~that perfectly suit the framework of this
article. 
All numerical experiments were conducted~employing the finite element
software \textsf{FEniCS} \cite{LW10}. All graphics are generated using
the \mbox{\textsf{Matplotlib}~library~\cite{Hun07}}.

We consider for $\Omega:=\left(-1,1\right)^2\subseteq \mathbb{R}^2$,
$T:=0.1$, $Q_T:=I\times \Omega$ and $p:=\frac{11}{5}$, the system
describing the unsteady motion of micropolar electrorheological fluids in two
dimensions, i.e.,\footnote{Here, $\levy:=(\begin{smallmatrix}
    0& 1\\
    \sm1& 0
  \end{smallmatrix})\in \mathbb{M}^{2\times 2}$ denotes the two-dimensional Levi--Civita tensor.
}
\begin{align}
  \begin{split}
    \begin{alignedat}{2}
      \partial_t \boldsymbol{u}-\divo
      \bS+\text{div}(\boldsymbol{u}\otimes \boldsymbol{u})
      +\nabla\boldsymbol{q}&=\boldsymbol{f}&&\quad\text{ in
      }I\times\Omega,
      \\
      \divo \boldsymbol{u}&=0&&\quad\text{ in }I\times\Omega,
      \\
      \partial_t \boldsymbol \omega -\divo \bN
      +\text{div}(\boldsymbol{\omega} \boldsymbol{u})
      &=\boldsymbol{\ell} - \levy:\bS &&\quad\text{ in }I\times\Omega,
      \\
      \boldsymbol{u}=\mathbf{0},\qquad
      \boldsymbol{\omega}&=\boldsymbol{0}&&\quad\text{ on }I\times
      \partial\Omega,
      \\
      \boldsymbol{u}(0)={u}_0,\qquad
      \boldsymbol{\omega}(0)&={\omega}_0&&\quad\text{ in }\Omega,
    \end{alignedat}
  \end{split}\label{eq:erf-2d}
\end{align}
where $\boldsymbol{\omega}\hspace*{-0.05em}\colon \hspace*{-0.05em}I\times \Omega \hspace*{-0.05em}\to\hspace*{-0.05em} \setR$ is the
scalar micro-rotation, $ \boldsymbol{u}\hspace*{-0.05em}\colon\hspace*{-0.05em} I\times \Omega\hspace*{-0.05em} \to\hspace*{-0.05em}
\setR^2$ the velocity and ${\boldsymbol{q}\hspace*{-0.05em}\colon\hspace*{-0.05em} I\hspace*{-0.075em}\times \hspace*{-0.075em}\Omega \hspace*{-0.05em}\to\hspace*{-0.05em}
\setR}$~the pressure. The system \eqref{eq:erf-2d} differs from its three-dimensional
counterpart mainly~in~equation~$\eqref{eq:erf-2d}_3$, which is now a
scalar equation that involves a scalar micro-rotation $\boldsymbol{\omega}$ (cf.~\cite{Lu01} for
the case $p=2$). The analogue to Theorem \ref{rem:7.4} in this setting
holds for $p>2$. Moreover, for the electric~field~$\boldsymbol E$, solving the
two-dimensional quasi-static Maxwell's equations \eqref{maxwell}, we
make the particular choice
\begin{align*}
  \boldsymbol{E}(t,{x}):=(t+x_2,t+x_1)^\top
\end{align*}
for all $(t,{x})^\top\hspace*{-0.1em}=\hspace*{-0.1em}(t,x_1,x_2)^\top\hspace*{-0.3em}\in\hspace*{-0.1em}
Q_T$.  We assume that the stress tensor
${\textrm{S}\hspace*{-0.1em}:\hspace*{-0.1em}\mathbb{M}_{\textup{sym}}^{2\times
    2}\hspace*{-0.1em} \times\hspace*{-0.1em}
  \mathbb{M}_{\textup{skew}}^{2\times
    2}\hspace*{-0.1em}\times\hspace*{-0.1em}
  \mathbb{R}^2\hspace*{-0.1em}\to \hspace*{-0.1em}\mathbb{M}^{2\times
    2}}$ and the couple stress tensor
$\textrm{N}:\mathbb{M}^{2\times 2} \times \mathbb{R}^2\to
\mathbb{M}^{2\times 2} $, for any
$D\in \mathbb{M}_{\textup{sym}}^{2\times 2}$,
${R\in \mathbb{M}_{\textup{skew}}^{2\times
    2}}$~and~${E,L\in \mathbb{R}^2}$, have the form
\begin{align*}  
  \textrm{S}(D,R,E)&:=(1+\vert {E}\vert^2)
                     (\kappa+\vert{D}\vert)^{p-2}{D}+
                     \vert{E}\vert^2(\kappa+\vert{R}\vert)^{p-2}
                     {R},
  \\
  {N}(L,E)&:=(1+\vert{E}\vert^2)
            (\kappa+\vert L\vert)^{p-2}L,
\end{align*}
with $\kappa=0.001$, where we used the notation $D = \bD u $,
$R=R( u, \omega):= W u +\levy\omega =W u+\big(\begin{smallmatrix}
  0&\omega\\
  -\omega&0
\end{smallmatrix}\big)$. We treat solutions with a
point singularity at the origin in the velocity and the
micro-rotation. More precisely, we assume that for every
$(t,{x})^\top=(t,x_1,x_2)^\top\in
Q_T$, there holds\footnote{The exact solutions
  do not satisfy the homogeneous boundary conditions
  \eqref{eq:erf-2d}$_4$. However, the error is mainly concentrated
  around the singularity in the origin and hence the small
  inconsistency with our theoretical~setup~does not have any influence
  in the results.\vspace*{-7mm}}
\begin{align}
  \boldsymbol{u}(t,{x}):=
  (t,t)^\top +\vert{x}\vert^{\alpha-1}(x_2,-x_1)^{\top},\quad
  \boldsymbol{\omega}(t,x):=t +\vert{x}\vert^{\alpha-1}x_1,
  \quad
  \boldsymbol{q}(t,{x}):=0,\label{exactsol} 
\end{align}
where $\alpha:=\frac{6}{5}-\frac{2}{p}\approx 0.291$. Then, making the choice
\begin{align}
  \begin{aligned}
    \boldsymbol{f}&:=\partial_t \boldsymbol{u}-\divo
    \bS+\divo(\boldsymbol{u}\otimes\boldsymbol{u})\in
    L^{p'}(I,(W_0^{1,p}(\Omega)^2)^*),\\
    \boldsymbol{l}&:=\partial_t \boldsymbol{\omega}-\divo
    \bN+\divo(\boldsymbol{\omega}\boldsymbol{u})+\levy:\bS\in
    L^{p'}(I,(W_0^{1,p}(\Omega))^*),
  \end{aligned}\label{rhs}
\end{align}
the functions \eqref{exactsol} solve \eqref{eq:erf-2d} with right-hand
sides \eqref{rhs}.  In particular, note that the parameter $\alpha$ is
chosen so small that just
$\boldsymbol{u}(t)\hspace*{-0.05em}\in\hspace*{-0.05em} W^{1,p}_{\textup{div}}(\Omega)^2$ for every
$t\hspace*{-0.05em}\in \hspace*{-0.05em}\overline{I}$, since
${\vert \bD\boldsymbol{u}(t,x)\vert\hspace*{-0.05em}\sim\hspace*{-0.05em} \vert x\vert^{\alpha-1}\hspace*{-0.05em}\in\hspace*{-0.05em}
L^p(\Omega)}$,~but simultaneously
$\boldsymbol{u}(t)\notin W^{2,1}(\Omega)^2$ for every
$t\in \overline{I}$. Similar, this choice guarantees
$\mathrm{S}(t)\in L^{p'}(\Omega)^{2\times 2}$ for every
$t\in \overline{I}$, but neither that
$\mathrm{S}(t)\notin W^{1,p'}(\Omega)^{2\times 2}$, nor that
$\divo\mathrm{S}(t)\in L^{p'}(\Omega)^2$ for
every~${t\in \overline{I}}$. Thus, the right-hand side has just enough
regularity, namely
$\boldsymbol{f}\in L^{p'}(I,(W_0^{1,p}(\Omega)^2)^*)$, to fall into
the framework of our weak convergence result Theorem
\ref{rem:7.4}. Exactly the same considerations~also apply to both
$\boldsymbol{\omega}(t)\in W^{1,p}(\Omega)$ and
$\boldsymbol{l}\in L^{p'}(I,(W_0^{1,p}(\Omega))^*)$. 

The chosen low regularity has the consequence that one still finds convergence of the scheme, i.e., at least weak convergence in the sense of Theorem \ref{5.17}, however, no stable convergence rates could be recorded experimentally (cf.~Tables~\ref{tab1}~to~\ref{tab4}).

The spatial discretization of our domain $\Omega$ is obtained by a
sequence of uniform finite element meshes
$(\mathcal{T}_{h_n})_{n\in \mathbb{N}}$ consisting of triangles with
straight sides and diameter
${h_n:=\frac{h_0}{2^n}}$,~${h_0:=2\sqrt{2}}$, for every
$n\in \mathbb{N}$. Beginning with $\mathcal{T}_{h_1}$, see, e.g.,
Figure \ref{refine} the first and the third picture, for every
$n\in \mathbb{N}$ with $n\ge 2$, the mesh $\mathcal{T}_{h_n}$ is a
refinement of $\mathcal{T}_{h_{n-1}}$ obtained by subdividing~each~triangle~into~four, which is based an edge midpoint or regular $1:4$
refinement algorithm.
		
We consider the MINI element (cf.~Table~\ref{tab1} and
Table~\ref{tab2})~and~the~spatially~\mbox{conforming}~Crouzeix--Raviart element
(cf.~Table~\ref{tab3} and Table~\ref{tab4}). Furthermore, we use the
time step-sizes
${\tau_n\hspace*{-0.1em}:=\hspace*{-0.1em}0.02\cdot2^{-n}}$, i.e., ${K_n\hspace*{-0.1em}:=\hspace*{-0.1em}10\cdot2^n}$, $n\in \mathbb{N}$. Then, the
iterates
$(({u}_n^k,\omega_n^k)^\top\hspace*{-0.1em})_{k=0,\dots,K_n}\hspace*{-0.25em}\subseteq\hspace*{-0.2em}
V_n$ solving the straightforward two-dimensional
analog~of~\eqref{eq:p-NSnon} are approximated employing Newton's
iteration.~Apart~from~that, let the mapping
${F}:\mathbb{M}^{2\times 2}_{\textup{sym}}\to
\mathbb{M}^{2\times 2}_{\textup{sym}}$ be defined by
${F}({A}):=(\kappa+\vert
{A}\vert)^{\frac{p-2}{2}}{A}$ for every
${A}\in \mathbb{M}^{2\times 2}_{\textup{sym}}$. Equally, let
the mapping ${G}:\mathbb{R}^2\to \mathbb{R}^2$ be defined~by
${G}({a}):=(\kappa+\vert
{a}\vert)^{\frac{p-2}{2}}{a}$ for every
${a}\in \mathbb{R}^2$.~Then, for $n=1,\dots,6$, we are
interested in the~parabolic~errors 
\begin{alignat*}{2}
  e_{{F},\boldsymbol{u}}^n&:=\bigg(\sum_{k=0}^{K_n}{\tau_n\|{F}(\bD\boldsymbol{u}(t_k))-{F}(\bD{u}_n^k)\|^2_{L^2(\Omega)^{2\times 2}}}\bigg)^{\frac{1}{2}}, \quad&&e_{L^2,\boldsymbol{u}}^n:=\max_{0\leq k \leq K_n}{\|\boldsymbol{u}(t_k)-{u}_n^k\|_{L^2(\Omega)^2}},\\
  e_{{G},\boldsymbol{\omega}}^n&:=\bigg(\sum_{k=0}^{K_n}{\tau_n\|{G}(\nabla\boldsymbol{\omega}(t_k))-{G}(\nabla\omega_n^k)\|^2_{L^2(\Omega)^2}}\bigg)^{\frac{1}{2}},
  &&e_{L^2,\boldsymbol{\omega}}^n:=\max_{0\leq k \leq
    K_n}{\|\boldsymbol{\omega}(t_k)-\omega_n^k\|_{L^2(\Omega)}}, 
\end{alignat*}
which can be considered as approximations of
$\|{F}(\bD\boldsymbol{u})-{F}(\bD\overline{\boldsymbol{u}}_n^\tau)\|_{L^2(Q_T)^{2\times
    2}}$,
$\|\boldsymbol{u}-\overline{\boldsymbol{u}}_n^\tau\|_{L^\infty(I,L^2(\Omega)^2)}$,
$\|{G}(\nabla\boldsymbol{\omega})-{G}(\nabla\overline{\boldsymbol{\omega}}_n^\tau)\|_{L^2(Q_T)^2}$
and
$\|\boldsymbol{u}-\overline{\boldsymbol{u}}_n^\tau\|_{L^\infty(I,L^2(\Omega))}$. In
particular, we are interested in the total parabolic errors
$e_{\textrm{tot},\boldsymbol{u}}^n:=e_{{F},\boldsymbol{u}}^n+e_{L^2,\boldsymbol{u}}^n$
and
$e_{\textrm{tot},\boldsymbol{\omega}}^n:=e_{{G},\boldsymbol{\omega}}^n+e_{L^2,\boldsymbol{\omega}}^n$
for $n=1,\dots,6$.  As an estimation of the convergence rates, we
employ the experimental order of convergence (\texttt{EOC}): 
\begin{align*}
  \texttt{EOC}(e^n):=\frac{\log\big(\frac{e^n}{e^{n-1}}\big)}{\log\big(\frac{h_n+\tau_n}{h_{n-1}+\tau_{n-1}}\big)},\quad
  n=2,\dots,6, 
\end{align*}
where $e^n$, $n=1,\dots,6$, either denote
$e_{{F},\boldsymbol{u}}^n$, $e_{L^2,\boldsymbol{u}}^n$,
$e_{\textrm{tot},\boldsymbol{u}}^n$,
$e_{{G},\boldsymbol{\omega}}^n$,
$e_{L^2,\boldsymbol{\omega}}^n$, or
$e_{\textrm{tot},\boldsymbol{\omega}}^n$, $n=1,\dots,6$, resp.

In order to obtain a higher accuracy in the computation of these errors,
in particular, with regard to the singularities of the exact solutions
around the origin, we interpolate~both~$\boldsymbol{u}(t_k)$~and~${u}_n^k$, or $\boldsymbol{\omega}(t_k)$ and $\omega_n^k$, into
 polynomial spaces of higher order with respect to a suitably refined~mesh, namely into $\mathcal{P}_5(\mathcal{T}_{h_n}')^2$, or
$\mathcal{P}_5(\mathcal{T}_{h_n}')$, resp., where $\mathcal{T}_{h_n}'$
is a refinement of $\mathcal{T}_{h_n}$, which is obtained by applying
the longest edge bisection method of \textsf{FEniCS} to
$\mathcal{T}_{h_{n+1}}$ for all cells $T\in \mathcal{T}_{h_{n+1}}$
that satisfy $\textup{dist}(T,{0}) < 0.25$ and subsequently on
the resulting refined mesh $\widetilde{\mathcal{T}}_{h_{n+1}}$ for all
cells $T\in \widetilde{\mathcal{T}}_{h_{n+1}}$~that satisfy
$\textup{dist}(T,{0})< 0.1$, see e.g. Figure \ref{refine} the
second and the fourth picture. 
\begin{center}
  \begin{figure}[ht]
    \centering \hspace*{-0.65cm}
    \includegraphics[width=15.3cm]{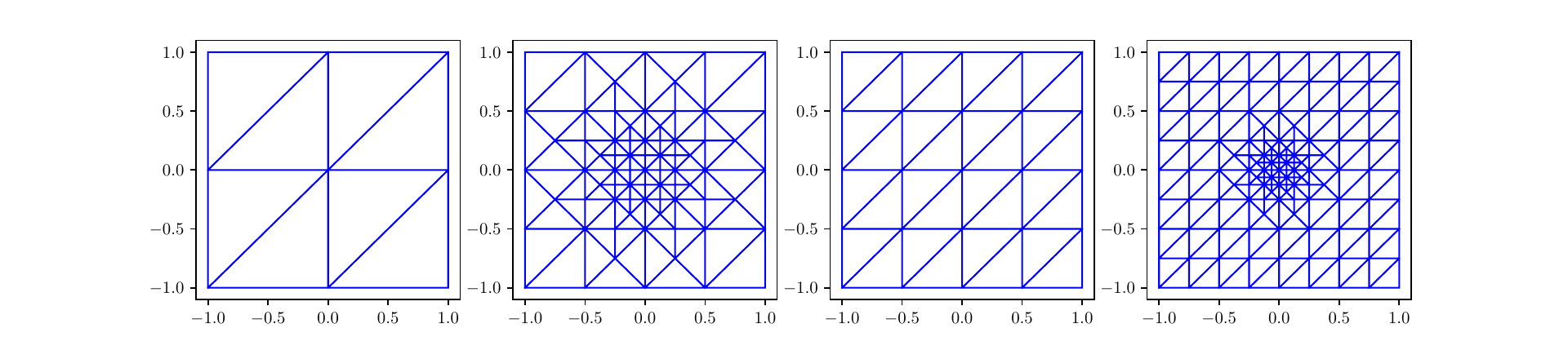}\\
    \caption{From the left to the right: snapshots of the meshes $\mathcal{T}_{h_1}$, $\mathcal{T}_{h_1}'$, $\mathcal{T}_{h_2}$ and $\mathcal{T}_{h_2}'$.}
    \label{refine}
  \end{figure}
\end{center} 

		In this manner, we obtain the following results:\vspace*{-0.1cm}
		
		\begin{table}[ht]
			\centering
			\begin{tabular}[t]{c|c|c|c|c|c|c|c} \toprule
				$n$&$h_n=\frac{h_0}{2^n}$&$\tau_n=\frac{0.2T}{2^n}$& $e_{L^2,\boldsymbol{u}}^n$&
				$\texttt{EOC}(e_{L^2,\boldsymbol{u}}^n)$
				& $e_{F,\boldsymbol{u}}^n$&
				$\texttt{EOC}(e_{F,\boldsymbol{u}}^n)$
				& $\texttt{EOC}(e_{\textrm{tot},\boldsymbol{u}}^n)$\\   
				\midrule
				$1$&  $1.414$ & $1.00\mathrm{e}{\sm2}$ & $0.463$ & -- & $0.548$ & --& --\\
				$2$&   $7.07\mathrm{e}{\sm1}$ & $5.00\mathrm{e}{\sm3}$ & $0.263$ & $0.81$ & $0.463$& $0.24$& $0.48$\\
				$3$&  $3.54\mathrm{e}{\sm1}$ &$2.50\mathrm{e}{\sm3}$ & $0.139$ & $0.92$ & $0.407$& $0.18$& $0.41$\\
				$4$&  $1.77\mathrm{e}{\sm1} $&$1.25\mathrm{e}{\sm3}$ & $0.075$ & $0.89$ & $0.362$& $0.17$ & $0.32$\\
				$5$&  $8.84\mathrm{e}{\sm2} $  &$6.25\mathrm{e}{\sm4}$& $0.042$ & $0.86$ & $ 0.323$& $0.17$& $0.26$\\
				$6$&  $4.42\mathrm{e}{\sm2} $  &$3.13\mathrm{e}{\sm4}$& $0.028$ & $0.57$ & $0.289$ & $0.16$ & $0.20$\\
				\bottomrule \hline
			\end{tabular}
			\caption{Error analysis with respect to $\boldsymbol{u}$ for the MINI element.}
			\label{tab1}
		\end{table}

		\begin{table}[ht]
			\centering
			\begin{tabular}[t]{c|c|c|c|c|c|c|c} \toprule
				$n$&$h_n=\frac{h_0}{2^n}$&$\tau_n=\frac{0.2T}{2^n}$& $e_{L^2,\boldsymbol{\omega}}^n$&
				$\texttt{EOC}(e_{L^2,\boldsymbol{\omega}}^n)$
				& $e_{G,\boldsymbol{\omega}}^n$&
				$\texttt{EOC}(e_{G,\boldsymbol{\omega}}^n)$
				& $\texttt{EOC}(e_{\textrm{tot},\boldsymbol{\omega}}^n)$\\   
				\midrule
				$1$&  $1.414$ & $1.00\mathrm{e}{\sm2}$ & $0.359$ & -- & $0.592$ & --& --\\
				$2$&   $7.07\mathrm{e}{\sm1}$ & $5.00\mathrm{e}{\sm3}$ & $0.203$ & $0.83$ & $0.527$& $0.17$& $0.38$\\
				$3$&  $3.54\mathrm{e}{\sm1}$ &$2.50\mathrm{e}{\sm3}$ & $0.105$ & $0.95$ & $0.465$& $0.18$& $0.36$\\
				$4$&  $1.77\mathrm{e}{\sm1} $&$1.25\mathrm{e}{\sm3}$ & $0.058$ & $0.87$ & $0.410$& $0.18$ & 
				$0.29$\\
				$5$&  $8.84\mathrm{e}{\sm2} $  &$6.25\mathrm{e}{\sm4}$& $0.039$ & $0.55$ & $ 0.363$& $0.18$& $0.22$\\
				$6$&  $4.42\mathrm{e}{\sm2} $  &$3.13\mathrm{e}{\sm4}$& $0.036$ & $0.13$ & $0.324$ & $0.16$ & $0.16$\\
				\bottomrule \hline
			\end{tabular}
			\caption{Error analysis with respect to $\boldsymbol{\omega}$ for the MINI element.}
			\label{tab2}
		\end{table}
		\begin{table}[ht]
			\centering
			\begin{tabular}[t]{c|c|c|c|c|c|c|c} \toprule
				$n$&$h_n=\frac{h_0}{2^n}$&$\tau_n=\frac{0.2T}{2^n}$& $e_{L^2,\boldsymbol{u}}^n$&
				$\texttt{EOC}(e_{L^2,\boldsymbol{u}}^n)$
				& $e_{{F},\boldsymbol{u}}^n$&
				$\texttt{EOC}(e_{{F},\boldsymbol{u}}^n)$
				& $\texttt{EOC}(e_{\textrm{tot},\boldsymbol{u}}^n)$\\   
				\midrule
				$1$&  $1.414$ & $1.00\mathrm{e}{\sm2}$ & $0.225$ & -- & $0.419$ & --& --\\
				$2$&   $7.07\mathrm{e}{\sm1}$ & $5.00\mathrm{e}{\sm3}$ & $0.134$ & $0.75$ & $0.367$& $0.19$& $0.36$\\
				$3$&  $3.54\mathrm{e}{\sm1}$ &$2.50\mathrm{e}{\sm3}$ & $0.086$ & $0.64$ & $0.329$& $0.16$& $0.27$\\
				$4$&  $1.77\mathrm{e}{\sm1} $&$1.25\mathrm{e}{\sm3}$ & $0.054$ & $0.67$ & $0.299$& $0.14$ & $0.23$\\
				$5$&  $8.84\mathrm{e}{\sm2} $  &$6.25\mathrm{e}{\sm4}$& $0.034$ & $0.66$ & $ 0.272$& $0.14$& $0.20$\\
				$6$&  $4.42\mathrm{e}{\sm2} $  &$3.13\mathrm{e}{\sm4}$& $0.026$ & $0.38$ & $0.249$ & $0.12$ & $0.15$\\
				\bottomrule \hline
			\end{tabular}
			\caption{Error analysis with respect to $\boldsymbol{u}$ for the spatially conforming Crouzeix--Raviart element.}
			\label{tab3}
		\end{table}
		\begin{table}[ht]
			\centering
			\begin{tabular}[t]{c|c|c|c|c|c|c|c} \toprule
				$n$&$h_n=\frac{h_0}{2^n}$&$\tau_n=\frac{0.2T}{2^n}$& $e_{L^2,\boldsymbol{\omega}}^n$&
				$\texttt{EOC}(e_{L^2,\boldsymbol{\omega}}^n)$
				& $e_{G,\boldsymbol{\omega}}^n$&
				$\texttt{EOC}(e_{G,\boldsymbol{\omega}}^n)$
				& $\texttt{EOC}(e_{\textrm{tot},\boldsymbol{\omega}}^n)$\\   
				\midrule
				$1$&  $1.414$ & $1.00\mathrm{e}{\sm2}$ & $0.359$ & -- & $0.591$ & --& --\\
				$2$&   $7.07\mathrm{e}{\sm1}$ & $5.00\mathrm{e}{\sm3}$ & $0.203$ & $0.82$ & $0.527$& $0.17$& $0.38$\\
				$3$&  $3.54\mathrm{e}{\sm1}$ &$2.50\mathrm{e}{\sm3}$ & $0.106$ & $0.94$ & $0.465$& $0.18$& $0.35$\\
				$4$&  $1.77\mathrm{e}{\sm1} $&$1.25\mathrm{e}{\sm3}$ & $0.058$ & $0.87$ & $0.410$& $0.18$ & 
				$0.29$\\
				$5$&  $8.84\mathrm{e}{\sm2} $  &$6.25\mathrm{e}{\sm4}$& $0.040$ & $0.55$ & $ 0.363$& $0.18$& $0.22$\\
				$6$&  $4.42\mathrm{e}{\sm2} $  &$3.13\mathrm{e}{\sm4}$& $0.036$ & $0.13$ & $0.324$ & $0.16$ & $0.16$\\
				\bottomrule \hline
			\end{tabular}
			\caption{Error analysis with respect to $\boldsymbol{\omega}$ for the spatially conforming Crouzeix--Raviart element.}
			\label{tab4}
		\end{table}

\section*{Acknowlegdements}

Luigi C. Berselli was partially supported by a grant of the group
GNAMPA of INdAM and by the University of Pisa within
the grant PRA$\_{}2018\_{}52$ UNIPI: ``\textit{Energy and regularity:
  New techniques for classical PDE problems.}'' We would like to thank
the referees for their valuable comments. 

\providecommand{\MR}[1]{}\def\cprime{$'$} \def\cprime{$'$} \def\cprime{$'$}
\providecommand{\noopsort}[1]{}
\providecommand{\arxiv}[1]{\href{http://www.arxiv.org/abs/#1}{arXiv~#1}}
\providecommand{\doi}[1]{\url{https://doi.org/#1}}
\providecommand{\href}[2]{#2}


\end{document}